\documentclass{amsart}
\usepackage{amscd}
\usepackage[mathscr]{eucal}
\usepackage{amssymb}
\usepackage{multirow}
\usepackage{bm}
\usepackage[all]{xy}
\newcommand{\bZ}{{\mathbb Z}}
\newcommand{\s}{\mathop{\mathrm{S}}\nolimits}
\newtheorem{thm}{Theorem}[section]
\newtheorem{prop}[thm]{Proposition}
\newtheorem{lemma}[thm]{Lemma}
\newtheorem{cor}[thm]{Corollary}
\newtheorem{thm'}{Theorem}[subsection]
\newtheorem{prop'}[thm']{Proposition}
\newtheorem{lemma'}[thm']{Lemma}
\newtheorem{cor'}[thm']{Corollary}

\theoremstyle{definition}
\newtheorem{defi}[thm]{Definition}
\newtheorem{defi'}[thm']{Definition}

\theoremstyle{remark}
\newtheorem{prob}[thm]{Problem}
\newtheorem{rem}[thm]{Remark}
\newtheorem{rem'}[thm']{Remark}
\newtheorem{conv}[thm]{Convention}

\newtheorem{exam'}[thm']{Example}
\newtheorem{prob'}[thm']{Problem}

\begin{document}
\allowdisplaybreaks
\numberwithin{equation}{section}

\title{Unstable higher Toda brackets}
\author{H. \=Oshima}
\address[Hideaki \=Oshima]{Professor Emeritus\\
Ibaraki University\\
Mito, Ibaraki 310-8512, Japan}
\email{hideaki.ooshima.mito@vc.ibaraki.ac.jp}

\author{K. \=Oshima}
\address[Katsumi \=Oshima]{1-4001-5 Ishikawa, Mito, Ibaraki 310-0905, Japan}
\email{k-oshima@mbr.nifty.com}
\subjclass[2010]{Primary 55Q05; Secondary 55Q35, 55P05}
\keywords{Toda bracket, Unstable higher Toda bracket, Higher composition, Cofibration, Coextension, Extension}
\dedicatory{Dedicated to memories of the Arakis}

\begin{abstract}
We define new unstable $n$-fold Toda brackets 
$\{\vec{\bm f}\,\}^{(aq\ddot{s}_2)},\ \{\vec{\bm f}\,\}^{(\ddot{s}_t)}$ 
for every composable sequence $\vec{\bm f}=(f_n,\dots,f_1)$ of 
pointed maps between well-pointed spaces 
$X_{n+1}\overset{f_n}{\longleftarrow} \cdots\overset{f_2}{\longleftarrow}X_2\overset{f_1}{\longleftarrow}X_1$ with $n\ge 3$. 
The brackets agree with the classical Toda bracket when $n=3$, 
and they are subsets of both the unstable $n$-fold Toda brackets of Gershenson and Cohen for every $n\ge 3$. 
\end{abstract}
\maketitle

\section{Introduction}

The Toda bracket \cite{T1,T3,T,Sp} is one of the basic tools in homotopy theory 
and often called a secondary composition or a $3$-fold bracket. 
After \cite{T3} a number of definitions of a higher Toda bracket, that is, 
an $n$-fold bracket for $n\ge 3$, have appeared in the literature. 
Stable higher Toda brackets are comparatively investigated in \cite{C,W2} (cf.\,\cite{G,K,L,Po}).  
In this paper we study mainly unstable higher Toda brackets. 
A sequence $(p_3,p_4,p_5,\dots)$, where $p_n$ is an unstable $n$-fold bracket, is called 
{\it a system of unstable higher Toda brackets} if it is defined systematically, and it 
is called {\it normal} if $p_3$ agrees with the classical Toda bracket up to sign. 
Systems of Spanier \cite{S}, Walker \cite{W, W2} (cf.\,Mori \cite{M}), Blanc \cite{B1}, 
Blanc--Markl \cite{B2}, and Marcum--Oda \cite{MO} (cf.\,\cite{HMO}) are normal; 
systems of Gershenson \cite{G} and Cohen \cite{C} are not normal. 
It seems difficult to nominate one of known systems 
as the standard system, 
because 
we have little information about their applications and relations between them. 
We provide two new candidates for the standard system by modifying the Gershenson's system  which originated with \cite{T3}, 
and study relations between new systems, the systems of Gershenson and Cohen, 
and the $4$-fold bracket of \^{O}guchi \cite{Og,OO1}. 
Two new systems are normal. 
Our method is classical and not so abstract as \cite{B1,B2}. 

Given a composable sequence $\vec{\bm f}=(f_n,\dots,f_1)$ of pointed maps 
between well-pointed spaces $f_i:X_i\to X_{i+1}$ with $n\ge 3$, 
we will define $\{\vec{\bm f}\,\}^{(\star)}$ which is a subset of 
the group $[\Sigma^{n-2}X_1,X_{n+1}]$, 
where $\star$ is one of twelve symbols defined in Definition~6.1.1(4). 
($[\Sigma^kX,Y]$ is the set of homotopy classes of pointed maps from the $k$-fold pointed suspension of $X$ to $Y$.) 
Hence we have twelve systems of unstable higher Toda brackets. 
Four of them, $\{\vec{\bm f}\,\}^{(\star)}$ for $\star=aq\ddot{s}_2, \ddot{s}_t, qs_2, q$, 
are essential; 
$\{\vec{\bm f}\,\}^{(aq\ddot{s}_2)}$ and $\{\vec{\bm f}\,\}^{(\ddot{s}_t)}$ 
are candidates for the standard $n$-fold bracket; $\{\vec{\bm f}\,\}^{(q)}$ is the largest of 
the twelve subsets and a revision of the $n$-fold 
$C$-composition product of Gershenson \cite[Definition~2.2D]{G}; 
they are the empty set for suitable $\vec{\bm f}$. 
For a pointed space $X$, we denote the set of homotopy classes of pointed 
homotopy equivalences $X\to X$ by $\mathscr{E}(X)$ which is a subset of $[X,X]$ and a group under the composition operation. 
The group $\mathscr{E}(\Sigma^{n-2}X_1)$ acts on $[\Sigma^{n-2}X_1,X_{n+1}]$ from the right by the composition:
$$
[\Sigma^{n-2}X_1,X_{n+1}]\times\mathscr{E}(\Sigma^{n-2}X_1)\to [\Sigma^{n-2}X_1,X_{n+1}],\ 
(\alpha,\varepsilon)\mapsto \alpha\circ\varepsilon.
$$
Our main results are (1.1) -- (1.11) below. 
\begin{enumerate}
\item[(1.1)] $\{\vec{\bm f}\,\}^{(aq\ddot{s}_2)}\cup\{\vec{\bm f}\,\}^{(\ddot{s}_t)}
\subset\{\vec{\bm f}\,\}^{(qs_2)}\subset\{\vec{\bm f}\,\}^{(q)}$; \ 
$\{\vec{\bm f}\,\}^{(q)}\circ\varepsilon=\{\vec{\bm f}\,\}^{(q)}$ for every 
$\varepsilon\in\mathscr{E}(\Sigma^{n-2}X_1)$; 
$\{\vec{\bm f}\,\}^{(qs_2)}
=\{\vec{\bm f}\,\}^{(aq\ddot{s}_2)}\circ\mathscr{E}(\Sigma^{n-2}X_1)=
\{\vec{\bm f}\,\}^{(\ddot{s}_t)}\circ\mathscr{E}(\Sigma^{n-2}X_1)$.
\item[(1.2)] If $\alpha\in\{\vec{\bm f}\,\}^{(q)}$, then there are 
$\theta,\theta'\in[\Sigma^{n-2}X_1, \Sigma^{n-2}X_1]$ such that 
$\alpha\circ\theta\in\{\vec{\bm f}\,\}^{(aq\ddot{s}_2)}$ 
and $\alpha\circ\theta'\in\{\vec{\bm f}\,\}^{(\ddot{s}_t)}$.
\item[(1.3)] If $\{\vec{\bm f}\,\}^{(\star)}$ is not empty for some $\star$, then $\{\vec{\bm f}\,\}^{(\star)}$ is not empty 
for all $\star$. 
\item[(1.4)] If $\{\vec{\bm f}\,\}^{(\star)}$ contains $0$ for some $\star$, then $\{\vec{\bm f}\,\}^{(\star)}$ contains $0$ 
for all $\star$. 
\item[(1.5)] $\Sigma\{\vec{\bm f}\,\}^{(\star)}\subset(-1)^n\{\Sigma\vec{\bm f}\,\}^{(\star)}$ for all $\star$, 
where $\Sigma\vec{\bm f}=(\Sigma f_n,\dots,\Sigma f_1)$.  
\item[(1.6)] $\{\vec{\bm f}\,\}^{(\star)}$ depends only on the homotopy classes of $f_i$ $(1\le i\le n)$ for all $\star$.
\item[(1.7)] $\{\vec{\bm f}\,\}^{(aq\ddot{s}_2)}\cup\{\vec{\bm f}\,\}^{(\ddot{s}_t)}\subset\langle \vec{\bm f}\,\rangle$, 
where $\langle \vec{\bm f}\,\rangle$ is the $n$-fold bracket of Cohen \cite{C}. 
\item[(1.8)] When $n=3$, we have 
$\{\vec{\bm f}\,\}^{(aq\ddot{s}_2)}=\{\vec{\bm f}\,\}^{(\ddot{s}_t)}
=\{\vec{\bm f}\,\}$, 
where $\{\vec{\bm f}\,\}=\{f_3,f_2,f_1\}$ is the classical unstable Toda bracket 
which does not necessarily coincide with either $\{\vec{\bm f}\,\}^{(qs_2)}$ or $\langle\vec{\bm f}\,\rangle$. 
\item[(1.9)] When $n=4$, we have 
$
\{\vec{\bm f}\,\}^{(\ddot{s}_t)}=\bigcup\{f_4,[f_3,A_2,f_2],(f_2,A_1,f_1)\}\supset\{\vec{\bm f}\,\}^{(1)},
$
where the union $\bigcup$ is taken over all triples $(A_3,A_2,A_1)$ of 
null-homotopies $A_i:f_{i+1}\circ f_i\simeq *\ (i=1,2,3)$ such that 
$[f_{i+1},A_i,f_i]\circ (f_i,A_{i-1},f_{i-1})$ $\simeq *\ (i=2,3)$, 
and $\{\vec{\bm f}\,\}^{(1)}$ is the $4$-fold bracket of \^{O}guchi \cite[(6.1)]{OO1}. 
(See Section 2 for definitions of $[f_{i+1},A_i,f_i]$ and $(f_i,A_{i-1},f_{i-1})$.)
\item[(1.10)] For two pointed maps 
$Z\overset{f}{\longleftarrow}Y\overset{g}{\longleftarrow}X$, we denote by 
$\{f,g\}^{(\star)}$ the one point set consisting of the homotopy class of $f\circ g$. 
Then
\begin{enumerate}
\item[\rm(1)] If $\{f_{n-1},\dots,f_1\}^{(q)}\ni 0$ and $\{f_n,f_{n-1},\dots,f_k\}^{(aq\ddot{s}_2)}=\{0\}$ for 
all $k$ with $2\le k<n$, then 
$\{f_n,\dots,f_1\}^{(\star)}$ is not empty for all $\star$. 
\item[\rm(2)] If $\{f_n,\dots,f_2\}^{(q)}\ni 0$ and $\{f_k,\dots,f_2,f_1\}^{(aq\ddot{s}_2)}=\{0\}$ for all $k$ with $2\le k< n$, 
then $\{f_n,\dots,f_1\}^{(\star)}$ is not empty for all $\star$.
\end{enumerate}
\item[(1.11)] If a pointed map $j:A\to X$ is a cofibration in the category of non-pointed spaces, 
then for any pointed map $f:X\to Y$ the pointed map $1_Y\cup Cj:Y\cup_{f\circ j}CA\to Y\cup_f CX$ 
between pointed mapping cones is a cofibration in the category of non-pointed spaces.
\end{enumerate}

It is not clear whether the $n$-fold bracket $\{\vec{\bm f}\,\}^{(\star)}$ agrees 
with one of the $n$-fold brackets in \cite{B1,B2,C,G,MO,S,W,W2} when $n\ge 4$. 
An advantage of our definition is that it can be generalized easily to the stable version (see \S6.9) 
and the subscripted version 
$\{\vec{\bm f}\,\}^{(\star)}_{\vec{\bm m}}\subset[\Sigma^{|\vec{\bm m}|+n-2}X_1,X_{n+1}]$ (cf.\,\cite[p.9]{T} when $n=3$), 
where $\vec{\bm m}=(m_n,\dots,m_1)$ is a sequence of non-negative integers, $|\vec{\bm m}|=m_n+\dots+m_1$, 
and $f_i:\Sigma^{m_i}X_i\to X_{i+1}$ $(1\le i\le n)$. 
We omit details of the subscripted version 
because they are complicated but similar to the non subscripted version. 

The referee pointed out that B. Gray defined unstable higher Toda brackets 
in his unpublished note. 
However we have not confirmed his definition because we could not get his note. 

In Section 2, we recall usual notions of homotopy theory and state two propositions 2.1 
and 2.2, where 2.1 is well-known and 2.2 is (1.11) above and a key to define $\{\vec{\bm f}\,\}^{(\star)}$. 
In Section~3, we study maps between mapping cones, that is, we prove a lemma 
which shall be used in Section 5, and recall results of Puppe \cite{P}. 
In Section 4, we introduce the notion of homotopy cofibre. 
In Section 5, we revise the notion of shaft of Gershenson \cite{G}. 
Section~6 consists of nine subsections \S6.1--\S6.9. 
In \S6.1 we define $\{\vec{\bm f}\,\}^{(\star)}$. 
In \S6.2 we prove (1.k)  for k=1,2,3,4 and state an example. 
In \S6.k we prove (1.k+2) for k=3,4,5,6,7. 
In \S6.8 we prove a proposition which is the same as (1.10). 
In \S6.9 we define stable higher Toda brackets. 
In Appendix~A, we prove Proposition~2.2.
In Appendix~B, we recall the definition of $\langle\vec{\bm f}\,\rangle$ and prove 
$\Sigma\langle\vec{\bm f}\,\rangle\subset(-1)^n\langle\Sigma\vec{\bm f}\,\rangle$.

\section{Preliminaries}

Let $\mathrm{TOP}$ denote the category of  topological spaces 
({\it spaces} for short) and continuous maps ({\it maps} for short). 
Let $I$ denote the unit interval $[0,1]$, 
$I^n=I\times\cdots\times I$ $(n\text{-times})$, and $\partial I^n$ the boundary of $I^n$. 
For a space $X$, we denote by $1_X:X\to X$ the identity map of $X$ and by $i_t^X:X\to X\times I$ for 
$t\in I$ the map $i_t^X(x)=(x,t)$. 
For a map $f:X\to Y$, we denote by $1_f:X\times I\to Y$ the map $1_f(x,t)=f(x)$, and 
we call $f$ {\it closed} if $f(A)$ is closed for every closed subset $A$ of $X$. 
Given maps $f,g:X\to Y$, if there is a map $H : X\times I\to Y$ such that $H_0=f$ and $H_1=g$, 
then we write $f\simeq g$ or $H:f\simeq g$, where $H_t=H\circ i^X_t:X\to Y$ i.e.  $H_t(x)=H(x,t)$. 
In the last case, the map $H$ is often denoted by $H_t$ and called a {\it homotopy} from $f$ to $g$. 
The homotopy relation $\simeq$ is an equivalence relation on the set of maps $X\to Y$ and  
the equivalence class of $f$ is called the {\it homotopy class} of $f$. 
Given a homotopy $H : X\times I\to Y$, the {\it inverse} homotopy $-H:X\times I\to Y$ is defined by $(-H)_t=H_{1-t}$; 
$H$ is a {\it null homotopy} if $H_1$ is a constant map to a point of $Y$. 
A map $f:X\to Y$ is a {\it homotopy equivalence} (denoted by $f:X\simeq Y$) if there is a map $g:Y\to X$ such that $g\circ f\simeq 1_X$ 
and $f\circ g\simeq 1_Y$, where $g$ is called a {\it homotopy inverse of $f$} and denoted often by $f^{-1}$. 
We write $X\simeq Y$ if there is a homotopy equivalence $X\to Y$. 
A map $j:A\to X$ is a {\it cofibration} if, for any space $Y$ and any maps $f:X\to Y$ and $G:A\times I\to Y$ 
such that $f\circ j=G\circ i_0^A$, there is a map $H:X\times I\to Y$ such that $H\circ (j\times 1_I)=G$ and $H\circ i_0^X=f$. 
$$
\xymatrix{
&  A\times I \ar[dr]_-{j\times 1_I} \ar@/^/[drr]^-G &&\\
A \ar[ur]^-{i_0^A} \ar[dr]_-j & & X\times I \ar@{.>}[r]^H & Y\\
& X \ar[ur]^-{i_0^X} \ar@/_/[urr]_-f
}
$$
By \cite[Theorem 1]{S1}, every cofibration $j:A\to X$ is an embedding, 
that is, $j$ gives a homeomorphism from $A$ to the subspace $j(A)$ of $X$ i.e.\,$j:A\approx j(A)$. 

Let $\mathrm{TOP}^*$ denote the category of spaces with base points ({\it pointed spaces} for short) and 
maps preserving base points ({\it pointed maps} for short). 
We often call a space, a map, and a cofibration in $\mathrm{TOP}$ a {\it free} space, a {\it free} map, 
and a {\it free} cofibration, respectively. 
For any pointed space $X$, we denote the base point of $X$ by $x_0$ or $*$. 
A pointed space $X$ is a {\it well-pointed space} ({\it w-space} for short) (resp.\,{\it clw-space}) if the inclusion $\{x_0\}\to X$ is 
a free (resp.\,closed free) cofibration. 
Let $\mathrm{TOP}^w$ (resp.\,$\mathrm{TOP}^{clw}$) denote the category of $w$-spaces (resp.\,$clw$-spaces) and pointed maps. 
Thus we have a sequence of categories: 
$
\mathrm{TOP}^{clw}\rightarrowtail\mathrm{TOP}^w\rightarrowtail\mathrm{TOP}^*\twoheadrightarrow \mathrm{TOP},
$ 
where $\twoheadrightarrow$ is the functor forgetting the base points, and $\mathscr{C}\rightarrowtail\mathscr{D}$ means that 
the category $\mathscr{C}$ is a full subcategory of the category $\mathscr{D}$ and $\mathscr{D}$ contains 
at least one object which is not in $\mathscr{C}$ (cf.\,Beispiele 1 and 2 \cite[pp.32-33]{DKP}). 
Homotopy, homotopy equivalence, cofibration, and some of other notions in $\mathrm{TOP}$ 
can be defined in other three categories of the above sequence 
exactly as in $\mathrm{TOP}$, except that all maps and homotopies are required to respect the base points. 
As remarked in \cite[p.438]{S3}, the proof of \cite[Theorem 1]{S1} can be modified to prove that all cofibrations 
in $\mathrm{TOP}^*$ are embeddings. 
When we set $\mathscr{C}_4=\mathrm{TOP}^{clw}$, $\mathscr{C}_3=\mathrm{TOP}^w$, 
$\mathscr{C}_2=\mathrm{TOP}^*$, $\mathscr{C}_1=\mathrm{TOP}$,  
if, for some $k>\ell$, a map $j:A\to X$ in $\mathscr{C}_k$ is a cofibration in $\mathscr{C}_\ell$, then 
$j$ is a cofibration in $\mathscr{C}_k$. 
For pointed spaces $X$ and $Y$, $[X,Y]$ denotes the set of pointed homotopy classes of pointed maps $X\to Y$, 
and the trivial map $X\to Y,\ x\mapsto y_0,$ is denoted by $*$ and its homotopy class is denoted by $0$; 
$[X,Y]$ is regarded as a pointed set with the base point $0$. 
For homotopies $H:Y\times I\to Z$ and $G, F: X\times I\to Y$ with $F_1=G_0$, homotopies 
$H\,\bar{\circ}\,G: X\times I\to Z$ and $G\bullet F:X\times I\to Y$ are defined by 
$$
H\,\bar{\circ}\,G(x,t)=H(G(x,t),t),\quad 
G\bullet F(x,t)=\begin{cases} F(x,2t) & 0\le t\le \frac{1}{2}\\ G(x,2t-1) & \frac{1}{2}\le t\le 1\end{cases}.
$$

Assign to the $n$-sphere $\s^n=\{(t_1,\dots,t_{n+1})\in\mathbb{R}^{n+1}\,|\,\sum t_i^2=1\}\ (n=0,1,2,\dots)$ 
and $I=[0,1]$ the base points $(1,0,\dots,0)$ and $1$, respectively. 
Then, as is well-known, $\s^n$ and $I$ are $clw$-spaces. 

For pointed spaces $X_1,\dots,X_n$, we denote by $X_1\wedge\cdots\wedge X_n$ the quotient space 
$$
(X_1\times\dots\times X_n)/(\bigcup_{i=1}^n X_1\times\dots\times X_{i-1}\times\{*_i\}\times X_{i+1}\times\dots\times X_n),
$$
where $*_i$ is the base point of $X_i$. 
In $X_1\wedge\cdots\wedge X_n$, the point represented by $(x_1,\dots,x_n)$ is denoted by $x_1\wedge\dots\wedge x_n$, 
and $*_1\wedge\dots\wedge *_n$ is the base point. 
For pointed maps $f_i:X_i\to Y_i$, we set $f_1\wedge\dots\wedge f_n:X_1\wedge\dots\wedge X_n\to 
Y_1\wedge\dots\wedge Y_n,\ x_1\wedge\dots\wedge x_n\mapsto f_1(x_1)\wedge\dots\wedge f_n(x_n)$. 
For a pointed space $X$ and an integer $n\ge 0$, we set $\Sigma^nX=X\wedge\s^n$ which is called 
the {\it $n$-fold pointed suspension} of $X$; for a pointed map $f:X\to Y$ 
we set $\Sigma^nf=f\wedge 1_{\s^n}:\Sigma^nX\to \Sigma^nY$. 

We identify $\s^n\ (n\ge 1)$ with $I^n/\partial I^n$ and $\s^1\wedge\cdots\wedge\s^1\ (n\text{-times})$ 
by the following way. 
Take and fix a relative homeomorphism 
$\psi_n:(I^n,\partial I^n)\to (\s^n,*)$ for each $n\ge 1$ (e.g. \cite[p.5]{T}). 
Identify $I^n/\partial I^n$ with $\s^n$ by the homeomorphism induced from $\psi_n$, and 
denote $\psi_n(t_1,\dots,t_n)$ by $\overline{t_1}\wedge\cdots\wedge\overline{t_n}$. 
Also identify $\s^n$ with $\s^1\wedge\cdots\wedge\s^1\ (n\text{-times})$ by the homeomorphism $h_n$ of  
the following commutative square with $q$ the quotient map. 
(Notice that $h_n(\overline{t_1}\wedge\cdots\wedge\overline{t_n})=\overline{t_1}\wedge\cdots\wedge\overline{t_n}$.) 
$$
\begin{CD}
I^n @>\psi_1\times\dots\times\psi_1>>  \s^1\times\dots\times\s^1 \\
@V\psi_nVV @VVqV\\
\s^n @>h_n>\approx>  \s^1\wedge\dots\wedge\s^1
\end{CD}
$$
Under the above identifications, we have $\s^m\wedge\s^n=\s^{m+n}=\s^n\wedge\s^m$, where, if 
$m,n\ge 1$, then 
\begin{gather*}
(x_1\wedge\cdots\wedge x_m)\wedge(x_{m+1}\wedge\cdots\wedge x_{m+n})
=x_1\wedge\cdots\wedge x_{m+n}\\
=(x_1\wedge\cdots\wedge x_n)\wedge(x_{n+1}\wedge\cdots\wedge x_{m+n})\quad (x_i\in\s^1\ (1\le i\le m+n)). 
\end{gather*}
Since spheres are compact and Hausdorff, it follows that, for any pointed space $X$, we have the identifications: 
\begin{equation}
\begin{split}
\Sigma^n\Sigma^mX&=(X\wedge\s^m)\wedge\s^n
=X\wedge(\s^m\wedge \s^n)=X\wedge\s^{m+n}\\
&=X\wedge(\s^n\wedge\s^m)
=(X\wedge\s^n)\wedge\s^m=\Sigma^m\Sigma^nX.
\end{split}
\end{equation}
The switching map 
\begin{equation}
\tau(\s^m,\s^n):\s^{m+n}=\s^m\wedge\s^n\to\s^n\wedge\s^m=\s^{m+n}, \ x\wedge y\mapsto y\wedge x,
\end{equation}
is a homeomorphism of the degree $(-1)^{mn}$. 

Given a space $A$, let $\mathrm{TOP}^A$ denote the category of spaces under $A$, that is, 
objects are free maps $i:A\to X$ and a morphism 
$f$ from $i:A\to X$ to $i':A\to X'$ is a free map $f:X\to X'$ with $f\circ i=i'$. 
\begin{equation}
\begin{split}
\xymatrix{
& A \ar[dl]_-i \ar[dr]^-{i'} &\\
X \ar[rr]^-f & & X'
}
\end{split}
\end{equation}
Let $\mathrm{TOP}^A(i,i')$ denote the set of all morphisms from $i:A\to X$ to $i':A\to X'$. 
For $f,f'\in\mathrm{TOP}^A(i,i')$, if there exists a homotopy $H:X\times I\to X'$ such that 
$H_0=f,\ H_1=f',\ H_t\in \mathrm{TOP}^A(i,i')$ for all $t\in I$, then we write $f\overset{A}{\simeq} f'$ or 
$H: f\overset{A}{\simeq}f'$. 
Note that $\mathrm{TOP}^{\{*\}}=\mathrm{TOP}^*$. 
The following is well-known (e.g. \cite[(3.6)]{D}, \cite[(5.2.5)]{tD}, \cite[(2.18)]{DKP}, \cite[(6.18)]{J}). 

\begin{prop}%2.1
Given a commutative triangle (2.3), if $i$ and $i'$ are cofibrations and $f:X\to X'$ is a homotopy equivalence 
in $\mathrm{TOP}$, then $f:i\to i'$ is a homotopy equivalence in $\mathrm{TOP}^A$, that is, 
there exists $g\in\mathrm{TOP}^A(i',i)$ with 
$g\circ f\overset{A}{\simeq}1_X$ and $f\circ g\overset{A}{\simeq} 1_{X'}$. 
\end{prop}

For spaces $X$ and $Y$, we denote by $X+Y$ the topological sum of them, that is, 
it is the disjoint union of them as a set and $A\subset X+Y$ is open if and only if $A\cap X$ is open in $X$ 
and $A\cap Y$ is open in $Y$. 

For a pointed space $X$, the {\it cone} $CX$ over it and the {\it suspension} $\Sigma X$ of it are 
defined by $CX=X\wedge I=(X\times I)/(\{x_0\}\times I\cup X\times\{1\})$ and 
$\Sigma X=(X\times I)/(\{x_0\}\times I\cup X\times\{0,1\})$.  
The point of $\Sigma X$ represented by $(x,t)\in X\times I$ is denoted by $x\wedge\overline{t}$. 
The space $\Sigma X$ is based by $x_0\wedge\overline{1}$. 
Usually we identify $\Sigma X=\Sigma^1X$. 
For a pointed map $f:X\to Y$, two maps $Cf:CX\to CY$ and 
$\Sigma f:\Sigma X\to \Sigma Y$ are defined by $Cf(x\wedge t)=f(x)\wedge t$ and 
$\Sigma f(x\wedge\overline{t})=f(x)\wedge\overline{t}$; 
the (pointed) {\it mapping cone} of $f$ is the space 
$C_{f}=Y\cup_{f} CX$ which is the quotient of $Y+CX$ by the equivalence relation generated by the relation 
$f(x)\sim x\wedge 0\ (x\in X)$ and is based by the point represented by $y_0$; 
the injection $i_f : Y\to Y\cup_f CX$ is a cofibration in $\mathrm{TOP}^*$ by \cite[Hilfssatz~6]{P} and so an embedding; 
let
\begin{gather*}
q_f : Y\cup_f CX\to (Y\cup_f CX)/Y=\Sigma X,\\
q_f':(Y\cup_f CX)\cup_{i_f}CY\to ((Y\cup_f CX)\cup_{i_f}CY)/CY=\Sigma X
\end{gather*}
denote the quotient maps, then $q_f=q_f'\circ i_{i_f}$ and  
$q_f'$ is a homotopy equivalence in $\mathrm{TOP}^*$ by \cite[Satz 3]{P}; for any integer $\ell\ge 1$ 
let $\psi^\ell_f:\Sigma^\ell Y\cup_{\Sigma^\ell f}C\Sigma^\ell X\approx \Sigma^\ell(Y\cup_f CX)$ 
denote the homeomorphism defined by $\psi^\ell_f(y\wedge s_\ell)=y\wedge s_\ell$ and 
$\psi^\ell_f(x\wedge s_\ell\wedge t)=x\wedge t\wedge s_\ell$ for $s_\ell\in\s^\ell$ and $t\in I$. 
If the first square of the following diagram in $\mathrm{TOP}^*$ is commutative, then there exists the map 
$b\cup Ca$ with the diagram commutative. 
$$
\xymatrix{
X\ar[r]^-f \ar[d]_-a &Y \ar[r]^-{i_f} \ar[d]_-b & Y\cup_f CX \ar[r]^-{q_f} \ar@{.>}[d]_-{b\cup Ca} & \Sigma X \ar[d]_-{\Sigma a} \\
X' \ar[r]^-{f'} &Y' \ar[r]^-{i_{f'}} &Y'\cup_{f'}CX' \ar[r]^-{q_{f'}} & \Sigma X'
}
$$

The next proposition is the same as (1.11) and shall be used to define {\it induced iterated mapping cones} in Definition~5.4. 

\begin{prop}%2.2
If a pointed map $j:A\to X$ is a free (resp.\,closed free) cofibration, 
then, for any pointed map $f:X\to Y$, $1_Y\cup Cj :Y\cup_{f\circ j}CA\to Y\cup_f CX$ is a free (resp.\,closed free) cofibration. 
\end{prop}

The above proposition may be folklorish, but we have not found its proof in the literature, 
and so we will prove it in Appendix~A for completeness. 

\begin{cor}
\begin{enumerate}
\item[(1)] If a pointed map $j:A\to X$ is a free (resp.\,closed free) cofibration, 
then $\Sigma j:\Sigma A\to\Sigma X$ is a free (resp.\,closed free) cofibration. 
\item[(2)] If $X$ is a $w$-space (resp.\,$clw$-space), 
then $\Sigma X$ and $CX$ are 
$w$-spaces (resp.\,$clw$-spaces), and 
$i_f:Y\to Y\cup_f CX$ is a free (resp.\,closed free) cofibration for every pointed map $f:X\to Y$. 
\item[(3)] If $f:X\to Y$ is a pointed map between $w$-spaces (resp.\,$clw$-spaces), 
then $Y\cup_f CX$ is a $w$-space (resp.\,$clw$-space). 
\end{enumerate}
\end{cor}
\begin{proof}
(1) By taking $Y=\{y_0\}$ in Proposition 2.2, the assertion follows. 

(2) Let $X$ be a $w$-space (resp.\,$clw$-space). 
Set $j:A=\{x_0\}\subset X$. 
The assertions about $\Sigma X$ and $i_f$ follow from (1) and Proposition 2.2. 
Since $i_{1_X}\circ j:\{x_0\}\to CX$ is a free (resp.\,closed free) cofibration, 
$CX$ is a $w$-space (resp.\,$clw$-space).  

(3) Let $X$ and $Y$ be $w$-spaces (resp.\,$clw$-spaces). 
Then $Y\cup_f CX$ is a $w$-space (resp.\,$clw$-spaces), since the composite of $\{y_0\}\subset Y$ 
with $i_f:Y\to Y\cup_f CX$ is a free (resp.\,closed free) cofibration. 
\end{proof}

Given pointed maps $f:X\to Y$ and $g:Y\to Z$ with a pointed null homotopy $H:g\circ f\simeq *$, we set
\begin{align*}
(g,H,f) &: \Sigma X\to Z\cup_g CY,\quad x\wedge\overline{t}\mapsto
\begin{cases} f(x)\wedge(1-2t) & 0\le t\le \frac{1}{2}\\ H(x,2t-1) & \frac{1}{2}\le t\le 1\end{cases},\\
[g,H,f] &: Y\cup_f CX\to Z,\quad y\mapsto g(y),\quad x\wedge t\mapsto H(x,t),
\end{align*}
which are called a {\it coextension} of $f$ with respect to $g$ and an {\it extension} of $g$ with respect to $f$, 
respectively (\cite{T,Og}). 
Given pointed maps $f_i:X_i\to X_{i+1}$ for $i=1,2,3$, the {\it Toda bracket} $\{f_3,f_2,f_1\}$ (\cite{T1,T3,T}) 
which is a subset of $[\Sigma X_1,X_4]$ 
is the set of homotopy classes of maps of the form $[f_3,A_2,f_2]\circ(f_2,A_1,f_1)$, where $A_j:f_{j+1}\circ f_j\simeq *$ for $j=1,2$. 
If $A_1$ or $A_2$ does not exist, then $\{f_3,f_2,f_1\}$ denotes the empty set. 
As is well-known, $\{f_3,f_2,f_1\}$ depends only on the homotopy classes of $f_i\ (i=1,2,3)$ (e.g. Section 3 of \cite{OO1}). 

\section{Maps between mapping cones} 

In this section we will work in $\mathrm{TOP}^*$. 

The following shall be used to prove Lemma~5.3 which defines induced iterated mapping cones. 

\begin{lemma}%3.1
Given two maps $j:Y\to Y'$ and $g':Y'\to Z$, 
the following diagram is homotopy commutative and $i_{g'\circ j}\cup Ci_j$ is a homotopy equivalence. 
$$
\xymatrix{
Y \ar@{=}[d] \ar[r]^-j & Y' \ar[d]_-{g'} \ar[r]^-{i_j} & Y'\cup_j CY \ar[d]_-{g'\cup C1_Y} \\
Y \ar[r]^-{g'\circ j} & Z \ar[d]_-{i_{g'}} \ar[r]^-{i_{g'\circ j}} & Z\cup_{g'\circ j} CY \ar[d]^-{i_{g'\cup C1_Y}} \ar[dl]_-{1_Z\cup Cj} \\
& Z\cup_{g'} CY' \ar[r]_-{i_{g'\circ j}\cup Ci_j} & (Z\cup_{g'\circ j} CY)\cup_{g'\cup C1_Y} C(Y'\cup_j CY)  
}
$$
\end{lemma}
\begin{proof}
Obviously three squares are commutative and $(1_Z\cup Cj)\circ i_{g'\circ j}=i_{g'}$. 
For simplicity, we set 
$$
g=g'\circ j, \quad h=g'\cup C1_Y, \quad 
k=1_Z\cup Cj,\quad \varphi=i_{g'\circ j}\cup Ci_j. 
$$
We should prove that 
$i_h\simeq \varphi\circ k$ and $\varphi$ is a homotopy equivalence. 
Let $z\in Z,\ y\in Y,\ y'\in Y'$ and $s,t,u\in I$. 
We define
$$
w:I\times I\to I,\quad G:(Z\cup_g CY)\times I\to (Z\cup_g CY)\cup_h C(Y'\cup_j CY)
$$ 
by
\begin{gather*}
w(s,t)=\begin{cases} 0 & s\le t\\ 2s-2t & \frac{s}{2}\le t\le s\\ s & t\le\frac{s}{2}\end{cases},\\  
G(z,t)=z,\quad G(y\wedge s, t)=y\wedge w(s,t)\wedge w(s,1-t).
\end{gather*}
Then $G:i_h\simeq \varphi\circ k\ \text{rel}\ Z$. 
In the rest of the proof we prove that $\varphi$ is a homotopy equivalence. 
Define $\psi:(Z\cup_g CY)\cup_h C(Y'\cup_j CY)\to Z\cup_{g'} CY'$ by 
\begin{gather*}
\psi(z)=z,\quad \psi(y\wedge t)=j(y)\wedge t,\quad 
\psi(y'\wedge t)=y'\wedge t,\\
\psi(y\wedge s\wedge t)=j(y)\wedge(s+(1-s)t).
\end{gather*}
As is easily seen, $\psi$ is well-defined, continuous, and $\psi\circ\varphi=1_{C_{g'}}$. 
We will show $1_{C_h}\simeq \varphi\circ\psi$. 
We have 
\begin{gather*}
\varphi\circ\psi(z)=z, \quad \varphi\circ\psi(y\wedge t)=j(y)\wedge t=y\wedge 0\wedge t, 
\quad \varphi\circ\psi(y'\wedge t)=y'\wedge t, \\
\varphi\circ\psi(y\wedge s\wedge t)=j(y)\wedge (s+(1-s)t)=y\wedge 0\wedge(s+(1-s)t).
\end{gather*}
Thus it suffices to construct a map 
$$H:\big((Z\cup_g CY)\cup_h C(Y'\cup_j CY)\big)\times I\to (Z\cup_g CY)\cup_h C(Y'\cup_j CY)$$ 
such that 
\begin{gather*}
H(z,u)=z,\quad H(y'\wedge t,u)=y'\wedge t,\\
H(y\wedge 0\wedge t,u)=H(j(y)\wedge t,u)=j(y)\wedge t=y\wedge 0\wedge t, \\ 
H(y\wedge t,0)=y\wedge t=y\wedge t\wedge 0,\quad  H(y\wedge t,1)=y\wedge 0\wedge t,\\
H(y\wedge s\wedge t,0)=y\wedge s\wedge t,\quad  
 H(y\wedge s\wedge t,1)=y\wedge 0\wedge(s+(1-s)t). 
\end{gather*}
The space 
$
K=I\times I\times\{0\}\cup \{0\}\times I\times I\cup I\times I\times\{1\}\cup I\times\{1\}\times I\cup \{1\}\times I\times I
$ 
is a retract of $I\times I\times I$. 
Indeed a retraction $r:I\times I\times I\to K$ is defined as follows: 
for $P\in I\times I\times I$, $r(P)$ is the intersection of $K$ and the half line which starts from 
$(\frac{1}{2},-\frac{1}{2},\frac{1}{2})$ and passes through $P$. 
Define $v':K\to I\times I$ by
\begin{gather*}
v'(s,t,0)=(s,t), \quad v'(0,t,u)=(0,t),\quad v'(s,t,1)=(0,s+(1-s)t),\\
 v'(s,1,u)=\begin{cases} (0,1) & s\le u\\ (2s-2u,1) & \frac{s}{2}\le u\le s\\ (s,1) & u\le \frac{s}{2}\end{cases},\\ 
v'(1,t,u)=\begin{cases} (1,t) & u\le  \frac{t}{2}\\ (1,2u) & \frac{t}{2}\le u\le \frac{1}{2}\\ (2-2u,1) & \frac{1}{2}\le u\le 1\end{cases}.
\end{gather*}
Then $v'$ is well-defined and continuous. 
Set $v=v'\circ r:I\times I\times I\to I\times I$. 
Then 
\begin{gather*}
v(\{1\}\times I\times I\cup I\times\{1\}\times I)\subset \{1\}\times I\cup I\times\{1\},\\
v(0,t,u)=(0,t),\quad v(s,t,0)=(s,t),\quad v(s,t,1)=(0,s+(1-s)t) .
\end{gather*}
Write $v(s,t,u)=(v_1(s,t,u),v_2(s,t,u))$ and define $H$ by 
\begin{gather*}
H(z,u)=z,\quad H(y\wedge s,u)=y\wedge v_1(s,0,u)\wedge v_2(s,0,u),\\
 H(y'\wedge t,u)=y'\wedge t,\quad 
H(y\wedge s\wedge t,u)=y\wedge v_1(s,t,u)\wedge v_2(s,t,u).
\end{gather*}
Then $H$ satisfies the desired properties. 
Therefore $\varphi$ is a homotopy equivalence. 
This completes the proof.
\end{proof}

\begin{defi}[(9) of \cite{P}]%3.2
Given a homotopy commutative square and a homotopy 
\begin{equation}
\begin{CD}
X@>f>>Y\\
@VaVV @VVbV\\
X'@>f'>>Y'
\end{CD}
\quad ,\quad J:b\circ f\simeq  f'\circ a,
\end{equation}
we define $\Phi(f,f',a,b;J):Y\cup_f CX\to Y'\cup_{f'}CX'$ by 
\begin{gather*}
\Phi(f,f',a,b;J)(y)=b(y),\\
\Phi(f,f',a,b;J)(x\wedge s)=\begin{cases} J(x,2s) & 0\le s\le \frac{1}{2}\\ a(x)\wedge (2s-1) & \frac{1}{2}\le s\le 1\end{cases}.
\end{gather*}
Given a homotopy $K:b\circ f\simeq f'\circ a$, 
if there is a free map $\varphi : X\times I\times I\to Y'$ such that 
$\varphi(x,s,0)=J(x,s)$, $\varphi(x,s,1)=K(x,s)$, $\varphi(*,s,t)=*$, $\varphi(x,0,t)=b\circ f(x)$, 
and $\varphi(x,1,t)=f'\circ a(x)$ for every $x\in X$ and $s,t\in I$, 
then we write $J\overset{X}{\simeq}K$ or $\varphi:J\simeq K$. 
\end{defi}

\begin{prop} %3.3
Suppose that (3.1) is given. 
\begin{enumerate}
\item[\rm(1)]$($\cite[Hilfssatz 7]{P}$)$
\begin{enumerate}
\item[\rm(a)] The following diagram is homotopy commutative such that 
the middle square is commutative. 
$$
\xymatrix{
X \ar[r]^-f \ar[d]_-a & Y \ar[r]^-{i_f} \ar[d]_-b & Y\cup_f CX \ar[r]^-{q_f} \ar[d]_-{\Phi(f,f',a,b;J)} & \Sigma X \ar[d]_-{\Sigma a}\\
X' \ar[r]^-{f'} & Y' \ar[r]^-{i_{f'}} & Y'\cup_{f'}CX' \ar[r]^-{q_{f'}} & \Sigma X'
}
$$
\item[\rm(b)] In the following diagram, the first square is commutative and the second square is homotopy commutative. 
$$
\xymatrix{
Y\cup_f CX \ar[r]^-{i_{i_f}} \ar[d]_-{\Phi(f,f',a,b;J)}  
& (Y\cup_f CX)\cup_{i_f}CY \ar[rr]^-{q_f'}  \ar[d]^-{\Phi(i_f,i_{f'},b,\Phi(f,f',a,b;J);1_{i_{f'}\circ b})} & & \Sigma X \ar[d]^-{\Sigma a} \\
Y'\cup_{f'}CX' \ar[r]^-{i_{i_{f'}}} & (Y'\cup_{f'} CX')\cup_{i_{f'}}CY' \ar[rr]^-{q_{f'}'} &  & \Sigma X'
}
$$
Also $\Phi(i_f,i_{f'},b,\Phi(f,f',a,b;J);1_{i_{f'}\circ b})\simeq \Phi(f,f',a,b;J)\cup Cb$. 
\item[\rm(c)] If $a$ and $b$ are homotopy equivalences, then $\Phi(f,f',a,b;J)$ is a homotopy equivalence. 
\item[\rm(d)] If furthermore $a':X'\to X'', b':Y'\to Y'', f'':X''\to Y''$ with $J':b'\circ f'\simeq f''\circ a'$ are given, then 
$$
\Phi(f',f'',a',b';J')\circ \Phi(f,f',a,b;J)\simeq \Phi(f,f'',a'\circ a,b'\circ b;(J'\overline{\circ}1_a)\bullet(1_{b'}\overline{\circ}J)).
$$
\end{enumerate}
\item[\rm(2)] Define $e_a:\Sigma X\to \Sigma X'$ by $e_a(x\wedge\overline{t})
=\begin{cases} a(x)\wedge\overline{0} & 0\le t\le\frac{1}{2}\\ a(x)\wedge\overline{2t-1} & \frac{1}{2}\le t\le 1\end{cases}$. 
Then $e_a\simeq \Sigma a$ and $q_{f'}'\circ (\Phi(f,f',a,b;J)\cup Cb)=e_a\circ q_f'\simeq \Sigma a\circ q_f'$. 
\item[\rm(3)] If the square in (3.1) is strictly commutative, then $\Phi(f,f',a,b;1_{b\circ f})\linebreak\simeq b\cup Ca$. 
\item[\rm(4)]{\rm (\cite[p.315]{P})} For homotopies $a_t:X\to X'$ and $b_t:Y \to Y'$, if there exists a homotopy 
$J^t : b_t\circ f\simeq f'\circ a_t$ for every $t\in I$ such that the function 
$X\times I\times I\to Y',\ (x,s,t)\mapsto J^t(x,s),$ is continuous, then the function 
$
\Phi:(Y\cup_f CX)\times I\to Y'\cup_{f'}CX',\  (z,t)\mapsto \Phi(f,f',a_t,b_t;J^t)(z),
$ 
is continuous and so 
$$\Phi(f,f',a_0,b_0;J^0)\simeq \Phi(f,f',a_1,b_1;J^1).$$

\item[\rm(5)] If $K:b\circ f\simeq f'\circ a$ satisfies $J\overset{X}{\simeq} K$, then
$\Phi(f,f',a,b;J) \simeq \Phi(f,f',a,b;K)$ as elements of $\mathrm{TOP}^Y(i_f,i_{f'}\circ b)$. 
\end{enumerate}
\end{prop}
\begin{proof}
We refer a proof of (1) to \cite{P}. 

Define $v:I\times I\to I$ and $F:\Sigma X\times I\to \Sigma X'$ by 
\begin{gather*}
v(t,u)=\begin{cases} 0 & u\le -2t+1\\ t+u/2-1/2 & 2t-1\le u\ \text{and}\, -2t+1\le u\\ 2t-1 & u\le 2t-1\end{cases},\\ 
 F(x\wedge\overline{t},u)=a(x)\wedge\overline{v(t,u)}.
\end{gather*}
Then $F: e_a\simeq \Sigma a$. 
As is easily seen, $q_{f'}'\circ (\Phi(f,f',a,b;J)\cup Cb)=e_a\circ q_f'$. 
Hence we obtain (2). 

(3) can be easily proved. 

For (4), define $\xi: (Y+X\times I)\times I\to Y'\cup_{f'}CX'$ by 
$$
\xi(y,t)=b_t(y),\quad \xi(x,s,t)=\begin{cases} J^t(x,2s) & 0\le s\le \frac{1}{2}\\ a_t(x)\wedge(2s-1) & \frac{1}{2}\le s\le 1\end{cases}. 
$$
Then it is continuous and satisfies $\xi=\Phi\circ(q\times 1_I)$, where $q:Y+X\times I\to Y\cup_f CX$ is the quotient map. 
Hence $\Phi$ is continuous. 
This proves (4). 

(5) is obtained by taking $a_t=a,\ b_t=b,\ J^t(x,s)=\varphi(x,s,t)$ in (4), where $\varphi:J\simeq K$. 
\end{proof}

\section{Homotopy cofibres}

In this section we will work in $\mathrm{TOP}^*$. 
Hence $i_f:Y\to Y\cup_f CX$ is always a cofibration for every map $f:X\to Y$. 

\begin{defi}%4.1
A map $j:Y\to Z$ is a {\it homotopy cofibre} of a map $f:X\to Y$ if $j$ is a cofibration and there exists 
a homotopy equivalence $a : Z\to Y\cup_f CX$ with $a\circ j\simeq i_f$.
\end{defi}

The notion ``homotopy cofibre'' is not new. 
Indeed we have the following.

\begin{lemma}%Lemma 4.2
Given maps $f:X\to Y$ and $j:Y\to Z$, $j$ is a homotopy cofibre of $f$ if and only if 
$j:Y\to Z$ is a cofibration and 
$X\overset{f}{\longrightarrow}Y\overset{j}{\longrightarrow}Z$ is a cofibre sequence, that is, 
there exists a homotopy commutative diagram with $b,c,d$ homotopy equivalences:
$$
\xymatrix{
X \ar[r]^-f \ar[d]_-b^\simeq & Y \ar[r]^-j \ar[d]_-c^\simeq & Z \ar[d]_-d^\simeq \\
X' \ar[r]^-{f'} & Y' \ar[r]^-{i_{f'}} & Y'\cup_{f'}CX'
}
$$
\end{lemma}
\begin{proof}
It suffices to prove ``if''-part. 
Let $J:c\circ f\simeq f'\circ b$. 
Then $\Phi=\Phi(f,f',b,c;J):Y\cup_f CX\to Y'\cup_{f'}CX'$ is a homotopy equivalence with $\Phi\circ i_f=i_{f'}\circ c$ 
by Proposition~3.3(1)(c). 
Set $a=\Phi^{-1}\circ d:Z\to Y\cup_f CX$. 
Then $a$ is a homotopy equivalence such that 
$a\circ j=\Phi^{-1}\circ d\circ j\simeq  \Phi^{-1}\circ i_{f'}\circ c= \Phi^{-1}\circ\Phi\circ i_f\simeq i_f$. 
Hence $j$ is a homotopy cofibre of $f$. 
\end{proof}

\begin{lemma}%4.3
Let $j:Y\to Z$ be a homotopy cofibre of $f:X\to Y$. 
\begin{enumerate}
\item[\rm(1)] There is a homotopy equivalence $a\in\mathrm{TOP}^Y(j,i_f)$ 
and its homotopy inverse $a^{-1}\in\mathrm{TOP}^Y(i_f,j)$ such that 
$a^{-1}\cup C1_Y:(Y\cup_f CX)\cup_{i_f}CY\to Z\cup_jCY$ is a homotopy inverse of $a\cup C1_Y$, that is, 
\begin{equation}
\left\{\begin{array}{@{\hspace{0.6mm}}l}
(a^{-1}\cup C1_Y)\circ (a\cup C1_Y) \simeq 1_{Z\cup_j CY},\\
(a\cup C1_Y)\circ(a^{-1}\cup C1_Y) \simeq 1_{(Y\cup_f CX)\cup_{i_f} CY}.
\end{array}\right.
\end{equation}
\item[\rm(2)] If $f':X\to Y$ satisfies $f\simeq f'$, then $j$ is a homotopy cofibre of $f'$. 
\item[\rm(3)] If $h:X\to X$ is a homotopy equivalence, then 
$j$ is a homotopy cofibre of $f\circ h$. 
\item[\rm(4)] If $f=g\circ h:X\overset{h}{\rightarrow}X'\overset{g}{\rightarrow}Y$ with $h$ a homotopy equivalence, 
then $j$ is a homotopy cofibre of $g$. 
\item[\rm(5)] If $j$ is a free cofibration, then $\Sigma^\ell j:\Sigma^\ell Y\to \Sigma^\ell Z$ is a 
homotopy cofibre of $\Sigma^\ell f:\Sigma^\ell X\to \Sigma^\ell Y$ for any positive integer $\ell$. 
\item[\rm(6)] If $h:Y\to Y'$ and $k:Z\to Z'$ are homeomorphisms, then 
$j'=k\circ j\circ h^{-1}:Y'\to Z'$ is a homotopy cofibre of $h\circ f:X\to Y'$. 
\item[\rm(7)] Given a map $g:Y\to W$, $g$ can be extended to Z if and only if $g\circ f\simeq *$. 
\end{enumerate}
\end{lemma}
\begin{proof}
(1) Suppose that $a':Z\to Y\cup_f CX$ is a homotopy equivalence and $g:a'\circ j\simeq i_f$ is a homotopy. 
Since $j$ is a cofibration by the assumption, there exists a homotopy $H:Z\times I\to Y\cup_f CX$ with 
$a'=H\circ i_0^Z$ and $g=H\circ (j\times 1_I)$. 
Then the map $a:Z\to Y\cup_f CX, z\mapsto H(z,1),$ is a homotopy equivalence with $a\circ j=i_f$ and so 
$a$ is a homotopy equivalence in $\mathrm{TOP}^Y(j,i_f)$ by Proposition~2.1. 
Let $a^{-1}\in \mathrm{TOP}^Y(i_f,j)$ be a homotopy inverse of $a$. 
Then $a^{-1}\circ i_f=j$ and there exist homotopies $K: a^{-1}\circ a\overset{Y}{\simeq} 1_Z$ and 
$L : a\circ a^{-1}\overset{Y}{\simeq} 1_{Y\cup_fCX}$. 
Hence $(a^{-1}\cup C1_Y)\circ (a\cup C1_Y)=K_0\cup C1_Y\simeq K_1\cup C1_Y=1_{Z\cup_j CY}$ 
and the second equation of (4.1) is obtained similarly. 
This proves (1). 

In the rest of the proof $a:Z\to Y\cup_fCX$ is a homotopy equivalence such that $a\circ j=i_f$. 

(2) Suppose that $J:f\simeq f'$. 
By Proposition 3.3(1), 
$$
\Phi(J):=\Phi(f,f',1_X,1_Y;J):Y\cup_f CX\to Y\cup_{f'}CX
$$
is a homotopy equivalence and $\Phi(J)\circ i_f = i_{f'}$. 
Hence $\Phi(J)\circ a:Z\to Y\cup_{f'}CX$ is a homotopy equivalence and $\Phi(J)\circ a\circ j=i_{f'}$. 
This proves (2). 

(3) Take $J : f\simeq f\circ h\circ h^{-1}$. 
Then $\Phi(f,f\circ h,h^{-1},1_Y;J)$ is a homotopy equivalence and $\Phi(f,f\circ h,h^{-1},1_Y;J)\circ a\circ j=i_{f\circ h}$. 
Hence $j$ is a homotopy cofibre of $f\circ h$. 

(4) Since $(1_Y\cup Ch)\circ a\circ j=(1_Y\cup Ch)\circ i_f=i_g$, (4) follows. 

(5) Suppose that $j$ is a free cofibration. 
Then $\Sigma^\ell  j$ is a free cofibration by Corollary~2.3(1). 
We set $a'=(\psi^\ell _f)^{-1}\circ \Sigma^\ell a$, where $\psi^\ell _f$ is the homeomorphism 
$\Sigma^\ell  Y\cup_{\Sigma^\ell  f}C\Sigma^\ell  X\approx \Sigma^\ell (Y\cup_f CX)$ defined in the section 2. 
Then $a':\Sigma^\ell  Z\to \Sigma^\ell  Y\cup_{\Sigma^\ell  f}C\Sigma^\ell  X$ is a homotopy equivalence 
and $a'\circ \Sigma^\ell  j=i_{\Sigma^\ell  f}$. 
Hence $\Sigma^\ell  j$ is a homotopy cofibre of $\Sigma^\ell  f$.

(6) Since $(h\cup C1_X)\circ a\circ k^{-1} : Z'\to Y'\cup_{h\circ f}CX$ is a homotopy equivalence and 
$(h\cup C1_X)\circ a\circ k^{-1}\circ j'=i_{h\circ f}$, it follows that $j'$ is a homotopy cofibre of $h\circ f$. 
This proves (6). 

(7) Let $a:Z\to Y\cup_f CX$ be a homotopy equivalence such that $a\circ j=i_f$, 
and $a^{-1}$ a homotopy inverse of $a$ such that $a^{-1}\circ i_f=j$. 
If $g\circ f\simeq *$, then there exists $\widetilde{g}:Y\cup_f CX\to W$ such that $\widetilde{g}\circ i_f=g$ and 
$\widetilde{g}\circ a$ is an extension of $g$ to $Z$. 
If $g$ has an extension $g':Z\to W$, then $g\circ f=g'\circ a^{-1}\circ i_f\circ f\simeq g'\circ a^{-1}\circ *=*$. 
\end{proof}

\begin{rem}%4.4
A map $a$ in Lemma 4.3(1) is not necessarily unique in the sense of \/ $\overset{Y}{\simeq}$. 
\end{rem}
\begin{proof}
Let $\nabla:\s^1\vee\s^1\to\s^1$ be the folding map. 
Then $\s^1\cup_{\nabla} C(\s^1\vee\s^1)=\s^2$ and 
$i_\nabla:\s^1\to\s^2$ can be identified with $j:\s^1\to\s^2,\ (x,y)\mapsto(x,y,0)$. 
Obviously $j$ is a homotopy cofibre of $\nabla$. 
We set $a:\s^2\to\s^2,\ (x,y,z)\mapsto (x,y,-z)$. 
Then $a,1_{\s^2}$ are homotopy equivalences in $\mathrm{TOP}^{\s^1}(j,i_\nabla)$. 
Their degrees are $-1$ and $1$, respectively. 
Hence $a\not\simeq 1_{\s^2}$ and so $a\overset{\s^1}{\simeq}1_{\s^2}$ does not hold. 
\end{proof}

\begin{lemma}%4.5
If $j:Y\to Z$ is a homotopy cofibre of $f:X\to Y$ and if a map $g:Y\to W$ satisfies $g\circ f\simeq *$, then, 
for any homotopy $A:g\circ f\simeq *$, we have 
$$
[g,A,f]\cup C1_Y\simeq (g,A,f)\circ q_f' : (Y\cup_f CX)\cup_{i_f} CY\to W\cup_g CY
$$
and, for any homotopy equivalence $a:Z\to Y\cup_f CX$ satisfying $a\circ j=i_f$, we have 
\begin{equation}
 ([g,A,f]\circ a\cup C1_Y)\circ\omega^{-1}=([g,A,f]\cup C1_Y)\circ(a\cup C1_Y)\circ\omega^{-1}\simeq (g,A,f),
\end{equation}
where $\omega^{-1}$ is a homotopy inverse of $\omega=q_f'\circ (a\cup C1_Y): Z\cup_j CY\to \Sigma X$. 
\end{lemma}
\begin{proof}
Consider the following diagram. 
$$
\xymatrix{
 & Y \ar@{=}[d] \ar[r]^-j & Z \ar[d]_-a^-\simeq \ar[r]^-{i_j} & Z\cup_j CY 
\ar[d]_-{a\cup C1_Y}^-\simeq \ar@/^/[dr]^-{\omega} &\\
X\ar[r]^-f & Y \ar[dr]_-g \ar[r]^-{i_f} & Y\cup_f CX \ar[d]^-{[g,A,f]} \ar[r]^-{i_{i_f}} & (Y\cup_f CX)\cup_{i_f} CY 
\ar[d]_-{[g,A,f]\cup C1_Y} \ar[r]^-{q_f'}_-\simeq & \Sigma X\ar@/^/[dl]^-{(g,A,f)}\\
& & W \ar[r]_-{i_g} & W\cup_g CY & \\
}
$$
The above diagram is commutative except the right lower triangle. 
Define $u:I\times I\to I$ and $H:((Y\cup_f CX)\cup_{i_f}CY)\times I\to W\cup_g CY$ by 
\begin{gather*}
u(s,t)=\begin{cases} s & s\ge t\\ 2s-t & 2s\ge t\ge s\\ -2s+t & 2s\le t\end{cases},\quad 
H(x\wedge s,t)=\begin{cases} f(x)\wedge u(s,t) & 2s\le t\\ A(x,u(s,t)) & 2s\ge t\end{cases},\\
H(y,t)=y\wedge t,\quad  H(y\wedge s,t)=y\wedge \max\{s,t\}.
\end{gather*}
Then $H:[g,A,f]\cup C1_Y\simeq (g,A,f)\circ q_f'$. 
Hence 
\begin{align*}
&([g,A,f]\circ a  \cup C1_Y)\circ\omega^{-1}=([g,A,f]\cup C1_Y)\circ (a\cup C1_Y)\circ\omega^{-1}\\
&\simeq (g,A,f)\circ q_f'\circ  (a\cup C1_Y)\circ\omega^{-1}=(g,A,f)\circ\omega\circ\omega^{-1}\simeq (g,A,f) .
\end{align*}
\end{proof}

\section{Iterated mapping cones}
In this section we will work in $\mathrm{TOP}^w$. 
Hence for any map $f:X\to Y$ the injection $i_f:Y\to Y\cup_f CX$ is 
in $\mathrm{TOP}^w$ and a free cofibration by Corollary~2.3(2),(3). 

By replacing the words ``$w$-space'' and ``free cofibration'' with 
``$clw$-space'' and ``closed free cofibration'' 
respectively, we can develop consideration of this section similarly in $\mathrm{TOP}^{clw}$. 

We will revise the notion of ``shaft'' of Gershenson \cite{G} and rename it ``iterated mapping cone''. 
Suppose that the diagram   
\begin{equation}
\begin{split}
\begin{CD}
 X_1 @. X_2 @. X_3 @. \cdots @. X_n @. \ \\
 @Vg_1VV @Vg_2VV @Vg_3VV @. @Vg_nVV @. \\
C_1@>j_1>>C_2@>j_2>>C_3@>j_3>>\cdots@>j_{n-1}>>C_n@>j_n>>C_{n+1}
\end{CD}
\end{split}
\end{equation}
is given with $n\ge 1$, where $j_s:C_s\to C_{s+1}$ is a ``free'' cofibration for every $s$. 
We denote the above diagram by 
$$
\mathscr{S}=(X_1,\dots,X_n;C_1,\dots,C_{n+1};g_1,\dots,g_n;j_1,\dots,j_n).
$$
We often add $C_0=\{*\}$ and the inclusion $j_0:C_0\to C_1$ to the above diagram. 

\begin{defi} %5.1
\begin{enumerate}
\item[\rm(1)] The sequence $(g_1,j_1,\dots,j_n)$ is called the {\it edge} of $\mathscr{S}$.  
\item[\rm(2)] $\mathscr{S}$ is a {\it quasi iterated mapping cone} of depth $n$ 
if $C_{s+1}\cup_{j_s}CC_s\simeq \Sigma X_s$ and 
$[X_s,Z]\overset{g_s^*}{\longleftarrow}[C_s,Z]\overset{j_s^*}{\longleftarrow}[C_{s+1},Z]$ is exact 
as a sequence of pointed sets for every space $Z$ and every $s\ge 1$ (cf.\,\cite[p.68]{tD}). 
If we choose a homotopy equivalence $\omega_s:C_{s+1}\cup_{j_s}CC_s\simeq \Sigma X_s$ for each $s\ge 1$, 
then the set $\Omega=\{\omega_s\,|\,1\le s\le n\}$ is called a {\it quasi-structure} on $\mathscr{S}$. 
We set $\omega_0=1_{C_1}:C_1\cup_{j_0}CC_0=C_1\to C_1$. 
\item[\rm(3)] $\mathscr{S}$ is an {\it iterated mapping cone} of depth $n$ if $j_s$ is 
a homotopy cofibre of $g_s$ for every $s\ge 1$. 
In this case a homotopy equivalence $a_s:C_{s+1}\longrightarrow C_s\cup_{g_s}CX_s$ 
and its homotopy inverse $a_s^{-1}$ can be taken such that 
\begin{equation}
a_s\circ j_s=i_{g_s},\ a_s^{-1}\circ i_{g_s}=j_s,\ a_s^{-1}\circ a_s\overset{C_s}{\simeq}1_{C_{s+1}},\ 
a_s\circ a_s^{-1}\overset{C_s}{\simeq}1_{C_s\cup_{g_s} CX_s}.
\end{equation}
If we choose such a homotopy equivalence $a_s$ for each $s\ge 1$, then 
we call the set $\mathscr{A}=\{a_s\,|\,1\le s\le n\}$ a {\it structure} on $\mathscr{S}$, 
and we set $\omega_s=q_{g_s}'\circ (a_s\cup C1_{C_s})$ and $\Omega(\mathscr{A})=\{\omega_s\,|\,1\le s\le n\}$ 
which is a quasi-structure on $\mathscr{S}$. 
\item[\rm(4)] $\mathscr{S}$ is {\it reduced} if $C_2=C_1\cup_{g_1}CX_1$ and $j_1=i_{g_1}$. 
A quasi-structure $\Omega$ on a reduced quasi iterated mapping cone is {\it reduced} if $\omega_1=q_{g_1}'$. 
A structure $\mathscr{A}$ on a reduced iterated mapping cone is {\it reduced} if $a_1=1_{C_2}$. 
\item[\rm(5)] Given a map $f:C_1\to Y$, we denote by $\overline{f}^s:C_s\to Y$ 
an extension of $f$ to $C_s$, that is, 
$f=\begin{cases} \overline{f}^1 & s=1\\ \overline{f}^s\circ j_{s-1}\circ\cdots\circ j_1 & s\ge 2\end{cases}$. 
We set $\overline{f}^0=*:C_0\to Y$.
\end{enumerate}
\end{defi}

\begin{conv}%5.2
When $\mathscr{S}$ is an iterated mapping cone of depth $n$ with a structure $\{a_s\,|\,1\le s\le n\}$, 
we denote by $a_s^{-1}$ a homotopy inverse of $a_s$ such that it satisfies (5.2).
\end{conv}

Note that an iterated mapping cone is a quasi iterated mapping cone. 
When $\mathscr{S}$ is a reduced iterated mapping cone, a structure $\mathscr{A}$ on $\mathscr{S}$ 
is reduced if and only if $\Omega(\mathscr{A})$ is reduced. 
Notice also that a quasi iterated mapping cone is a revised version of the one called a shaft by Gershenson 
in \cite[Definition~1.2D]{G} where he did not suppose that the cofibrations $j_i$ are free. 

Let $\mathscr{S}=(X_1,\dots,X_n;C_1,\dots,C_{n+1};g_1,\dots,g_n;j_1,\dots,j_n)$ 
be a quasi iterated mapping cone of depth $n$
with a quasi-structure $\Omega=\{\omega_s\,|\,1\le s\le n\}$ and $f:C_1\to Y$ 
a map with an extension $\overline{f}^{n+1}$ to $C_{n+1}$. 
We define maps for $0\le s\le n$ as follows:
\begin{equation}
\left\{\begin{array}{@{\hspace{0.2mm}}l}
\overline{f}^s=\overline{f}^{n+1}\circ j_n\circ\cdots\circ j_s:C_s\to Y,\\
h_{s+1}=\overline{f}^{s+1}\cup C1_{C_s} : C_{s+1}\cup_{j_s} CC_s\to Y\cup_{\overline{f}^s} CC_s,\\
k_{s+1} =1_Y\cup Cj_s : Y\cup_{\overline{f}^s} CC_s\to Y\cup_{\overline{f}^{s+1}} CC_{s+1},\\
\widetilde{g}_{s+1}=
\begin{cases} f:C_1\to Y & s=0\\
 h_{s+1}\circ\omega_s^{-1} : 
\Sigma X_s \to Y\cup_{\overline{f}^s}CC_s & s\ge 1\end{cases},\\
\xi_{s+1}:(Y\cup_{\overline{f}^{s+1}}CC_{s+1})\cup_{k_{s+1}}C(Y\cup_{\overline{f}^s}CC_s)\\
\hspace{2cm}\to (Y\cup_{1_Y}CY)\cup_{\overline{f}^{s+1}\cup C\overline{f}^s}C(C_{s+1}\cup_{j_s}CC_s),\\
\hspace{0.5cm} y\mapsto y,\ c_{s+1}\wedge t\mapsto c_{s+1}\wedge t,\ y\wedge t\mapsto y\wedge t,\ c_s\wedge u\wedge t
\mapsto c_s\wedge t\wedge u,\\
\widetilde{\omega_s}=\Sigma\omega_s\circ q_{\overline{f}^{s+1}\cup C\overline{f}^s}\circ \xi_{s+1}
 :C_{k_{s+1}}\to \begin{cases} \Sigma C_1 & s=0\\ \Sigma\Sigma X_s & s\ge 1\end{cases},
\end{array}\right.
\end{equation}
where $y\in Y,\ c_{s+1}\in C_{s+1},\ c_s\in C_s,\ t,u\in I$, and $\omega_s^{-1}$ is a homotopy inverse of $\omega_s$. 
Since $\omega_s^{-1}$ is determined by $\omega_s$ up to homotopy, so is $\widetilde{g}_{s+1}$ for $s\ge 1$. 

\begin{lemma}%5.3
Under the above situation, we have $C_{\overline{f}^0}=Y$, $\overline{f}^1=h_1=\widetilde{g}_1=f$, $k_1=i_{f}$, 
$\widetilde{\omega_0}=q_f'$, $\xi_{s+1}$ is a homeomorphism, $\widetilde{\omega_s}$ is a homotopy equivalence, and 
the following diagram (5.4) is a reduced iterated mapping cone of depth $n+1$ with a reduced quasi-structure 
$\widetilde{\Omega}=\{\widetilde{\omega_s}\,|\,0\le s\le n\}$.
\begin{equation}
\begin{CD}
C_1 @. \Sigma X_1 @. \Sigma X_2 @.   \cdots @. \Sigma X_n @. \ \\
@V \widetilde{g}_1VV @V \widetilde{g}_2VV @V \widetilde{g}_3VV @. @V \widetilde{g}_{n+1}VV @.\\
C_{\overline{f}^0} @>k_1>> C_{\overline{f}^1} @>k_2>> C_{\overline{f}^2} @>k_3>> \cdots @>k_n>> 
C_{\overline{f}^n} @ >k_{n+1}>> C_{\overline{f}^{n+1}}
\end{CD}
\end{equation}
\end{lemma}
\begin{proof}
Since $k_1=i_{\widetilde{g}_1}$, $k_1$ is a free cofibration and a homotopy cofibre of $\widetilde{g}_1$, 
and $\widetilde{\omega_0}=q_{\widetilde{g}_1}'$. 
Let $1\le s\le n$. 
By Proposition~2.2, $k_{s+1}$ is a free cofibration. 
Take $J: h_{s+1}\simeq \widetilde{g}_{s+1}\circ \omega_s$ and set 
$\Phi(J,s+1)=\Phi(h_{s+1}, \widetilde{g}_{s+1}, \omega_s, 1_{C_{\overline{f}^s}};J)$. 
Then we have the following diagram. 
$$
\xymatrix{
C_s\ar[r]^-{j_s} \ar@{=}[d] & C_{s+1} \ar[r]^-{i_{j_s}} \ar[d]_-{\overline{f}^{s+1}} & C_{s+1}\cup_{j_s}CC_s 
\ar[d]_-{h_{s+1}} \ar[r]^-{\omega_s}_-\simeq 
& \Sigma X_s \ar[d]_-{\widetilde{g}_{s+1}} \\
C_s \ar[r]^-{\overline{f}^s} & Y \ar[r]^-{i_{\overline{f}^s}} \ar[d]_-{i_{\overline{f}^{s+1}}} 
& Y\cup_{\overline{f}^s}CC_s \ar[dl]_-{k_{s+1}} \ar[d]^-{i_{h_{s+1}}} 
\ar@{=}[r] & C_{\overline{f}^s} \ar[d]_-{i_{\widetilde{g}_{s+1}}} \\
& C_{\overline{f}^{s+1}}  \ar[r]_-{i_{\overline{f}^s}\cup Ci_{j_s}}^-\simeq 
& C_{h_{s+1}}  
 \ar[r]^-\simeq_-{\Phi(J,s+1)} & C_{\widetilde{g}_{s+1}} 
}
$$
By Proposition 3.3(1), $\Phi(J,s+1)$ is a homotopy equivalence and $\Phi(J,s+1)\circ i_{h_{s+1}}=i_{\widetilde{g}_{s+1}}$. 
By Lemma~3.1, $i_{\overline{f}^s}\cup Ci_{j_s}$ is a homotopy equivalence and 
$(i_{\overline{f}^s}\cup Ci_{j_s})\circ k_{s+1}\simeq i_{h_{s+1}}$. 
Hence $\Phi(J,s+1)\circ(i_{\overline{f}^s}\cup Ci_{j_s})\circ k_{s+1}\simeq i_{\widetilde{g}_{s+1}}$. 
Thus $k_{s+1}$ is a homotopy cofibre of $\widetilde{g}_{s+1}$. 
Hence (5.4) is a reduced iterated mapping cone of depth $n+1$. 
As is easily seen, $\xi_{s+1}$ is a homeomorphism, and 
$q_{\overline{f}^{s+1}\cup C\overline{f}^s}:C_{\overline{f}^{s+1}\cup C\overline{f}^s}\to \Sigma(C_{s+1}\cup_{j_s}CC_s)$ 
and $\Sigma\omega_s$ are homotopy equivalences. 
Hence $\widetilde{\omega_s}$ is a homotopy equivalence. 
Therefore $\widetilde{\Omega}$ is a reduced quasi-structure on (5.4). 
\end{proof}

\begin{defi}%5.4
We denote the iterated mapping cone (5.4) by $\mathscr{S}(\overline{f}^{n+1},\Omega)$, that is,
\begin{align*}
\mathscr{S}(\overline{f}^{n+1},\Omega)&=
\big(C_1,\Sigma X_1,\dots,\Sigma X_n;Y,C_{\overline{f}^1},\dots,C_{\overline{f}^{n+1}}; \\
&\hspace{5cm}f,\widetilde{g}_2,\dots,\widetilde{g}_{n+1};k_1,\dots,k_{n+1}\big),
\end{align*}
and call it the {\it iterated mapping cone induced from} $\mathscr{S}$ 
{\it by} $\overline{f}^{n+1}$ {\it and} $\Omega$, 
and we call $\widetilde{\Omega}=\{\widetilde{\omega_s}\,|\,0\le s\le n\}$ the 
{\it typical quasi-structure} on $\mathscr{S}(\overline{f}^{n+1},\Omega)$. 
When $\mathscr{S}$ is an iterated mapping cone with a structure $\mathscr{A}$, 
we denote the reduced iterated mapping cone $\mathscr{S}(\overline{f}^{n+1},\Omega(\mathscr{A}))$ 
by $\mathscr{S}(\overline{f}^{n+1},\mathscr{A})$ and call it the {\it iterated mapping cone induced from} 
$\mathscr{S}$ {\it by} $\overline{f}^{n+1}$ {\it and} $\mathscr{A}$. 
(Notice that we do not have typical structure on $\mathscr{S}(\overline{f}^{n+1},\Omega)$ 
even if $\mathscr{S}$ is an iterated mapping cone.)
\end{defi}

\begin{rem}%5.5
We have easily the following from definitions. 
\begin{enumerate}
\item[\rm(1)] The iterated mapping cone $\mathscr{S}(\overline{f}^{n+1},\Omega)$ depends 
on the map $\overline{f}^{n+1}$ and spaces $X_1,\dots,X_n$ but not on maps $g_1$, $\dots$, $g_n$.
\item[\rm(2)] The edge of $\mathscr{S}(\overline{f}^{n+1},\Omega)$ does not depend on $\Omega$. 
\item[\rm(3)] If two quasi iterated mapping cones $\mathscr{S}, \mathscr{S}'$ of depth $n$ have the same edge 
$X_1\overset{g_1}{\rightarrow}C_1\overset{j_1}{\rightarrow}C_2\overset{j_2}{\rightarrow}\cdots\overset{j_n}{\rightarrow}C_{n+1}$ 
and if a map $f:C_1\to Y$ has an extension $\overline{f}^{n+1}$ to $C_{n+1}$, then two iterated mapping cones 
$\mathscr{S}(\overline{f}^{n+1},\Omega), \mathscr{S}'(\overline{f}^{n+1},\Omega')$ have the same edge for any quasi-structures 
$\Omega, \Omega'$ on $\mathscr{S},\mathscr{S}'$, respectively.
\end{enumerate}
\end{rem}

If the following problem is solved affirmatively, the number of systems which shall be defined 
in the next section decreases by 2 to 10. 

\begin{prob}%5.6
Is every quasi iterated mapping cone of depth $1$ an iterated mapping cone?
\end{prob}

\section{Unstable higher Toda brackets} 

In this section we will work in $\mathrm{TOP}^w$. 
(As indicated in the previous section we can develop our consideration of this section similarly in $\mathrm{TOP}^{clw}$.) 

Sometimes, without particular comments, we do not distinguish in notation between a map and its homotopy class. 

Throughout the section 6, we denote by $\vec{\bm \alpha}=(\alpha_n,\dots,\alpha_1)$ a sequence of homotopy classes  
\begin{equation}
\alpha_i\in[X_i,X_{i+1}]\quad (i=1,2,\dots,n;\ n\ge 3).
\end{equation}
If a map $f_i:X_i\to X_{i+1}$ represents $\alpha_i$, then the sequence $\vec{\bm f}=(f_n,\dots,f_1)$ is called 
a {\it representative} of $\vec{\bm \alpha}$. 
We denote by $\mathrm{Rep}(\vec{\bm \alpha})$ the set of representatives of $\vec{\bm \alpha}$. 

\numberwithin{equation}{subsection}

\subsection{Definition of higher Toda brackets}

Given $\vec{\bm f}\in\mathrm{Rep}(\vec{\bm \alpha})$, 
we consider collections $\{\mathscr{S}_r,\overline{f_r},\Omega_r\,|\,2\le r\le n\}$, 
$\{\mathscr{S}_2,\overline{f_2},\Omega_2\}\cup\{\mathscr{S}_r,\overline{f_r},\mathscr{A}_r\,|\,3\le r\le n\}$, 
and 
$\{\mathscr{S}_r,\overline{f_r},\mathscr{A}_r\,|\,2\le r\le n\}$ (provided $\mathscr{S}_2$ is an iterated mapping cone) 
which satisfy the following (i), (ii), and (iii). 
(There is a possibility that such collections do not exist for suitable $\vec{\bm f}$.) 
\begin{enumerate}
\item[(i)] $\mathscr{S}_2$ is a quasi iterated mapping cone of depth $1$ 
as displayed in 
$$
\xymatrix{
X_1 \ar[d]_-{f_1} &\\
X_2 \ar[r]^-{j_{2,1}} & C_{2,2} 
}
$$
with $\Omega_2$ a quasi-structure and $\mathscr{A}_2$ a structure provided $\mathscr{S}_2$ is an iterated mapping cone. 
\item[(ii)] 
$\mathscr{S}_r$ is an iterated mapping cone of depth $r-1$ for $3\le r\le n$ as displayed in
$$
\xymatrix{
X_{r-1} \ar[d]_-{f_{r-1}} & \Sigma X_{r-2} \ar[d]_-{g_{r,2}} & \Sigma^2X_{r-3} \ar[d]_-{g_{r,3}} &\cdots & \Sigma^{r-2}X_1 \ar[d]_-{g_{r,r-1}}&\\
X_r  \ar[r]^-{j_{r,1}} & C_{r,2}\ar[r]^-{j_{r,2}} & C_{r,3}\ar[r]^-{j_{r,3}} &\cdots\ar[r]^-{j_{r,r-2}} & C_{r,r-1}\ar[r]^-{j_{r,r-1}} & C_{r,r}   
}
$$
with $\Omega_r$ a quasi-structure and $\mathscr{A}_r$ a structure. 
\item[(iii)] $\overline{f_r}:C_{r,r}\to X_{r+1}$ is an extension of $f_r$ to $C_{r,r}$ for $2\le r\le n-1$, 
and $\overline{f_n}:C_{n,n-1}\to X_{n+1}$ is an extension of $f_n$ to $C_{n,n-1}$.
\end{enumerate}
We use the following notations: 
\begin{itemize}
\item $C_{r,0}=\{*\}$, $C_{r,1}=X_r$, $j_{r,0}=*:C_{r,0}\to C_{r,1}$, $f_r^0=f_r\circ j_{r,0}:C_{r,0}\to X_{r+1}$ 
for $1\le r\le n$, and $\overline{f_1}=f_1:C_{1,1}\to X_2$;
\item $g_{r,1}=f_{r-1}$ for $2\le r\le n$; 
\item $\overline{f_r}^s=\begin{cases} f_r & 1=s\le r\le n\\\overline{f_r}\circ j_{r,r-1}\circ\cdots\circ j_{r,s} 
& 0\le s<r\le n-1\\      \overline{f_r} & 1\le s=r\le n-1\end{cases}\ :C_{r,s}\to X_{r+1} $; 
\item $\overline{f_n}^s=\begin{cases} {\overline{f_n}}\circ j_{n,n-2}\circ\cdots\circ j_{n,s} & 0\le s\le n-2\\
\overline{f_n} & s=n-1\end{cases}\  :C_{n,s}\to X_{n+1}$; 
\item $\Omega_r=\{\omega_{r,s}\,|\,1\le s<r\}$ and 
$\omega_{r,0}=1_{X_r}$ for $2\le r\le n$, where $\omega_{r,s}:C_{r,s+1}\cup_{j_{r,s}}CC_{r,s}\to\Sigma\Sigma^{s-1}X_{r-s}$;
\item $\mathscr{A}_r=\{a_{r,s}\,|\,1\le s<r\}$ and $\Omega(\mathscr{A}_r)=\{\omega_{r,s}\,|\,1\le s<r\}$, where 
$a_{r,s}:C_{r,s+1}\to C_{r,s}\cup_{g_{r,s}}C\Sigma^{s-1}X_{r-s}$ and
$$
\omega_{r,s}=q_{g_{r,s}}'\circ (a_{r,s}\cup C1_{C_{r,s}}):C_{r,s+1}\cup_{j_{r,s}}CC_{r,s}\simeq \Sigma\Sigma^{s-1}X_{r-s},
$$
and $a_{r,s}^{-1}$ is a homotopy inverse of $a_{r,s}$ such that 
$$
a_{r,s}^{-1}\circ i_{g_{r,s}}=j_{r,s},\quad a_{r,s}^{-1}\circ a_{r,s}\overset{C_{r,s}}{\simeq}1_{C_{r,s+1}},\quad 
a_{r,s}\circ a_{r,s}^{-1}\overset{C_{r,s}}{\simeq}1_{C_{g_{r,s}}}.
$$ 
\end{itemize} 

\begin{defi'}%6.1.1
Various presentations of $\vec{\bm f}$ and related notions are defined as follows 
(if Problem 5.6 is affirmative, (a) (resp.\,(d)) equals with (a$'$) (resp.\,(d$'$)). 
\begin{enumerate}
\item[(1)] A collection $\{\mathscr{S}_r,\,\overline{f_r},\,\Omega_r\,|\,2\le r\le n\}$ is
\begin{enumerate}
\item[(a)] a {\it $q$-presentation} if $\mathscr{S}_{r+1}=\mathscr{S}_r(\overline{f_r},\Omega_r)$ 
for $2\le r<n$;
\item[(a$'$)] a {\it $qs_2$-presentation} if $\mathscr{S}_2$ is an iterated mapping cone and 
$\mathscr{S}_{r+1}=\mathscr{S}_r(\overline{f_r},\Omega_r)$ for $2\le r<n$;
\item[(b)] a {\it $q\dot{s}_2$-presentation} if it is a $qs_2$-presentation and $\mathscr{S}_2$ is reduced;
\item[(c)] a {\it $q\ddot{s}_2$-presentation} if it is a $q\dot{s}_2$-presentation and $\Omega_2$ is reduced;
\item[(d)] an {\it $aq$-presentation} if $\mathscr{S}_{r+1}=\mathscr{S}_r(\overline{f_r},\Omega_r)$ and 
$\Omega_{r+1}=\widetilde{\Omega_r}$ for $2\le r< n$;
\item[(d$'$)] an {\it $aqs_2$-presentation} if $\mathscr{S}_2$ is an iterated mapping cone and 
$\mathscr{S}_{r+1}=\mathscr{S}_r(\overline{f_r},\Omega_r)$ and $\Omega_{r+1}=\widetilde{\Omega_r}$ for $2\le r< n$;
\item[(e)] an {\it $aq\dot{s}_2$-presentation} if it is an $aqs_2$-presentation and $\mathscr{S}_2$ is reduced; 
\item[(f)] an {\it $aq\ddot{s}_2$-presentation} if it is an $aq\dot{s}_2$-presentation and $\Omega_2$ is reduced.
\end{enumerate}
\item[(2)] A collection 
$\{\mathscr{S}_2,\overline{f_2},\Omega_2\}\cup\{\mathscr{S}_r,\overline{f_r},\mathscr{A}_r\,|\,3\le r\le n\}$ is a 
{\it $q_2$-presentation} if $\mathscr{S}_3=\mathscr{S}_2(\overline{f_2},\Omega_2)$, 
$\mathscr{S}_{r+1}=\mathscr{S}_r(\overline{f_r},\mathscr{A}_r)\ (3\le r<n)$, and $\mathscr{A}_{r}$ is reduced for 
$3\le r\le n$. 
\item[(3)] A collection $\{\mathscr{S}_r,\,\overline{f_r},\,\mathscr{A}_r\,|\,2\le r\le n\}$ is
\begin{enumerate}
\item[(g)] an {\it $s_t$-presentation} if  
$\mathscr{S}_{r+1}=\mathscr{S}_r(\overline{f_r},\mathscr{A}_r)$ and $\mathscr{A}_{r+1}$ is reduced for $2\le r< n$; 
\item[(h)] an {\it $\dot{s}_t$-presentation} if it is an $s_t$-presentation and $\mathscr{S}_2$ is reduced; 
\item[(i)] an {\it $\ddot{s}_t$-presentation} if it is an $\dot{s}_t$-presentation and $\mathscr{A}_2$ is reduced. 
\end{enumerate}
\item[(4)] Let $\star$ denote one of the following: 
$q$, $aq$, $qs_2$, $q\dot{s}_2$, $q\ddot{s}_2$, $aqs_2$, $aq\dot{s}_2$, $aq\ddot{s}_2$, 
$s_t$, $\dot{s}_t$, $\ddot{s}_t$, and $q_2$. 
$\vec{\bm f}$ is {\it $\star$-presentable} if it has a $\star$-presentation, and 
$\vec{\bm \alpha}$ is {\it $\star$-presentable} if it has a $\star$-presentable representative. 
\end{enumerate}
\end{defi'}

In the above definitions we used the following abbreviations: $q$=``quasi-structure''; 
$s_2$=``$\mathscr{S}_2$ is an iterated mapping cone''; $\dot{s}_2$=``$\mathscr{S}_2$ is a reduced iterated mapping cone''; 
$\ddot{s}_2$=``$\mathscr{S}_2$ is a reduced iterated mapping cone with $\Omega_2$ reduced''; 
$a$=``asymptotic''; $s_t$=``structure''; $\dot{s}_t$=``$s_t$ and $\dot{s}_2$''; $\ddot{s}_t$=``$\dot{s}_t$ and $\mathscr{A}_2$ is reduced''. 

\begin{defi'}%6.1.2
We denote the set of homotopy classes of $\overline{f_n}\circ g_{n,n-1}$ 
for all $\star$-presentations of $\vec{\bm f}$ by $\{\vec{\bm f}\,\}^{(\star)}$ or $\{f_n,\dots,f_1\}^{(\star)}$ 
which is called the $\star$-{\it bracket} of $\vec{\bm f}$. 
It is a subset of $[\Sigma^{n-2}X_1,X_{n+1}]$ and there is a possibility that it is the empty set. 
For convenience we denote by $\{f_2,f_1\}^{(\star)}$ the one point set consisting of the homotopy class of $f_2\circ f_1$. 
\end{defi'}

Notice that $\vec{\bm f}$ is $\star$-presentable if and only if 
$\{\vec{\bm f}\,\}^{(\star)}$ is not empty. 
As shall be seen in \S 6.4, we can denote $\{\vec{\bm f}\,\}^{(\star)}$ by $\{\vec{\bm \alpha}\,\}^{(\star)}$ 
for any $\vec{\bm f}\in \mathrm{Rep}(\vec{\bm \alpha})$. 

\begin{rem'}%6.1.3
It follows from definitions that if $\vec{\bm \alpha}$ is $q$-presentable, then $\alpha_{r+1}\circ\alpha_r=0$ 
for $1\le r\le n-1$, and that we have 
the commutative diagram 
\begin{equation}
\begin{split}
\xymatrix{
\{\vec{\bm f}\,\}^{(q\ddot{s}_2)} \ar[r]^-{!!} &\{\vec{\bm f}\,\}^{(q\dot{s}_2)} \ar[drr]^-{!}& &
\{\vec{\bm f}\,\}^{(\ddot{s}_t)} \ar[r]^-{!} & \{\vec{\bm f}\,\}^{(\dot{s}_t)} \ar[d]^-{!}\\
\{\vec{\bm f}\,\}^{(aq\ddot{s}_2)} \ar[r]  \ar[u] & \{\vec{\bm f}\,\}^{(aq\dot{s}_2)}  \ar[r]^-{!} \ar[u] & 
\{\vec{\bm f}\,\}^{(aqs_2)} \ar[d]^-{\#} \ar[r] & \{\vec{\bm f}\,\}^{(qs_2)} \ar[d]^-{\#} & \{\vec{\bm f}\,\}^{(s_t)} \ar[l] \ar[d]\\
& &   \{\vec{\bm f}\,\}^{(aq)} \ar[r] & \{\vec{\bm f}\,\}^{(q)} & \{\vec{\bm f}\,\}^{(q_2)} \ar[l]
}
\end{split}
\end{equation}
where arrows are inclusions, $!!$ is $=$ for $n\ge 4$, and 
four $!$'s are $=$ as shall be shown in Theorem 6.2.1. 
Notice that if Problem 5.6 is affirmative, two $\#$'s are $=$. 
\end{rem'}

The following two propositions are easy consequences of definitions.  

\begin{prop'}%6.1.4
Let $\{\mathscr{S}_r,\overline{f_r},\Omega_r\,|\,2\le r\le n\}$ be a $q$-presentation of $\vec{\bm f}$. 
Then 
\begin{align*}
C_{r,2}&=X_r\cup_{\overline{f_{r-1}}^1}CX_{r-1}\ (3\le r\le n),\quad C_{3,3}=X_3\cup_{\overline{f_2}}CC_{2,2},\\ 
C_{r,s}&=X_r\cup_{\overline{f_{r-1}}^{s-1}}C(X_{r-1}\cup_{\overline{f_{r-2}}^{s-2}}C(X_{r-2}\cup\cdots\\
&\hspace{1cm} \cup_{\overline{f_{r-s+2}}^2}C(X_{r-s+2}\cup_{\overline{f_{r-s+1}}^1}CX_{r-s+1})\cdots))\ (3\le s<r\le n),\\
C_{r,r}&=X_r\cup_{\overline{f_{r-1}}^{r-1}}C(X_{r-1}\cup_{\overline{f_{r-2}}^{r-2}} C( \cdots \cup_{\overline{f_3}^3}
C(X_3\cup_{\overline{f_2}^2}C_{2,2})\cdots ))\\
&\hspace{9cm} (4\le r\le n).
\end{align*}
\end{prop'}

\begin{prop'}%6.1.5
If $\{\vec{\bm f}\,\}^{(\star)}$ is not empty, then $\{f_m,f_{m-1},\dots,f_\ell\}^{(\star)}$ contains $0$ for $1\le\ell<m\le n,\ 
(\ell,m)\ne(1,n)$. 
\end{prop'}

\begin{defi'}%6.1.6
If there exist null-homotopies $A_i:f_{i+1}\circ f_i\simeq *\ (1\le i\le n-1)$ such that 
$[f_{i+2},A_{i+1},f_{i+1}]\circ(f_{i+1},A_i,f_i)\simeq *\ (1\le i\le n-2)$, then we call 
$\vec{\bm f}$ and $(\vec{\bm f};\vec{\bm A})$ {\it admissible}, where $\vec{\bm A}=(A_{n-1},\dots,A_1)$. 
We call $\vec{\bm \alpha}$ {\it admissible} if it has an admissible representative. 
\end{defi'}

It follows from Proposition 2.11 of \cite{OO1} that if $\vec{\bm \alpha}$ is admissible, 
then every representative of it is admissible. 
From results in forthcoming sub-sections, we can prove the following without difficulties: 
when $n=3$, $\vec{\bm \alpha}$ is admissible if and only if 
$\{\vec{\bm \alpha}\,\}^{(\star)}$ contains $0$ for all $\star$; 
when $n=4$, $\vec{\bm \alpha}$ is admissible if and only if $\vec{\bm \alpha}$ 
is $\star$-presentable for all $\star$; 
when $n\ge 5$, if $\vec{\bm \alpha}$ is $\star$-presentable for some $\star$, then $\vec{\bm \alpha}$ is admissible.  

\begin{rem'} The following is obvious by definitions: 
if $f_i:X_i\to X_{i+1}$ is a map in $\mathrm{TOP}^{clw}$ for every $i$, 
then the $\star$-brackets of $\vec{\bm f}$ in $\mathrm{TOP}^{clw}$ 
and $\mathrm{TOP}^w$ are the same for 
$\star=q\dot{s}_2, q\ddot{s}_2, aq\dot{s}_2, aq\ddot{s}_2, \dot{s}_t, \ddot{s}_t$. 
\end{rem'}

\subsection{Relations between twelve brackets}%subsection 6.2

In this subsection we prove three results and state an example. 
From 6.2.1, 6.2.2, and (6.1.1), we have (1.1)--(1.4). 

\begin{thm'}%Theorem 6.2.1
\begin{enumerate}
\item[\rm(1)] $\{\vec{\bm f}\,\}^{(\ddot{s}_t)}=\{\vec{\bm f}\,\}^{(\dot{s}_t)}
=\{\vec{\bm f}\,\}^{(s_t)}$. 
\item[\rm(2)] $\{\vec{\bm f}\,\}^{(aqs_2)}=\{\vec{\bm f}\,\}^{(aq\dot{s}_2)}
=\{\vec{\bm f}\,\}^{(aq\ddot{s}_2)}\circ \Sigma^{n-3}\mathscr{E}(\Sigma X_1)$, 
$\{\vec{\bm f}\,\}^{(qs_2)}=\{\vec{\bm f}\,\}^{(q\dot{s}_2)}
=\{\vec{\bm f}\,\}^{(aq\ddot{s}_2)}\circ\mathscr{E}(\Sigma^{n-2}X_1)$, and 
$$
\{\vec{\bm f}\,\}^{(q\ddot{s}_2)}=\begin{cases} \{\vec{\bm f}\,\}^{(aq\ddot{s}_2)} & n=3\\
\{\vec{\bm f}\,\}^{(aq\ddot{s}_2)}\circ\mathscr{E}(\Sigma^{n-2}X_1) & n\ge 4\end{cases}.
$$ 
\item[\rm(3)] $\{\vec{\bm f}\,\}^{(q)}=\{\vec{\bm f}\,\}^{(aq)}\circ\mathscr{E}(\Sigma^{n-2}X_1)$. 
\item[\rm(4)] $\{\vec{\bm f}\,\}^{(\ddot{s}_t)}\circ\mathscr{E}(\Sigma^{n-2}X_1)
=\{\vec{\bm f}\,\}^{(aq\ddot{s}_2)}\circ\mathscr{E}(\Sigma^{n-2}X_1)$. 
\item[\rm(5)] If $\alpha\in\{\vec{\bm f}\,\}^{(q)}$, then there are $\theta,\theta'\in[\Sigma^{n-2}X_1, \Sigma^{n-2}X_1]$ 
such that $\alpha\circ \theta\in\{\vec{\bm f}\,\}^{(aq\ddot{s}_2)}$ and 
$\alpha\circ \theta'\in\{\vec{\bm f}\,\}^{(\ddot{s}_t)}$. 
\end{enumerate}
\end{thm'}

\begin{cor'}%Corollary 6.2.2
\begin{enumerate}
\item[\rm(1)] $\{\vec{\bm f}\,\}^{(q)}\circ\varepsilon=\{\vec{\bm f}\,\}^{(q)}$ and 
$\{\vec{\bm f}\,\}^{(qs_2)}\circ\varepsilon=\{\vec{\bm f}\,\}^{(qs_2)}$ for every 
$\varepsilon\in\mathscr{E}(\Sigma^{n-2}X_1)$, and 
$\{\vec{\bm f}\,\}^{(aqs_2)}\circ \Sigma^{n-3}\gamma=\{\vec{\bm f}\,\}^{(aqs_2)}$ 
for every $\gamma\in\mathscr{E}(\Sigma X_1)$.
\item[\rm(2)] $\{\vec{\bm f}\,\}^{(aq)}\circ \Sigma^{n-2}\gamma=\{\vec{\bm f}\,\}^{(aq)}$ 
for every $\gamma\in\mathscr{E}(X_1)$, 
and $-\{\vec{\bm f}\,\}^{(aq)}=\{\vec{\bm f}\,\}^{(aq)}$.
\item[\rm(3)] If the suspension $\Sigma^{n-2}:\mathscr{E}(X_1)\to\mathscr{E}(\Sigma^{n-2}X_1)$ 
is surjective, for example if $X_1$ is a sphere of positive dimension, then 
$\{\vec{\bm f}\,\}^{(q)}=\{\vec{\bm f}\,\}^{(aq)}$. 
\item[\rm(4)] If $\{\vec{\bm f}\,\}^{(\star)}$ is not empty for some $\star$, then 
$\{\vec{\bm f}\,\}^{(\star)}$ is not empty for all $\star$. 
\item[\rm(5)] If $\{\vec{\bm f}\,\}^{(\star)}$ contains $0$ for some $\star$, then 
$\{\vec{\bm f}\,\}^{(\star)}$ contains $0$ for all $\star$. 
\item[\rm(6)] $-\{\vec{\bm f}\,\}^{(\star)}=\{\vec{\bm f}\,\}^{(\star)}$ for $\star=q,\, qs_2,\, aq,\,aqs_2$.
\item[\rm(7)] If $n\ge 4$ and $\vec{\bm f}$ is $\star$-presentable for some $\star$, 
then $\vec{\bm f}$ is admissible and $\star$-presentable for all $\star$. 
\item[\rm(8)] If $\{\vec{\bm f}\,\}^{(\star)}=\{0\}$ for some $\star$, 
then $\{\vec{\bm f}\,\}=\{0\}$ for all $\star$ except $aq, q, q_2$. 
\end{enumerate}
\end{cor'}

\begin{prop'}[cf.\,p.26, p.25, and p.33 of \cite{W2}]%Proposition 6.2.3
Given maps $f_{n+1}:X_{n+1}\to X_{n+2}$ and $f_0:X_0\to X_1$, we have 
\begin{gather}
f_{n+1}\circ\{f_n,\dots,f_1\}^{(\star)}\subset\{f_{n+1}\circ f_n,f_{n-1},\dots,f_1\}^{(\star)},\\
\{f_{n+1}\circ f_n,f_{n-1},\dots,f_1\}^{(\star)}\subset
\{f_{n+1},f_n\circ f_{n-1},f_{n-2},\dots,f_1\}^{(\star)},\\
\begin{split}
\{f_n,\dots,f_1\}^{(aq\ddot{s}_2)}\circ \Sigma^{n-2}f_0\subset \{f_n,&\dots,f_2,f_1\circ f_0\}^{(aq\ddot{s}_2)}\\
&\subset\{f_n,\dots,f_3,f_2\circ f_1,f_0\}^{(aq\ddot{s}_2)}.
\end{split}
\end{gather}
\end{prop'}

\begin{rem'}
We can prove the following analogues relations of (6.2.3). 
\begin{align*}
\{f_n,\dots,f_1\}^{(\ddot{s}_t)}\circ \Sigma^{n-2}f_0\subset \{f_n,&\dots,f_2,f_1\circ f_0\}^{(\ddot{s}_t)}\\
&\subset\{f_n,\dots,f_3,f_2\circ f_1,f_0\}^{(\ddot{s}_t)}.
\end{align*}
Details shall be written elsewhere.
\end{rem'}

\begin{exam'}[cf.\,Lemma 4.10 of \cite{T3}, Lemma 5.1 of \cite{G}, Example 6.6.2(2) below]%Example 6.2.5
Let $p$ be an odd prime and 
$\alpha_1(3):\s^{2p}\to\s^3$ a map of which the homotopy class is of order $p$. 
For every integer $n\ge 3$, we set $\alpha_1(n)=\Sigma^{n-3}\alpha_1(3):\s^{n+2p-3}\to \s^n$ and 
$\Xi=\{\alpha_1(n),\alpha_1(n+2p-3),\alpha_1(n+2(2p-3)),\dots,\alpha_1(n+(p-1)(2p-3))\}^{(\star)}\subset\pi_{n+2p(p-1)-2}(\s^n)$. 
Take $n$ such that $n\ge 2p(p-1)$. 
Then the $p$-primary component of $\pi_{n+2p(p-1)-2}(\s^n)$ is $\bZ_p$, and the following can be proved: 
$\Xi$ contains an element of order $p$ and the order of any element of $\Xi$ is a multiple of $p$ so that 
$\Xi$ does not contain $0$. 
We need an argument for the proof, but we omit details.
\end{exam'}

\begin{proof}[Proof of Theorem 6.2.1(1)] 
It suffices to show that $\{\vec{\bm f}\,\}^{(s_t)}\subset \{\vec{\bm f}\,\}^{(\ddot{s}_t)}$. 
Let $\alpha\in\{\vec{\bm f}\,\}^{(s_t)}$ and $\{\mathscr{S}_r,\overline{f_r},\mathscr{A}_r\,|\,2\le r\le n\}$ an $s_t$-presentation 
of $\vec{\bm f}$ with $\alpha=\overline{f_n}\circ g_{n,n-1}$, where 
\begin{align*}
&\mathscr{S}_r=(X_{r-1},\Sigma X_{r-2},\dots,\Sigma^{r-2}X_1;C_{r,1},\dots,C_{r,r};\\
&\hspace{5cm} g_{r,1},\dots,g_{r,r-1};j_{r,1},\dots,j_{r,r-1});\\
&C_{r,1}=X_r,\ g_{r,1}=f_{r-1},\ \mathscr{A}_r=\{a_{r,s}\,|\,1\le s<r\},\  \Omega(\mathscr{A}_r)=\{\omega_{r,s}\,|\,1\le s<r\};\\
&\text{if $3\le r\le n$, then $C_{r,2}=X_r\cup_{f_{r-1}}CX_{r-1},\ j_{r,1}=i_{f_{r-1}}$, and $a_{r,1}=1_{C_{r,2}}$}. 
\end{align*}
We are going to construct an $\ddot{s}_t$-presentation $\{\mathscr{S}_r',\overline{f_r}',\mathscr{A}_r'\,|\,2\le r\le n\}$ of 
$\vec{\bm f}$ such that $\overline{f_n}'\circ g_{n,n-1}'=\alpha$. 

First we set $\mathscr{S}_2'=(X_1;X_2,X_2\cup_{f_1}CX_1;f_1;i_{f_1})$, $a_{2,1}'=1_{C_{2,2}'}$, $\mathscr{A}_2'=\{a_{2,1}'\}$, 
$\Omega_2'=\Omega(\mathscr{A}_2')=\{q_{f_1}'\}$, $e_{2}=a_{2,1}^{-1}:C_{2,2}'\to C_{2,2}$, and 
$\overline{f_2}'=\overline{f_2}\circ e_{2}$. 
Then $C_{2,1}'=C_{2,1}$, $e_{2}\circ j_{2,1}'=j_{2,1}$ and 
\begin{align*}
\omega_{2,1}\circ(e_2\cup C1_{X_2})&=q_{f_1}'\circ (a_{2,1}\cup C1_{X_2})\circ(e_2\cup C1_{X_2})\\
&\simeq q_{f_1}'=\omega_{2,1}'\quad(\text{by (4.1)}). 
\end{align*}

Secondly we set $\mathscr{S}_3'=\mathscr{S}_2'(\overline{f_2}',\mathscr{A}_2')$ and
$$
e_{3}=1_{X_3}\cup Ce_{2}:C_{3,3}'=X_3\cup_{\overline{f_2}'}C(X_2\cup_{f_1}CX_1)\to C_{3,3}=X_3\cup_{\overline{f_2}}CC_{2,2}.
$$
Then $C_{3,s}'=C_{3,s}\,(s=1,2),\  j_{3,1}'=j_{3,1},\  j_{3,2}'=1_{X_3}\cup Cj_{2,1}',\  e_{3}\circ j_{3,2}'=j_{3,2}$, 
and 
\begin{align*}
g_{3,2}'&=(\overline{f_2}'\cup C1_{X_2})\circ \omega_{2,1}'^{-1} 
=(\overline{f_2}\cup C1_{X_2})\circ(e_{2}\cup C1_{X_2})\circ \omega_{2,1}'^{-1}\\
&\simeq 
(\overline{f_2}\cup C1_{X_2})\circ \omega_{2,1}^{-1}=g_{3,2}.
\end{align*}
Take a homotopy $K^3:g_{3,2}\simeq g_{3,2}'$ and set 
\begin{align*}
&\Phi(K^3)=\Phi(g_{3,2},g_{3,2}',1_{\Sigma X_1},1_{X_3\cup_{f_2}CX_2};K^3) \\
&\hspace{2cm}: (X_3\cup_{f_2}CX_2)\cup_{g_{3,2}}C\Sigma X_1\to (X_3\cup_{f_2}CX_2)\cup_{g_{3,2}'}C\Sigma X_1,\\
&a_{3,2}'=\Phi(K^3)\circ a_{3,2}\circ e_{3}:C_{3,3}'\to (X_3\cup_{f_2}CX_2)\cup_{g_{3,2}'}C\Sigma X_1,\\
&a_{3,1}'=1_{C_{3,2}'},\ \mathscr{A}_3'=\{a_{3,1}', a_{3,2}'\},\ 
\overline{f_3}'=\begin{cases} \overline{f_3}:C_{3,2}'=C_{3,2}\to X_4 & n=3\\ \overline{f_3}\circ e_{3}:C_{3,3}'\to X_4 & n\ge 4\end{cases}.
\end{align*}
Then $\mathscr{A}_3'$ is a reduced structure on $\mathscr{S}_3'$. 
When $n=3$, $\{\mathscr{S}_r',\overline{f_r}',\mathscr{A}_r'\,|\, r=2,3\}$ is an $\ddot{s}_t$-presentation of $\vec{\bm f}$ 
such that $\overline{f_3}'\circ g_{3,2}'=\alpha$. 
When $n\ge 4$, by repeating the above process, we have 
$\mathscr{S}_r',\overline{f_r}',\mathscr{A}_r'$ and $e_{r}:C_{r,r}'\simeq C_{r,r}$ for $4\le r\le n$ such that 
$$
\left\{\begin{array}{@{\hspace{0.6mm}}l}
\mathscr{S}_r'=\mathscr{S}_{r-1}'(\overline{f_{r-1}}',\mathscr{A}_{r-1}'),\ e_r=1_{X_r}\cup Ce_{r-1};\\
C_{r,s}'=C_{r,s}\ (1\le s\le r-1),\  C_{r,r}'=X_r\cup_{\overline{f_{r-1}}'}CC_{r-1,r-1}';\\
j_{r,s}'=j_{r,s},\ a_{r,s}'=a_{r,s},\ g_{r,s}'=g_{r,s}\ (1\le s\le r-2);\\
\omega_{r,r-1}'\simeq \omega_{r,r-1}\circ (e_{r}\cup C1_{C_{r,r-1}}),\ g_{r,r-1}'\simeq g_{r,r-1};\\
\overline{f_r}'=\begin{cases} \overline{f_n}:C_{n,n-1}'=C_{n,n-1}\to X_{n+1} & r=n\\ \overline{f_r}\circ e_{r}:C_{r,r}'\to X_{r+1} & r<n\end{cases};\\
a_{r,r-1}'=\Phi(K^r)\circ a_{r,r-1}\circ e_r:C_{r,r}'\to C_{r,r-1}'\cup_{g_{r,r-1}'}C\Sigma^{r-2}X_1,
\end{array}\right.
$$
where $K^r:g_{r,r-1}\simeq g_{r,r-1}'$ and 
\begin{align*}
&\Phi(K^r)=\Phi(g_{r,r-1},g_{r,r-1}',1_{\Sigma^{r-2}X_1},1_{C_{r,r-1}};K^r)\\
&\hspace{2cm}:C_{r,r-1}\cup_{g_{r,r-1}}C\Sigma^{r-2}X_1\to C_{r,r-1}'\cup_{g_{r,r-1}'}C\Sigma^{r-2}X_1.
\end{align*}
Then $\mathscr{A}_r'$ is a reduced structure on $\mathscr{S}_r'$. 
Therefore $\{\mathscr{S}_r',\overline{f_r}',\mathscr{A}_r'\,|\,2\le r\le n\}$ is an $\ddot{s}_t$-presentation of 
$\vec{\bm f}$ such that $\overline{f_n}'\circ g_{n,n-1}'\simeq \overline{f_n}\circ g_{n,n-1}$ and hence 
$\alpha\in\{\vec{\bm f}\,\}^{(\ddot{s}_t)}$. 
This proves Theorem~6.2.1(1). 
\end{proof}

\begin{proof}[Proof of Theorem 6.2.1(2)] 
First we prove $\{\vec{\bm f}\,\}^{(aqs_2)}\subset\{\vec{\bm f}\,\}^{(aq\dot{s}_2)}$ 
which is equivalent to the first equality. 
Let $\alpha\in\{\vec{\bm f}\,\}^{(aqs_2)}$ 
and $\{\mathscr{S}_r,\overline{f_r},\Omega_r\,|\,2\le r\le n\}$ an $aqs_2$-presentation of 
$\vec{\bm f}$ with $\alpha=\overline{f_n}\circ g_{n,n-1}$. 
It suffices to construct an $aq\dot{s}_2$-presentation $\{\mathscr{S}_r',\overline{f_r}',\Omega_r'\,|\,2\le r\le n\}$ 
with $\alpha=\overline{f_n}'\circ g_{n,n-1}'$. 
Set $\mathscr{S}_2'=(X_1;X_2,X_2\cup_{f_1}CX_1;f_1;i_{f_1})$. 
Since $\mathscr{S}_2$ is a shft, we can take $e_2:C_{2,2}'=X_2\cup_{f_1}CX_1\simeq C_{2,2}$ such that 
$e_2\circ j_{2,1}'=j_{2,1}$. 
Set $\overline{f_2}'=\overline{f_2}\circ e_2$, 
$\omega_{2,1}'=\omega_{2,1}\circ (e_2\cup C1_{X_2}):C_{2,2}'\cup_{j_{2,1}'}CC_{2,1}'\to \Sigma X_1$, and 
$\Omega_2'=\{\omega_{2,1}'\}$. 
Set $\mathscr{S}_3'=\mathscr{S}_2'(\overline{f_2}',\Omega_2')$, $\Omega_3'=\widetilde{\Omega_2'}$, 
$e_3=1_{X_3}\cup Ce_2:C_{3,3}'\to C_{3,3}$, 
and $\overline{f_3}'=\begin{cases} \overline{f_3}:C_{3,2}'=C_{3,2}\to X_4 & n=3\\ 
\overline{f_3}\circ e_3:C_{3,3}'\to X_4 & n\ge 4\end{cases}$. 
Then $\omega_{3,2}'=\omega_{3,2}\circ(e_3\cup C1_{C_{3,2}})$ and 
$g_{3,2}'=(\overline{f_2}'\cup C1_{X_2})\circ \omega_{2,1}'^{-1}\simeq (\overline{f_2}\cup C1_{X_2})\circ
\omega_{2,1}^{-1}=g_{3,2}$. 
By continuing the construction inductively, we obtain an $aq\dot{s}_2$-presentation 
$\{\mathscr{S}_r',\overline{f_r}',\Omega_r'\,|\,2\le r\le n\}$ and 
$e_r: C_{r,r}'\simeq C_{r,r}$ such that $C_{r,s}'=C_{r,s}\ (1\le s<r<n)$, 
$\omega_{r,r-1}'=\omega_{r,r-1}\circ (e_r\cup 1_{C_{r,r-1}}):C_{r,r}'\cup CC_{r,r-1}'\to C_{r,r}\cup CC_{r,r-1}$ 
for $r< n$, and $\overline{f_r}'=\begin{cases} \overline{f_n} : C_{n,n-1}'=C_{n,n-1}\to X_{n+1} & r=n\\ 
\overline{f_r}\circ e_r:C_{r,r}'\to X_{r+1} & r<n\end{cases}$ 
so that $g_{n,n-1}'=(\overline{f_{n-1}}'\cup C1_{n-1,n-2})\circ\omega_{n-1,n-2}'^{-1}\simeq g_{n,n-1}$. 
Hence $\overline{f_n}'\circ g_{n,n-1}'\simeq \overline{f_n}\circ g_{n,n-1}$. 
This proves the first equality in (2). 

Secondly we prove the second equality in (2). 
Let $\{\mathscr{S}_r,\overline{f_r},\Omega_r\,|\,2\le r\le n\}$ be an $aq\dot{s}_2$-presentation of $\vec{\bm f}$ and set 
$\alpha=\overline{f_n}\circ g_{n,n-1}$. 
Set $\mathscr{S}_2'=\mathscr{S}_2$, $\omega_{2,1}'=q_{f_1}'$, $\Omega_2'=\{\omega_{2,1}'\}$, 
and $\theta=\omega_{2,1}\circ\omega_{2,1}'^{-1}\in\mathscr{E}(\Sigma X_1)$. 
By Remark 5.5(3), we define inductively 
$\mathscr{S}_3'=\mathscr{S}_2'(\overline{f_2},\Omega_2'),\ 
\Omega_3'=\widetilde{\Omega_2'};\ \dots\ ;\ \mathscr{S}_n'=\mathscr{S}_{n-1}'(\overline{f_{n-1}},\Omega_{n-1}'),\ 
  \Omega_n'=\widetilde{\Omega_{n-1}'}$. 
Then $\{\mathscr{S}_r',\overline{f_r},\Omega_r'\,|\,2\le r\le n\}$ is an $aq\ddot{s}_2$-presentation of $\vec{\bm f}$ 
such that $\omega_{r,s}'=\omega_{r,s}$ for $1\le s\le r-2$, and 
$\Sigma^{r-2}\theta\circ \omega_{r,r-1}'=\omega_{r,r-1}$.  
Hence $\alpha\circ \Sigma^{n-3}\theta=\overline{f_n}\circ g_{n,n-1}'\in\{\vec{\bm f}\,\}^{(aq\ddot{s}_2)}$ and so 
$\alpha\in \{\vec{\bm f}\,\}^{(aq\ddot{s}_2)}\circ \Sigma^{n-3}\theta^{-1}\subset
\{\vec{\bm f}\,\}^{(aq\ddot{s}_2)}\circ \Sigma^{n-3}\mathscr{E}(\Sigma X_1)$. 
Thus 
$\{\vec{\bm f}\,\}^{(aq\dot{s}_2)}\subset \{\vec{\bm f}\,\}^{(aq\ddot{s}_2)}\circ \Sigma^{n-3}\mathscr{E}(\Sigma X_1)$. 

Conversely let $\alpha\in\{\vec{\bm f}\,\}^{(aq\ddot{s}_2)}$ and $\theta\in \mathscr{E}(\Sigma X_1)$. 
Let $\{\mathscr{S}_r',\overline{f_r}',\Omega_r'\,|\,2\le r\le n\}$ be an $aq\ddot{s}_2$-presentation of $\vec{\bm f}$ 
with $\alpha=\overline{f_n}'\circ g_{n,n-1}'$. 
Let $\mathscr{S}_r$ be the iterated mapping cone which is obtained from $\mathscr{S}_r'$ by replacing $g_{r,r-1}'$ with 
$g_{r,r-1}'\circ\Sigma^{r-3}\theta$, and $\Omega_r$ the quasi-structure on $\mathscr{S}_r$ which is obtained 
from $\Omega_r'$ by replacing $\omega_{r,r-1}'$ with $\Sigma^{r-2}\theta^{-1}\circ\omega_{r,r-1}'$. 
Then $\{\mathscr{S}_r,\overline{f_r}',\Omega_r\,|\,2\le r\le n\}$ is an $aq\dot{s}_2$-presentation of $\vec{\bm f}$ 
such that $g_{r,s}=g_{r,s}'$ for $1\le s\le r-2$ and $g_{r,r-1}=g_{r,r-1}'\circ \Sigma^{r-3}\theta$.  
Hence $\alpha\circ \Sigma^{n-3}\theta=\overline{f_n}'\circ g_{n,n-1}'\circ \Sigma^{n-3}\theta
=\overline{f_n}'\circ g_{n,n-1}\in\{\vec{\bm f}\,\}^{(aq\dot{s}_2)}$. 
Thus 
$\{\vec{\bm f}\,\}^{(aq\dot{s}_2)}\supset \{\vec{\bm f}\,\}^{(aq\ddot{s}_2)}\circ \Sigma^{n-3}\mathscr{E}(\Sigma X_1)$. 
This proves the second equality in (2). 

Thirdly we prove the third and fourth equalities in (2). 
We prove 
\begin{gather}
\{\vec{\bm f}\,\}^{(qs_2)}\subset\{\vec{\bm f}\,\}^{(aq\ddot{s}_2)}\circ\mathscr{E}(\Sigma^{n-2}X_1),\\
\{\vec{\bm f}\,\}^{(aq\ddot{s}_2)}\circ\mathscr{E}(\Sigma^{n-2}X_1)\subset \begin{cases} \{\vec{\bm f}\,\}^{(q\dot{s}_2)} &n=3\\
\{\vec{\bm f}\,\}^{(q\ddot{s}_2)} & n\ge 4\end{cases}.
\end{gather}
If these are proved, then 
\begin{equation}
\begin{split}
\{\vec{\bm f}\,\}^{(qs_2)}&\subset\{\vec{\bm f}\,\}^{(aq\ddot{s}_2)}\circ\mathscr{E}(\Sigma^{n-2}X_1)\\
&\subset
\begin{cases}
\{\vec{\bm f}\,\}^{(q\dot{s}_2)}\subset\{\vec{\bm f}\,\}^{(qs_2)} & n=3\\
\{\vec{\bm f}\,\}^{(q\ddot{s}_2)}\subset\{\vec{\bm f}\,\}^{(q\dot{s}_2)}\subset\{\vec{\bm f}\,\}^{(qs_2)} & n\ge 4
\end{cases}
\end{split}
\end{equation}
so that the third and fourth equalities in (2) follow. 
To prove (6.2.4), let $\alpha\in\{\vec{\bm f}\,\}^{(qs_2)}$ and $\{\mathscr{S}_r,\overline{f_r},\Omega_r\,|\,2\le r\le n\}$ a 
$qs_2$-presentation of $\vec{\bm f}$ with $\alpha=\overline{f_n}\circ g_{n,n-1}$. 
Set 
$$
\mathscr{S}_2'=(X_1;X_2,X_2\cup_{f_1}CX_1;f_1;i_{f_1}),\quad j_{2,1}'=i_{f_1},\quad \omega_{2,1}'=q_{f_1}',\quad \Omega_2'=\{\omega_{2,1}'\}.
$$ 
Since $j_{2,1}$ is a homotopy cofibre of $f_1$ by the hypothesis, there exists a homotopy equivalence 
$e_2:C_{2,2}'=X_2\cup_{f_1}CX_1\to C_{2,2}$ such that $e_2\circ j_{2,1}'=j_{2,1}$. 
Set $\overline{f_2}'=\overline{f_2}\circ e_2$. 
Then $\overline{f_2}'$ is an extension of $f_2$ to $C_{2,2}'$. 
Set 
\begin{gather*}
\mathscr{S}_3'=\mathscr{S}_2'(\overline{f_2}',\Omega_2'),\ \Omega_3'=\widetilde{\Omega_2'},\ 
e_3=1_{X_3}\cup Ce_2:C_{3,3}'\to C_{3,3},\\
\overline{f_3}'=\begin{cases} \overline{f_3}:C_{3,2}'=C_{3,2}\to X_4 & n=3\\ \overline{f_3}\circ e_3:C_{3,3}'\to X_4 & n\ge 4\end{cases}.
\end{gather*}
Proceeding with the construction, we have an $aq\ddot{s}_2$-presentation 
$\{\mathscr{S}_r',\overline{f_r}',\Omega_r'\,|\,2\le r\le n\}$ of $\vec{\bm f}$ such that 
\begin{gather*}
C_{r,s}'=C_{r,s}\ (1\le s\le r-1),\ j_{r,s}'=j_{r,s}\ (1\le s\le r-2),\\
e_r=1_{X_r}\cup Ce_{r-1}:C_{r,r}'\to C_{r,r}\ (3\le r\le n),\\
\overline{f_r}'=\begin{cases} \overline{f_r}\circ e_r : C_{r,r}'\to X_{r+1} & r<n\\ \overline{f_n} : C_{n,n-1}'=C_{n,n-1}\to X_{n+1} & r=n\end{cases}.
\end{gather*} 
Set $\theta=\omega_{n-1,n-2}\circ(e_{n-1}\cup C1_{C_{n-1,n-2}})\circ \omega_{n-1,n-2}'^{-1} : \Sigma^{n-2}X_1\to \Sigma^{n-2}X_1$. 
Then $\theta$ is a homotopy equivalence and
\begin{align*}
\alpha&=\overline{f_n}\circ g_{n,n-1}=\overline{f_n}\circ (\overline{f_{n-1}}\cup C1_{C_{n-1,n-2}})\circ\omega_{n-1,n-2}^{-1}\\
&=\overline{f_n}\circ(\overline{f_{n-1}}\cup C1_{C_{n-1,n-2}})\circ (e_{n-1}\cup C1_{C_{n-1,n-2}})\circ\omega_{n-1,n-2}'^{-1}\circ\theta^{-1}\\
&=\overline{f_n}'\circ (\overline{f_{n-1}}'\cup C1_{n-1,n-2})\circ\omega_{n-1,n-2}'^{-1}\circ\theta^{-1}\\
&\in\{\vec{\bm f}\,\}^{(aq\ddot{s}_2)}\circ\theta^{-1}\subset\{\vec{\bm f}\,\}^{(aq\ddot{s}_2)}\circ\mathscr{E}(\Sigma^{n-2}X_1).
\end{align*} 
This proves (6.2.4). 

To prove (6.2.5), let $\alpha\in\{\vec{\bm f}\,\}^{(aq\ddot{s}_2)}$, 
$\varepsilon\in\mathscr{E}(\Sigma^{n-2}X_1)$, and 
$\{\mathscr{S}_r,\overline{f_r},\Omega_r|2\le r\le n\}$ 
an $aq\ddot{s}_2$-presentation of $\vec{\bm f}$ such that 
$\alpha=\overline{f_n}\circ g_{n,n-1}$. 
Let $\{\mathscr{S}_r',\overline{f_r}',\Omega_r'\,|\,
2\le r\le n\}$ be obtained from 
$\{\mathscr{S}_r,\overline{f_r},\Omega_r|2\le r\le n\}$ by 
replacing $\omega_{n-1,n-2}$ and $\mathscr{S}_n$ with $\varepsilon^{-1}\circ\omega_{n-1,n-2}$ and 
$\mathscr{S}_n'=\mathscr{S}_{n-1}(\overline{f_{n-1}},\Omega_{n-1}')$, respectively. 
Then $\{\mathscr{S}_r',\overline{f_r}',\Omega_r'\,|\,2\le r\le n\}$ is 
a $q\dot{s}_2$-presentation of $\vec{\bm f}$ if $n=3$ and 
a $q\ddot{s}_2$-presentation of $\vec{\bm f}$ if $n\ge 4$, 
and $g_{n,n-1}'=g_{n,n-1}\circ\varepsilon$ for $n\ge 3$.  
Hence 
$\alpha\circ\varepsilon=\overline{f_n}\circ g_{n,n-1}\circ\varepsilon=\overline{f_n}'\circ g_{n,n-1}'\in
\begin{cases} \{\vec{\bm f}\,\}^{(q\dot{s}_2)} & n=3\\
\{\vec{\bm f}\,\}^{(q\ddot{s}_2)} & n\ge 4\end{cases}.
$  
This proves (6.2.5). 

Fourthly it follows from (6.2.6) that 
$$
\{\vec{\bm f}\,\}^{(q\dot{s}_2)}=\begin{cases} \{\vec{\bm f}\,\}^{(aq\ddot{s}_2)}\circ\mathscr{E}(\Sigma X_1) & n=3 \\
\{\vec{\bm f}\,\}^{(q\ddot{s}_2)} & n\ge 4\end{cases}.
$$ 
By definition, we have $\{\vec{\bm f}\,\}^{(aq\ddot{s}_2)}=\{\vec{\bm f}\,\}^{(q\ddot{s}_2)}$ for $n=3$. 
Hence we obtain the fifth equality in (2). 
\end{proof}

\begin{proof}[Proof of Theorem 6.2.1(3)] 
First we prove $\{\vec{\bm f}\,\}^{(q)}\subset \{\vec{\bm f}\,\}^{(aq)}\circ\mathscr{E}(\Sigma^{n-2}X_1)$. 
Let $\alpha\in \{\vec{\bm f}\,\}^{(q)}$ and $\{\mathscr{S}_r,\overline{f_r},\Omega_r\,|\,2\le r\le n\}$ 
a $q$-presentation of $\vec{\bm f}$ with $\alpha=\overline{f_n}\circ g_{n,n-1}$. 
We define inductively
$$
\mathscr{S}_2'=\mathscr{S}_2,\, \Omega_2'=\Omega_2;\ \mathscr{S}_{k+1}'
=\mathscr{S}_k'(\overline{f_k},\Omega_k'),\, \Omega_{k+1}'=\widetilde{\Omega_k'}\ (2\le k<n).
$$
By Remark 5.5(3), this definition is possible and $\mathscr{S}_r', \mathscr{S}_r$ have the same edge. 
Then $\{\mathscr{S}_r',\overline{f_r},\Omega_r'\,|\,2\le r\le n\}$ is an $aq$-presentation of $\vec{\bm f}$ and 
\begin{align*}
\{\vec{\bm f}\}^{(aq)}\ni \overline{f_n}\circ g_{n,n-1}'
&=\overline{f_n}\circ(\overline{f_{n-1}}\cup C1_{C_{n-1,n-2}})\circ\omega_{n-1,n-2}'^{-1}\\
&=\overline{f_n}\circ(\overline{f_{n-1}}\cup C1_{C_{n-1,n-2}})\circ\omega_{n-1,n-2}^{-1}\circ\varepsilon_0
=\alpha\circ\varepsilon_0
\end{align*}
where $\varepsilon_0=\omega_{n-1,n-2}\circ\omega_{n-1,n-2}'^{-1}\in\mathscr{E}(\Sigma^{n-2}X_1)$. 
Hence 
$$
\alpha\in\{\vec{\bm f}\,\}^{(aq)}\circ\varepsilon_0^{-1}\subset\{\vec{\bm f}\,\}^{(aq)}\circ\mathscr{E}(\Sigma^{n-2}X_1)
$$
and so $\{ \vec{\bm f}\,\}^{(q)}\subset \{\vec{\bm f}\,\}^{(aq)}\circ\mathscr{E}(\Sigma^{n-2}X_1)$.  

Secondly we prove $\{ \vec{\bm f}\,\}^{(q)}\supset \{\vec{\bm f}\,\}^{(aq)}\circ\mathscr{E}(\Sigma^{n-2}X_1)$. 
Let $\alpha\in\{\vec{\bm f}\,\}^{(aq)}$, $\varepsilon\in\mathscr{E}(\Sigma^{n-2}X_1)$, and  
$\{\mathscr{S}_r,\overline{f_r},\Omega_r\,|\,2\le r\le n\}$ an $aq$-presentation of $\vec{\bm f}$ 
with $\alpha=\overline{f_n}\circ g_{n,n-1}$. 
We set 
$$
\Omega_{n-1}'=\begin{cases} \{\varepsilon^{-1}\circ \omega_{2,1}\} & n=3\\ 
\{\omega_{n-1,s},\,\varepsilon^{-1}\circ\omega_{n-1,n-2}\,|\,1\le s\le n-3\} & n\ge 4\end{cases}
$$ 
which is a quasi-structure on $\mathscr{S}_{n-1}$. 
Set $\mathscr{S}_n'=\mathscr{S}_{n-1}(\overline{f_{n-1}},\Omega_{n-1}')$. 
Since $\mathscr{S}_n'$ is obtained from $\mathscr{S}_n$ by replacing $g_{n,n-1}$ with $g_{n,n-1}\circ\varepsilon$, 
it follows that
$$
\{\mathscr{S}_r,\overline{f_r},\Omega_r\,|\,2\le r\le n-2\}\cup\{\mathscr{S}_{n-1},
\overline{f_{n-1}},\Omega_{n-1}',\mathscr{S}_n',\overline{f_n},\Omega_n\}
$$
is a $q$-presentation of $\vec{\bm f}$ and it represents $\alpha\circ \varepsilon$. 
Hence $\alpha\circ \varepsilon\in\{\vec{\bm f}\,\}^{(q)}$. 
Thus $\{\vec{\bm f}\,\}^{(q)}\supset\{\vec{\bm f}\,\}^{(aq)}\circ\mathscr{E}(\Sigma^{n-2}X_1)$. 

Therefore $\{\vec{\bm f}\,\}^{(q)}=\{\vec{\bm f}\,\}^{(aq)}\circ\mathscr{E}(\Sigma^{n-2}X_1)$. 
\end{proof}

\begin{proof}[Proof of Theorem 6.2.1(4)] 
We prove
\begin{gather}
\{\vec{\bm f}\,\}^{(\ddot{s}_t)}\subset \{\vec{\bm f}\,\}^{(aq\ddot{s}_2)}\circ\mathscr{E}(\Sigma^{n-2}X_1),\\ 
\{\vec{\bm f}\,\}^{(\ddot{s}_t)}\circ\mathscr{E}(\Sigma^{n-2}X_1)\supset \{\vec{\bm f}\,\}^{(aq\ddot{s}_2)}. 
\end{gather}
If these are done, then, by applying $\mathscr{E}(\Sigma^{n-2}X_1)$ to them from the right, we have the equality.  
To prove (6.2.7), let 
$\alpha\in\{\vec{\bm f}\,\}^{(\ddot{s}_t)}$ and $\{\mathscr{S}_r,\overline{f_r},\mathscr{A}_r\,|\,2\le r\le n\}$ 
an $\ddot{s}_t$-presentation of $\vec{\bm f}$ with $\alpha=\overline{f_n}\circ g_{n,n-1}$. 
We define inductively
$$
\mathscr{S}_2'=\mathscr{S}_2, \Omega_2'=\Omega(\mathscr{A}_2);\, 
\mathscr{S}_{k+1}'=\mathscr{S}_k'(\overline{f_k},\Omega_k'),\ \Omega_{k+1}'=
\widetilde{\Omega_k'}\ (2\le k< n).
$$
By Remark 5.5(3), this definition is possible, and $\mathscr{S}_r', \mathscr{S}_r$ have the same edge. 
Then $\{\mathscr{S}_r',\overline{f_r},\Omega_r'\,|\,2\le r\le n\}$ is an $aq\ddot{s}_2$-presentation  of 
$\vec{\bm f}$ and 
\begin{align*}
\{\vec{\bm f}\,\}^{(aq\ddot{s}_2)}\ni \overline{f_n}\circ g_{n,n-1}'
&=\overline{f_n}\circ(\overline{f_{n-1}}\cup C1_{C_{n-1,n-2}})\circ\omega_{n-1,n-2}'^{-1}\\
&=\overline{f_n}\circ(\overline{f_{n-1}}\cup C1_{C_{n-1,n-2}})\circ\omega_{n-1,n-2}^{-1}\circ\varepsilon_0
=\alpha\circ\varepsilon_0,
\end{align*}
where $\varepsilon_0=\omega_{n-1,n-2}\circ\omega_{n-1,n-2}'^{-1}\in\mathscr{E}(\Sigma^{n-2}X_1)$. 
Hence 
$$
\alpha\in\{\vec{\bm f}\,\}^{(aq\ddot{s}_2)}\circ\varepsilon^{-1}_0\subset
\{\vec{\bm f}\,\}^{(aq\ddot{s}_2)}\circ\mathscr{E}(\Sigma^{n-2}X_1). 
$$
This proves (6.2.7). 

To prove (6.2.8), let $\alpha\in\{\vec{\bm f}\,\}^{(aq\ddot{s}_2)}$ and $\{\mathscr{S}_r,\overline{f_r},\Omega_r\,|\,2\le r\le n\}$ 
an $aq\ddot{s}_2$-presentation of $\vec{\bm f}$ with $\alpha=\overline{f_n}\circ g_{n,n-1}$. 
Set $\mathscr{S}_2'=\mathscr{S}_2$ and $\mathscr{A}_2'=\{1_{C_{2,2}}\}$. 
We define inductively 
$\mathscr{S}_{r+1}'=\mathscr{S}_r'(\overline{f_r},\mathscr{A}_r')$ and $\mathscr{A}_{r+1}'$ is a reduced structure on $\mathscr{S}_{r+1}'$ for $r\ge 2$. 
By Remark 5.5(3), this definition is possible, and $\mathscr{S}_r', \mathscr{S}_r$ have the same edge. 
Then $\{\mathscr{S}_r',\overline{f_r},\mathscr{A}_r'\,|\,2\le r\le n\}$ is an 
$\ddot{s}_t$-presentation of $\vec{\bm f}$ and 
\begin{align*}
\{\vec{\bm f}\,\}^{(\ddot{s}_t)}\ni \overline{f_n}\circ g_{n,n-1}'
&=\overline{f_n}\circ (\overline{f_{n-1}}\cup C1_{C_{n-1,n-2}})\circ\omega_{n-1,n-2}'^{-1}\\
&=\overline{f_n}\circ (\overline{f_{n-1}}\cup C1_{C_{n-1,n-2}})\circ\omega_{n-1,n-2}^{-1}\circ\varepsilon_0
=\alpha\circ\varepsilon_0,
\end{align*}
where 
$\Omega(\mathscr{A}_{n-1}')=\{\omega_{n-1,s}'\,|\,1\le s<n-1\}$ and 
$\varepsilon_0=\omega_{n-1,n-2}\circ\omega_{n-1,n-2}'^{-1}\in\mathscr{E}(\Sigma^{n-2}X_1)$. 
Hence $\alpha\in\{\vec{\bm f}\,\}^{(\ddot{s}_t)}\circ\varepsilon_0^{-1}\subset \{\vec{\bm f}\,\}^{(\ddot{s}_t)}\circ\mathscr{E}(\Sigma^{n-2}X_1)$. 
This proves (6.2.8) and completes the proof of Theorem~6.2.1(4).
\end{proof}

\begin{proof}[Proof of Theorem 6.2.1(5)]
Let $\alpha\in\{\vec{\bm f}\}^{(q)}$ and $\{\mathscr{S}_r,\overline{f_r},\Omega_r\,|\,2\le r\le n\}$ a 
$q$-presenta-tion of $\vec{\bm f}$ with $\alpha=\overline{f_n}\circ g_{n,n-1}$. 
We are going to define an $aq\ddot{s}_2$-presentation $\{\mathscr{S}_r',\overline{f_r}',\Omega_r'\,|\,2\le r\le n\}$ of $\vec{\bm f}$ such that 
$\overline{f_n}'\circ g_{n,n-1}'=\alpha\circ \theta$ for some map $\theta:\Sigma^{n-2}X_1\to \Sigma^{n-2}X_1$ 
(notice that $\theta$ is not necessarily a homotopy equivalence). 
Now set
$$
\mathscr{S}_2'=(X_1;X_2,X_2\cup_{f_1}CX_1;f_1;i_{f_1}),\quad \Omega_2'=\{q_{f_1}'\}.
$$
Since $\mathscr{S}_2$ is a quasi iterated mapping cone and $j_{2,1}'=i_{f_1}$ is a cofibration, there exists a map 
(not necessarily a homotopy equivalence) $e_2:C_{2,2}'\to C_{2,2}$ such that $e_2\circ j_{2,1}'=j_{2,1}$. 
Set $\overline{f_2}'=\overline{f_2}\circ e_2$. 
Then $\overline{f_2}'\circ j_{2,1}'=f_2$ and so $\overline{f_2}'$ is an extension of $f_2$ to $C_{2,2}'$. 
Set 
\begin{align*}
&\mathscr{S}_3'=\mathscr{S}_2'(\overline{f_2}',\Omega_2'),\ \Omega_3'=\widetilde{\Omega_2'},\ 
 e_3=1_{X_3}\cup Ce_2:C_{3,3}'\to C_{3,3},\\ 
&\hspace{2cm}\overline{f_3}'=\begin{cases} \overline{f_3}:C_{3,2}'=C_{3,2}\to X_4 & n=3\\ \overline{f_3}\circ e_3:C_{3,3}'\to X_4 & n\ge 4\end{cases}.
\end{align*}
Proceeding with the construction, we have an $aq\ddot{s}_2$-presentation $\{\mathscr{S}_r',\overline{f_r}',\Omega_r'\,|\, 2\le r\le n\}$ 
of $\vec{\bm f}$ and maps $e_r:C_{r,r}'\to C_{r,r}\ (2\le r\le n-1)$ such that  
\begin{gather*}
C_{r,s}'=C_{r,s}\ (1\le s\le r-1),\ j_{r,s}'=j_{r,s}\ (1\le s\le r-2),\\
 e_r\circ j_{r,r-1}'=j_{r,r-1}\ (2\le r\le n-1),\\ 
\overline{f_r}'=\begin{cases} \overline{f_n} : C_{n,n-1}'=C_{n,n-1}\to X_{n+1} & r=n\\ \overline{f_r}\circ e_r : C_{r,r}'\to X_{r+1} & r<n\end{cases}.
\end{gather*}
Set $\theta=\omega_{n-1,n-2}\circ (e_{n-1}\cup C1_{C_{n-1,n-2}'})\circ\omega_{n-1,n-2}'^{-1}:\Sigma^{n-2}X_1\to \Sigma^{n-2}X_1$. 
Then $\omega_{n-1,n-2}^{-1}\circ \theta=(e_{n-1}\cup C1_{C_{n-1,n-2}'})\circ\omega_{n-1,n-2}'^{-1}$ and 
\begin{align*}
\{\vec{\bm f}\,\}^{(aq\ddot{s}_2)}\ni&\overline{f_n}'\circ g_{n,n-1}'=\overline{f_n}'\circ(\overline{f_{n-1}}'\cup C1_{C_{n-1,n-2}})\circ\omega_{n-1,n-2}'^{-1}\\
&=\overline{f_n}\circ(\overline{f_{n-1}}\cup C1_{C_{n-1,n-2}})\circ (e_{n-1}\cup C1_{C_{n-1,n-2}})\circ\omega_{n-1,n-2}'^{-1}\\
&=\overline{f_n}\circ(\overline{f_{n-1}}\cup C1_{C_{n-1,n-2}})\circ\omega_{n-1,n-2}^{-1}\circ \theta\\
&=\alpha\circ \theta.
\end{align*}
Since $\{\vec{\bm f}\,\}^{(aq\ddot{s}_2)}\subset \{\vec{\bm f}\,\}^{(aq\ddot{s}_2)}\circ\mathscr{E}(\Sigma^{n-2}X_1)=
\{\vec{\bm f}\,\}^{(\ddot{s}_t)}\circ\mathscr{E}(\Sigma^{n-2}X_1)$ by (4), we have $\alpha\circ\theta=\beta\circ\gamma$ 
for some $\beta\in\{\vec{\bm f}\,\}^{(\ddot{s}_t)}$ and $\gamma\in\mathscr{E}(\Sigma^{n-2}X_1)$. Set $\theta'=\theta\circ\gamma^{-1}$. 
Then $\alpha\circ\theta'=\beta\in\{\vec{\bm f}\,\}^{(\ddot{s}_t)}$. 
\end{proof}

\begin{proof}[Proof of Corollary 6.2.2]
To prove (1), 
let $\varepsilon\in\mathscr{E}(\Sigma^{n-2}X_1)$. 
By composing $\varepsilon$ from the right to equalities $\{\vec{\bm f}\,\}^{(q)}=\{\vec{\bm f}\,\}^{(aq)}\circ\mathscr{E}(\Sigma^{n-2}X_1)$ 
in Theorem~6.2.1(3) and 
$\{\vec{\bm f}\,\}^{(qs_2)}=\{\vec{\bm f}\,\}^{(aq\ddot{s}_2)}\circ\mathscr{E}(\Sigma^{n-2}X_1)$ in Theorem 6.2.1(2), we have 
$\{\vec{\bm f}\,\}^{(q)}\circ\varepsilon=\{\vec{\bm f}\,\}^{(q)}$ and 
 $\{\vec{\bm f}\,\}^{(qs_2)}\circ\varepsilon=\{\vec{\bm f}\,\}^{(qs_2)}$. 
Let $\gamma\in\mathscr{E}(\Sigma X_1)$. 
By composing $\Sigma^{n-3}\gamma$ from the right to equality 
$\{\vec{\bm f}\,\}^{(aqs_2)}=\{\vec{\bm f}\,\}^{(aq\ddot{s}_2)}\circ \Sigma^{n-3}\mathscr{E}(\Sigma X_1)$
 in Theorem 6.2.1(2), we have $\{\vec{\bm f}\,\}^{(aqs_2)}\circ \Sigma^{n-3}\gamma=\{\vec{\bm f}\,\}^{(aqs_2)}$. 

To prove (2), 
let $\alpha\in\{\vec{\bm f}\,\}^{(aq)}$ and $\{\mathscr{S}_r,\overline{f_r},\Omega_r\,|\,2\le r\le n\}$ 
an $aq$-presen-tation of $\vec{\bm f}$ with $\alpha=\overline{f_n}\circ g_{n,n-1}$. 
For the first equality in (2), let $\gamma\in\mathscr{E}(X_1)$, and set 
$$
\omega_{r,s}'=\begin{cases} \omega_{r,s} & 1\le s\le r-2\\ \Sigma^{r-1}\gamma\circ\omega_{r,r-1} & 1\le s=r-1\end{cases}.
$$
Then $\Omega_r'=\{\omega_{r,s}'\,|\,1\le s<r\}$ is a quasi-structure on $\mathscr{S}_r$ and $\Omega_{r+1}'=\widetilde{\Omega_r'}$. 
For $r\ge 2$ and $1\le s<r<n$, we set 
$g_{r+1,s+1}'=\begin{cases} g_{r+1,s+1} & s\le r-2\\ g_{r+1,r}\circ \Sigma^{r-1}\gamma^{-1} & s=r-1 \end{cases}$, that is, 
$$
g_{r+1,s+1}'=\begin{cases} (\overline{f_r}^{s+1}\cup C1_{C_{r,s}})\circ\omega_{r,s}^{-1} & s\le r-2\\
(\overline{f_r}\cup C1_{C_{r,r-1}})\circ\omega_{r,r-1}^{-1}\circ \Sigma^{r-1}\gamma^{-1} & s=r-1
\end{cases}.
$$
Set $\mathscr{S}_2'=\mathscr{S}_2$ and let $\mathscr{S}_{r+1}'$ be the iterated mapping cone obtained from $\mathscr{S}_{r+1}$ by 
replacing $g_{r+1,r}$ with $g_{r+1,r}'$ for $2\le r\le n-1$. 
Then $\mathscr{S}_r'$ and $\mathscr{S}_r$ have the same edge, and $\{\mathscr{S}_r',\overline{f_r},\Omega_r'\,|\,2\le r\le n\}$ 
is an $aq$-presentation of $\vec{\bm f}$ such that 
$$
\overline{f_n}\circ g_{n,n-1}'=\overline{f_n}\circ g_{n,n-1}\circ \Sigma^{n-2}\gamma^{-1}=\alpha\circ \Sigma^{n-2}\gamma^{-1}.
$$
Hence $\{\vec{\bm f}\,\}^{(aq)}\circ \Sigma^{n-2}\gamma^{-1}\subset\{\vec{\bm f}\,\}^{(aq)}$ and so 
$\{\vec{\bm f}\,\}^{(aq)}\subset \{\vec{\bm f}\,\}^{(aq)}\circ \Sigma^{n-2}\gamma$. 
By taking $\gamma^{-1}$ instead of $\gamma$, we have 
$\{\vec{\bm f}\,\}^{(aq)}\circ \Sigma^{n-2}\gamma\subset\{\vec{\bm f}\,\}^{(aq)}$. 
Therefore 
we obtain the first equality in (2). 
For the second equality in (2), 
set $\omega_{2,1}^*=(-1_{\Sigma X_1})\circ\omega_{2,1},\ 
\omega_{3,2}^*=\Sigma(-1_{\Sigma X_1})\circ\omega_{3,2},\dots,\omega_{n,n-1}^*
=\Sigma^{n-2}(-1_{\Sigma X_1})\circ\omega_{n,n-1}$, $\Omega_2^*=\{\omega_{2,1}^*\}$, 
and $\Omega_r^*=\{\omega_{r,1},\dots,\omega_{r,r-2},\omega_{r,r-1}^*\}$ for $3\le r\le n$. 
Set $\mathscr{S}_2^*=\mathscr{S}_2$ and, for $r\ge 3$, let $\mathscr{S}_r^*$ be 
the iterated mapping cone obtained from $\mathscr{S}_r$ by replacing $g_{r,r-1}$ with 
$g_{r,r-1}^*=g_{r,r-1}\circ\Sigma^{r-3}(-1_{\Sigma X_1})$. 
Then $\{\mathscr{S}_r^*,\overline{f_r},\Omega_r^*\,|\,2\le r\le n\}$ is an $aq$-presentation of $\vec{\bm f}$ 
by Lemma 4.3(3) such that $-\alpha=\overline{f_n}\circ g_{n,n-1}^*\in\{\vec{\bm f}\,\}^{(aq)}$. 
Hence $-\{\vec{\bm f}\,\}^{(aq)}\subset\{\vec{\bm f}\,\}^{(aq)}$. 
By composing $-1_{\Sigma^{n-2}X_1}$ from the right to the last relation, we have 
$\{\vec{\bm f}\,\}^{(aq)}\subset -\{\vec{\bm f}\,\}^{(aq)}$. 
Therefore $-\{\vec{\bm f}\,\}^{(aq)}=\{\vec{\bm f}\,\}^{(aq)}$. 

The assertion (3) follows from (2) and Theorem 6.2.1(3). 

If $\{\vec{\bm f}\,\}^{(\star)}$ is not empty for some $\star$, 
then $\{\vec{\bm f}\,\}^{(q)}$ is not empty by (6.1.1) so that $\{\vec{\bm f}\,\}^{(aq\ddot{s}_2)}$ and 
$\{\vec{\bm f}\,\}^{(\ddot{s}_t)}$ are not empty by Theorem 6.2.1(5), and so $\{\vec{\bm f}\,\}^{(\star)}$ is 
not empty for every $\star$ by (6.1.1) and Theorem 6.2.1. 
This proves (4). 

If $\{\vec{\bm f}\,\}^{(\star)}$ contains $0$ for some $\star$, then $\{\vec{\bm f}\,\}^{(q)}$ contains $0$ by (6.1.1), 
and so $\{\vec{\bm f}\,\}^{(\star)}$ contains $0$ for every $\star$ by Theorem 6.2.1(5)  and (6.1.1). 
This proves (5). 

By setting $\varepsilon=-1_{\Sigma^{n-2}X_1}$ and $\gamma=-1_{\Sigma X_1}$ in (1), 
we have (6) for $\star=q, qs_2, aqs_2$. 
The assertion (6) for $\star=aq$ was proved in (2). 

Suppose that $n\ge 4$ and $\vec{\bm f}$ is $\star$-presentable for some $\star$. 
Then $\vec{\bm f}$ is $\star$-presentable for all $\star$ by (4) so that in particular 
it is $aq\ddot{s}_2$-presentable and so it is admissible by definitions and (4.2). 
This proves (7). 

We have (8) from (6.1.1) and Theorem 6.2.1(1),(2),(4).  
\end{proof}

\begin{proof}[Proof of Proposition 6.2.3]
(6.2.1) is easily obtained from definitions. 

About (6.2.2), it suffices to prove it for $\star=aq\ddot{s}_2, \ddot{s}_t,aq,q_2$ by Theorem~6.2.1(1)-(3). 
We prove (6.2.2) for $\star=aq\ddot{s}_2$, because other cases can be proved similarly. 
Let $\alpha\in\{f_{n+1}\circ f_n,f_{n-1},\dots,f_1\}^{(aq\ddot{s}_2)}$ and 
$\{\mathscr{S}_r,\overline{f_r},\Omega_r\,|\,2\le r< n\}\cup\{\mathscr{S}_n,\overline{f_{n+1}\circ f_n},\Omega_n\}$ 
an $aq\ddot{s}_2$-presentation of $(f_{n+1}\circ f_n,f_{n-1},\dots,f_1)$ such that $\alpha=\overline{f_{n+1}\circ f_n}\circ g_{n,n-1}$. 
Then $C_{n,r}=X_n\cup_{\overline{f_{n-1}}^{r-1}}CC_{n-1,r-1}$ and $\overline{f_{n+1}\circ f_n}$ is an extension of $f_{n+1}\circ f_n$ to $C_{n,n-1}$. 
Set $\mathscr{S}_n'=\mathscr{S}_{n-1}(f_n\circ\overline{f_{n-1}},\Omega_{n-1})$ and $\Omega_n'$ the typical quasi-structure on $\mathscr{S}_n'$. 
Then $C_{n,r}'=X_{n+1}\cup_{f_n\circ\overline{f_{n-1}}^{r-1}}CC_{n-1,r-1}$. 
Define $\overline{f_n}':C_{n,n-1}'=X_{n+1}\cup_{f_n\circ\overline{f_{n-1}}^{n-2}}CC_{n-1,n-2}\to X_{n+2}$ 
by $\overline{f_n}'|_{X_{n+1}}=f_{n+1}$ and $\overline{f_n}'|_{CC_{n-1,n-2}}=\overline{f_{n+1}\circ f_n}|_{CC_{n-1,n-2}}$. 
The map $\overline{f_n}'$ is a well-defined extension of $f_{n+1}$ to $C_{n,n-1}'$ and 
$\overline{f_n}'\circ(f_n\cup C1_{C_{n-1,n-2}})=\overline{f_{n+1}\circ f_n}$. 
Hence $\{\mathscr{S}_r,\overline{f_r},\Omega_r\,|\,2\le r\le n-2\}\cup\{\mathscr{S}_{n-1},f_n\circ\overline{f_{n-1}},
\Omega_{n-1}\}\cup\{\mathscr{S}_n',\overline{f_n}',\Omega_n'\}$ is an $aq\ddot{s}_2$-presentation of 
$(f_{n+1},f_n\circ f_{n-1},\dots,f_1)$ and it represents $\overline{f_n}'\circ (f_n\cup C1_{C_{n-1,n-2}})\circ g_{n,n-1}=\alpha$ 
so that $\alpha\in\{f_{n+1},f_n\circ f_{n-1},\dots,f_1\}^{(aq\ddot{s}_2)}$. 
This proves (6.2.2) for $\star=aq\ddot{s}_2$. 

To prove the first containment of (6.2.3), let 
$\{\mathscr{S}_r,\overline{f_r},\Omega_r\,|\, 2\le r\le n\}$ 
be an $aq\ddot{s}_2$-presentation 
of $(f_n,\dots,f_1)$ and set $\alpha=\overline{f_n}\circ g_{n,n-1}$. 
We are going to construct an $aq\ddot{s}_2$-presentation $\{\mathscr{S}_r',\overline{f_r}',\Omega_r'\,|\,2\le r\le n\}$ 
of $(f_n,\dots,f_2,f_1\circ f_0)$ with $\overline{f_n}'\circ g_{n,n-1}'=\alpha\circ \Sigma^{n-2}f_0$. 
We set 
\begin{gather*}
X_1'=X_0,\ X_k'=X_k\ (2\le k\le n+1),\ f_1'=f_1\circ f_0,\ f_k'=f_k\ (2\le k\le n),\\
e_{1,0}=1_{\{*\}}:C_{1,0}'\to C_{1,0},\ e_{1,1}=f_0:C_{1,1}'=X_0\to C_{1,1}=X_1,\\
\mathscr{S}_2'=(X_1';X_2',X_2'\cup_{f_1'}CX_1';f_1';i_{f_1'}),\ \Omega_2'=\{q'_{f_1'}\},\\
e_{2,0}=1_{\{*\}}:C_{2,0}'\to C_{2,0},\ e_{2,s}=1_{X_2}\cup Ce_{1,s-1}:C_{2,s}'\to C_{2,s}\ (s=1,2),\\
\overline{f_2}'=\overline{f_2}\circ e_{2,2}:C_{2,2}'\to X_3,\ 
\mathscr{S}_3'=\mathscr{S}_2'(\overline{f_2}',\Omega_2'),\ \Omega_3'=\widetilde{\Omega_2'}.
\end{gather*}
Then $C_{3,s}'=C_{3,s}$ for $s=0,1,2$. 
We set 
\begin{gather*}
e_{3,0}=1_{\{*\}}:C_{3,0}'\to C_{3,0},\ e_{3,s}=1_{X_3}\cup Ce_{2,s-1}:C_{3,s}'\to C_{3,s}\ (1\le s\le 3),\\
\overline{f_3}'=\begin{cases} \overline{f_3}\circ e_{3,2}:C_{3,2}'\to X_4 & n=3\\
\overline{f_3}\circ e_{3,3}:C_{3,3}'\to X_4 & n\ge 4\end{cases}.
\end{gather*}
Then $e_{3,s}=1_{C_{3,s}}$ for $0\le s\le 2$ and 
$g_{3,2}\circ \Sigma f_0\simeq g_{3,2}'$. 
By repeating the process, we have an $aq\ddot{s}_2$-presentation 
$\{\mathscr{S}_r',\overline{f_r}',\Omega_r\,|\,2\le r\le n\}$ of 
$(f_n,\dots,f_2,f_1\circ f_0)$ such that $C_{r,s}'=C_{r,s}\ (1\le s<r)$, 
$\overline{f_n}'=\overline{f_n}$ and 
$g_{n,n-1}\circ\Sigma^{n-2}f_0\simeq g_{n,n-1}'$ so that 
$\overline{f_n}\circ g_{n,n-1}\circ \Sigma^{n-2}f_0\simeq\overline{f_n}'\circ g_{n,n-1}'$. 
Hence $\alpha\circ\Sigma^{n-2}f_0\in\{f_n,\dots,f_2,f_1\circ f_0\}^{(aq\ddot{s}_2)}$. 
This proves the first containment of (6.2.3). 

In the rest of the proof we prove the second containment of (6.2.3). 
Let $\alpha\in\{f_n,\dots,f_2,f_1\circ f_0\}^{(aq\ddot{s}_2)}$ and 
$\{\mathscr{S}_r,\overline{f_r},\Omega_r\,|\,2\le r\le n\}$ an 
$aq\ddot{s}_2$-presentation of $(f_n,\dots,f_2,f_1\circ f_0)$ with 
$\overline{f_n}\circ g_{n,n-1}=\alpha$. 
Set 
\begin{gather*}
X_1^*=X_0,\quad X_2^*=X_1,\quad f_1^*=f_0,\quad X_k^*=X_k\ (3\le k\le n),\\
 f_2^*=f_2\circ f_1, \quad f_k^*=f_k\ (3\le k\le n),\\ 
\mathscr{S}_2^*=(X_1^*;X_2^*, X_2^*\cup_{f_1^*}CX_1^*;f_1^*;i_{f_1^*}),\quad \Omega_2^*=\{q'_{f_1^*}\},\\ 
e_{2,2}=f_1\cup C1_{X_0}:C_{2,2}^*=X_1\cup_{f_0}CX_0\to C_{2,2}=X_2\cup_{f_1\circ f_0}CX_0,\\
e_{2,1}=f_1:C_{2,1}^*\to C_{2,1},\quad \overline{f_2}^*=\overline{f_2}\circ e_{2,2}:C_{2,2}^*\to X_3,\\
\mathscr{S}_3^*=\mathscr{S}_2^*(\overline{f_2}^*,\Omega_2^*),\quad \Omega_3^*=\widetilde{\Omega_2^*},\quad 
e_{3,1}=1_{X_3},\\
e_{3,2}=1_{X_3}\cup Cf_1:C_{3,2}^*=X_3\cup_{f_2\circ f_1}CX_1
\to X_3\cup_{f_2}CX_2=C_{3,2},\\
e_{3,3}=1_{X_3}\cup Ce_{2,2}:C_{3,3}^*=X_3\cup_{\overline{f_2}^*}CC_{2,2}^*\to X_3\cup_{\overline{f_2}}CC_{2,2}=C_{3,3},\\
\overline{f_3}^*=\begin{cases}\overline{f_3}\circ e_{3,2}:C_{3,2}^*\to X_4 & n=3\\
\overline{f_3}\circ e_{3,3}:C_{3,3}^*\to X_4 & n\ge 4\end{cases}. 
\end{gather*}
Then $e_{3,s+1}\circ j_{3,s}^*=j_{3,s}\circ e_{3,s}\ (s=1,2)$ and $e_{3,2}\circ g_{3,2}^*\simeq g_{3,2}$. 
When $n=3$, $\{\mathscr{S}_r^*,\overline{f_r}^*,\Omega_r^*\,|\,r=2,3\}$ is 
an $aq\ddot{s}_2$-presentation of $(f_3,f_2\circ f_1,f_0)$ and $\overline{f_3}\circ g_{3,2}\simeq \overline{f_3}^*\circ g_{3,2}^*$ so that $\alpha\in\{f_3,f_2\circ f_1,f_0\}^{(aq\ddot{s}_2)}$. 
Suppose $n\ge 4$. 
Set 
$\mathscr{S}_4^*=\mathscr{S}_3^*(\overline{f_3}^*,\Omega_3^*)$ and 
$\Omega_4^*=\widetilde{\Omega_3^*}$. 
Then $C_{4,s}^*=C_{4,s}\ (s=1,2)$. 
Set \begin{align*}
&e_{4,s}=1_{C_{4,s}}:C_{4,s}^*\to C_{4,s}\ (s=1,2),\\
&e_{4,s+1}=1_{X_4}\cup Ce_{3,s}\\
&\hspace{1cm}:C_{4,s+1}^*=X_4\cup_{\overline{f_3}^{*s}}CC_{3,s}^*\to 
C_{4,s+1}=X_4\cup_{\overline{f_3}^s}CC_{3,s}\ (s=2,3),\\
&\overline{f_4}^*=\begin{cases} \overline{f_4}\circ e_{4,3}:C_{4,3}^*\to X_5 & n=4\\
\overline{f_4}\circ e_{4,4}:C_{4,4}^*\to X_5 & n\ge 5\end{cases}.
\end{align*}
Then $e_{4,s+1}\circ j_{4,s}^*=j_{4,s}\circ e_{4,s}\ (1\le s\le 3)$ and 
$e_{4,3}\circ g_{4,3}^*\simeq g_{4,3}$. 
By repeating the process, we obtain $\{\mathscr{S}_r^*,\overline{f_r}^*,\Omega_r^*\,|\,2\le r\le n\}$ which is an $aq\ddot{s}_2$-presentation of $(f_n,\dots,f_2\circ f_1, f_0)$ such that $e_{n,n-1}\circ g_{n,n-1}^*\simeq g_{n,n-1}$ and 
$\overline{f_n}^*=\overline{f_n}\circ e_{n,n-1}$ so that 
$\overline{f_n}^*\circ g_{n,n-1}^*\simeq \overline{f_n}\circ g_{n,n-1}$. 
This shows $\alpha\in\{f_n,\dots,f_2\circ f_1,f_0\}^{(aq\ddot{s}_2)}$ and completes the proof of Proposition 6.2.3.
\end{proof}

\subsection{Suspension of higher Toda brackets}
Let $\vec{\bm f}=(f_n,\dots,f_1)$. 
Given $\ell\ge 1$,
$\{\Sigma^\ell\vec{\bm f}\,\}^{(\star)}$ can be considered, since $\Sigma^\ell X_i$ is well-pointed for every $i$ 
by Corollary~2.3(2). 
We prove the following which implies (1.5) (cf. \cite[Lemma~2.3D]{G}, \cite[Lemma~2.3]{W2}). 

\begin{thm'}%Theorem 6.3.1
\begin{enumerate}
\item[\rm(1)] $\Sigma^\ell\{\vec{\bm f}\,\}^{(\star)}\subset (-1)^{\ell n}\{\Sigma^\ell\vec{\bm f}\,\}^{(\star)}$ 
for all $\ell\ge 1$ and $\star$.
\item[\rm(2)] $\Sigma^\ell\{\vec{\bm f}\,\}^{(\star)}\subset \{\Sigma^\ell\vec{\bm f}\,\}^{(\star)}$ for 
all $\ell\ge 1$ and $\star=q,\, qs_2,\, aq,\, aqs_2$. 
\end{enumerate}
\end{thm'}
\begin{proof}
(2) follows from (1) and Corollary 6.2.2(6). 
We are going to prove (1). 
Note that 
$$
\Sigma^\ell\{\vec{\bm f}\,\}^{(\star)}\subset[\Sigma^\ell\Sigma^{n-2}X_1,\Sigma^\ell X_{n+1}],\ 
\{\Sigma^\ell\vec{\bm f}\,\}^{(\star)}\subset[\Sigma^{n-2}\Sigma^\ell X_1,\Sigma^\ell X_{n+1}],
$$
where $\Sigma^\ell\Sigma^{n-2}X_1=\Sigma^{n-2}\Sigma^\ell X_1$ by the identification (2.1). 
Hence (1) is equivalent to
\begin{equation}
\Sigma^\ell\{\vec{\bm f}\,\}^{(\star)}\subset\{\Sigma^\ell\vec{\bm f}\,\}^{(\star)}\circ(1_{X_1}\wedge\tau(\s^{n-2},\s^\ell)),
\end{equation} 
where $\tau(\s^{n-2},\s^\ell):\s^{n-2+\ell}=\s^{n-2}\wedge\s^\ell\to\s^\ell\wedge\s^{n-2}=\s^{n-2+\ell}$ 
is the switching homeomorphism defined in (2.2). 
We prove (6.3.1) when $\ell=1$, because (6.3.1) for $\ell\ge 2$ is obtained by an induction. 
By Theorem 6.2.1(1)-(4), it suffices to prove (6.3.1) for $\star=\ddot{s}_t$, $aq\ddot{s}_2$, $aq$, $q\ddot{s}_2$, $q_2$. 
We prove (6.3.1) for only the case of $\star=\ddot{s}_t$, 
because the cases of $\star=aq\ddot{s}_2$, $aq$, $q\ddot{s}_2$, $q_2$ can be treated similarly or more easily. 
Let $\alpha\in\{\vec{\bm f}\,\}^{(\ddot{s}_t)}$ and $\{\mathscr{S}_r, \overline{f_r},\mathscr{A}_r\,|\,2\le r\le n\}$ 
an $\ddot{s}_t$-presentation of $\vec{\bm f}$ with $\alpha=\overline{f_n}\circ g_{n,n-1}$. 
Set $C^*_{r,1}=\Sigma X_r\ (2\le r\le n)$ and $f^*_r=\Sigma f_r:\Sigma X_r\to \Sigma X_{r+1}\ (1\le r\le n)$. 
We are going to construct an $\ddot{s}_t$-presentation $\{\mathscr{S}_r^*,\overline{f_r^*},\mathscr{A}_r^*\,|\,2\le r\le n\}$  
of $\Sigma\vec{\bm f}$ such that $\Sigma(\overline{f_n}\circ g_{n,n-1})\simeq(\overline{f_n^*}\circ g_{n,n-1}^*)\circ
(1_{X_1}\wedge\tau(\s^{n-2},\s^1))$, where 
\begin{align*}
&\mathscr{S}^*_r=(\Sigma X_{r-1},\Sigma^2 X_{r-2},\dots,\Sigma^{r-1} X_1;
C^*_{r,1},\dots,C^*_{r,r};\\
&\hspace{4cm}g^*_{r,1},\dots,g^*_{r,r-1};j^*_{r,1},\dots,j^*_{r,r-1}),\\
&g^*_{r,1}=f_{r-1}^*,\quad C^*_{r,2}=\Sigma X_r\cup_{f^*_{r-1}}C\Sigma X_{r-1},\quad  j_{r,1}^*=i_{g^*_{r,1}}. 
\end{align*}

Set $\mathscr{S}_2^*=(\Sigma X_1;\Sigma X_2,\Sigma X_2\cup_{f_1^*}C\Sigma X_1;f_1^*;i_{f_1^*})$. 
Then $\mathscr{S}_2^*$ is an iterated mapping cone with a reduced structure $\mathscr{A}_2^*=\{a_{2,1}^*\}$ 
and a reduced quasi-structure $\Omega(\mathscr{A}_2^*)=\{\omega_{2,1}^*\}$, 
where $a_{2,1}^*=1_{C^*_{2,2}}$ and $\omega^*_{2,1}=q'_{f^*_1}$. 
Set 
$e_{2,s}=\begin{cases} 1_{\{*\}} & s=0\\ \psi_{f_1^{s-1}} & s=1,2\end{cases}:C^*_{2,s}\approx\Sigma C_{2,s}$ 
and $\overline{f_2^*}=\Sigma\overline{f_2}\circ e_{2,2}:C_{2,2}^*\to \Sigma X_3$. 
Then $e_{2,1}=1_{\Sigma X_2}$, $\overline{f_2^*}$ is an extension of $f_2^*$ to $C^*_{2,2}$,  and 
\begin{align*}
& a_{2,1}^*=(e^{-1}_{2,1}\cup C(1_{X_1}\wedge\tau(\s^0,\s^1)))\circ(\psi^1_{g_{2,1}})^{-1}\circ\Sigma a_{2,1}\circ e_{2,2},\\
&e_{2,s+1}\circ j_{2,s}^*=\Sigma j_{2,s}\circ e_{2,s}\ (s=0,1),\\ 
&{\overline{f_2^*}}^s=\Sigma\overline{f_2}^s\circ e_{2,s}:C^*_{2,s}\to \Sigma X_3\ (s=0,1,2),\\  
&\omega_{2,1}^*=(1_{X_1}\wedge\tau(\s^0\wedge\s^1,\s^1))\circ\Sigma\omega_{2,1}\circ\psi^1_{j_{2,1}}\circ(e_{2,2}\cup Ce_{2,1})\\
&\hspace{2cm}:C^*_{2,2}\cup_{j^*_{2,1}} CC^*_{2,1}\to \Sigma\Sigma X_1.
\end{align*}

Set $\mathscr{S}_3^*=\mathscr{S}_2^*(\overline{f_2^*},\mathscr{A}_2^*)$ and 
\begin{gather*}
e_{3,s}=\begin{cases} 1_{\{*\}} & s=0\\ \psi^1_{\overline{f_2}^{s-1}}\circ(1_{\Sigma X_3}\cup Ce_{2,s-1}) & s=1,2,3\end{cases}
:C_{3,s}^*\approx \Sigma C_{3,s},\\
\overline{f_3^*}=\begin{cases} \Sigma\overline{f_3}\circ e_{3,2} : C^*_{3,2}\to \Sigma X_4& n=3\\ 
\Sigma\overline{f_3}\circ e_{3,3}:C^*_{3,3}\to \Sigma X_4 & n\ge 4\end{cases}.
\end{gather*}
Then $\mathscr{S}_3^*$ is an iterated mapping cone, $\overline{f_3^*}$ is an extension of $f^*_3$ to 
$C^*_{3,2}$ or $C^*_{3,3}$ according as $n=3$ or $n\ge 4$, and
\begin{gather*}
e_{3,1}=1_{\Sigma X_3},\quad e_{3,s+1}\circ j_{3,s}^*=\Sigma j_{3,s}\circ e_{3,s}\ (s=0,1,2),\\ \overline{f_3^*}^s=\Sigma\overline{f_3}^s\circ e_{3,s}\ (s=1,2),\\
g^*_{3,1}=f_2^*,\quad g^*_{3,2}=(\overline{f_2^*}\cup C1_{C^*_{2,1}})\circ {\omega_{2,1}^*}^{-1}:\Sigma\Sigma X_1\to C^*_{3,2}.
\end{gather*} 
The next relation holds when $r=3$. 
\begin{equation}
g^*_{r,s}\simeq e_{r,s}^{-1}\circ\Sigma g_{r,s}\circ(1_{X_{r-s}}\wedge\tau(\s^{s-1},\s^1))^{-1}:\Sigma^{s-1}\Sigma X_{r-s}\to C^*_{r,s}\ (1\le s<r).
\end{equation}
Indeed  
\begin{align*}
&g^*_{r,s}= (\overline{f_{r-1}^*}^s\cup C1_{C^*_{r-1,s-1}})\circ{\omega^*_{r-1,s-1}}^{-1}\\
&=(\Sigma\overline{f_{r-1}}^s\circ e_{r-1,s}\cup C1_{C^*_{r-1,s-1}})\circ{\omega_{r-1,s-1}^*}^{-1}\\
&=(1_{\Sigma X_r}\cup Ce_{r-1,s-1})^{-1}\circ(\Sigma\overline{f_{r-1}}^s\cup C1_{\Sigma C_{r-1,s-1}})\circ(e_{r-1,s}\cup Ce_{r-1,s-1})\\
&\hspace{1cm}\circ{\omega^*_{r-1,s-1}}^{-1}\\
&\simeq  (1_{\Sigma X_r}\cup Ce_{r-1,s-1})^{-1}\circ(\Sigma\overline{f_{r-1}}^s\cup C1_{\Sigma C_{r-1,s-1}})\circ(e_{r-1,s}\cup Ce_{r-1,s-1})\\
&\hspace{1cm}\circ (e_{r-1,s}\cup Ce_{r-1,s-1})^{-1}\circ(\psi^1_{j_{r-1,s-1}})^{-1}\\
&\hspace{1cm}\circ(\Sigma\omega_{r-1,s-1})^{-1}\circ (1_{X_{r-s}}\wedge\tau(\s^{s-1},\s^1))^{-1}\\
&= (1_{\Sigma X_r}\cup Ce_{r-1,s-1})^{-1}\circ(\Sigma\overline{f_{r-1}}^s\cup C1_{\Sigma C_{r-1,s-1}})\circ(\psi^1_{j_{r-1,s-1}})^{-1}\\
&\hspace{1cm} \circ(\Sigma\omega_{r-1,s-1})^{-1}\circ (1_{X_{r-s}}\wedge\tau(\s^{s-1},\s^1))^{-1}\\
&=(1_{\Sigma X_r}\cup Ce_{r-1,s-1})^{-1}\circ(\psi^1_{\overline{f_{r-1}}^{s-1}})^{-1}\circ \Sigma(\overline{f_{r-1}}^s\cup C1_{C_{r-1,s-1}})\\
&\hspace{1cm}\circ(\Sigma\omega_{r-1,s-1})^{-1}\circ(1_{X_{r-s}}\wedge\tau(\s^{s-1},\s^1))^{-1}\\
&=e_{r,s}^{-1}\circ \Sigma(\overline{f_{r-1}}^s\cup C1_{C_{r-1,s-1}})\circ(\Sigma\omega_{r-1,s-1})^{-1}\circ (1_{X_{r-s}}\wedge\tau(\s^{s-1},\s^1))^{-1}\\
&\simeq e_{r,s}^{-1}\circ\Sigma((\overline{f_{r-1}}^s\cup C1_{C_{r-1,s-1}})\circ\omega_{r-1,s-1}^{-1})\circ(1_{X_{r-s}}\wedge\tau(\s^{s-1},\s^1))^{-1}\\
&=e_{r,s}^{-1}\circ\Sigma g_{r,s}\circ(1_{X_{r-s}}\wedge\tau(\s^{s-1},\s^1))^{-1}.
\end{align*}
Set 
\begin{gather*}
a_{3,s}^*=(e^{-1}_{3,s}\cup C(1_{X_{3-s}}\wedge\tau(\s^{s-1},\s^1)))\circ(\psi^1_{g_{3,s}})^{-1}\circ\Sigma a_{3,s}\circ e_{3,s+1}\\
 : C^*_{3,s+1}\simeq 
C^*_{3,s}\cup_{g^*_{3,s}} C\Sigma^s X_{3-s}\quad (s=1,2).
\end{gather*}
Then $\mathscr{A}_3^*=\{a_{3,1}^*, a_{3,2}^*\}$ is a reduced structure on $\mathscr{S}_3^*$. 
Indeed, $a_{3,1}^*=1_{C^*_{3,2}}$ is obvious and, from the commutative 
diagram below for $r=s=3$, we have $a^*_{3,2}\circ j^*_{3,2}=i_{g^*_{3,2}}$. 
\begin{equation}
\begin{split}
\xymatrix{
C^*_{r,s-1}\ar[d]^-{j^*_{r,s-1}} \ar[r]^-{e_{r,s-1}} & \Sigma C_{r,s-1} \ar[d]^-{\Sigma j_{r,s-1}}\\
C^*_{r,s} \ar[d]^-{a^*_{r,s-1}} \ar[r]^-{e_{r,s}} & \Sigma C_{r,s} \ar[d]^-{\Sigma a_{r,s-1}}\\
C^*_{r,s-1}\cup_{g^*_{r,s-1}}C\Sigma^{s-2}\Sigma X_{r-s+1} & \Sigma(C_{r,s-1}\cup_{g_{r,s-1}}C\Sigma^{s-2}X_{r-s+1})\ar[d]^-{(\psi^1_{g_{r,s-1}})^{-1}}\\
& \Sigma C_{r,s-1}\cup C\Sigma\Sigma^{s-2}X_{r-s+1} \ar[ul]^-{e^{-1}_{r,s-1}\cup 
C(1_{X_{r-s+1}}\wedge\tau(\s^{s-2},\s^1))}
}
\end{split}
\end{equation}
Hence $\mathscr{S}_3^*$ is an iterated mapping cone with a reduced structure $\mathscr{A}_3^*$ and 
a reduced quasi-structure $\Omega(\mathscr{A}_3^*)=\{\omega_{3,s}^*\,|\,s=1,2\}$, where $\omega_{3,s}^*= q'_{g^*_{3,s}}\circ(a^*_{3,s}\cup C1_{C^*_{3,s}})$. 
We are going to prove 
\begin{align*}
\omega^*_{3,s}&=(1_{X_{3-s}}\wedge\tau(\s^{s-1}\wedge\s^1,\s^1))\circ\Sigma\omega_{3,s}\circ\psi^1_{j_{3,s}}\circ(e_{3,s+1}\cup Ce_{3,s})\\
&:C^*_{3,s+1}\cup CC^*_{3,s}\to \Sigma \Sigma^{s-1}\Sigma X_{3-s}\quad(s=1,2),
\end{align*}
that is, we are going to prove that the following diagram is commutative when $r=3$. 
\begin{equation}
\begin{split}
\xymatrix{
C^*_{r,s+1}\cup CC^*_{r,s}\ar[d]_-{a^*_{r,s}\cup C1_{C^*_{r,s}}}\ar[r]^-{e_{r,s+1}\cup Ce_{r,s}}&\Sigma C_{r,s+1}\cup C\Sigma C_{r,s}\ar[d]_-{\psi^1_{j_{r,s}}}\\
(C^*_{r,s}\cup_{g^*_{r,s}}C\Sigma^{s-1}\Sigma X_{r-s})\cup CC^*_{r,s}\ar[d]_-{q'_{g^*_{r,s}}} &\Sigma(C_{r,s+1}\cup CC_{r,s})\ar[d]_-{\Sigma\omega_{r,s}}\\
(\Sigma\Sigma^{s-1})\Sigma X_{r-s}&\Sigma(\Sigma\Sigma^{s-1})X_{r-s}\ar[l]^-{1_{X_{r-s}}\wedge\tau(\s^{s-1}\wedge\s^1,\s^1)}
}
\end{split}
\end{equation}
Recall that 
\begin{align*}
&\omega_{r,s}=q'_{g_{r,s}}\circ(a_{r,s}\cup C1_{C_{r,s}})\\
&:C_{r,s+1}\cup_{j_{r,s}} CC_{r,s}\to( C_{r,s}\cup_{g_{r,s}}C\Sigma^{s-1}X_{r-s})\cup_{i_{g_{r,s}}} CC_{r,s}\to \Sigma\Sigma^{s-1}X_{r-s}.
\end{align*} 
Since $CC^*_{r,s}$ is mapped finally to $*$ by the maps in the last diagram, it suffices to show that the two maps 
from $C^*_{r,s+1}$ to $(\Sigma\Sigma^{s-1})\Sigma X_{r-s}$ are the same, that is, 
$q_{g^*_{r,s}}\circ a^*_{r,s}=(1_{X_{r-s}}\wedge\tau(\s^{s-1}\wedge\s^1,\s^1))\circ\Sigma q_{g_{r,s}}\circ \Sigma a_{r,s}\circ e_{r,s+1}$. 
Since
$$
a^*_{r,s}=(e^{-1}_{r,s}\cup(1_{X_{r-s}}\wedge\tau(\s^{s-1},\s^1))\circ(\psi^1_{g_{r,s}})^{-1}\circ\Sigma a_{r,s}\circ e_{r,s+1}
$$
by the definition, it suffices to show 
\begin{align*}
&q_{g_{r,s}^*}\circ(e^{-1}_{r,s}\cup C(1_{X_{r-s}}\wedge\tau(\s^{s-1},\s^1))\circ(\psi^1_{g_{r,s}})^{-1}\\
&=
(1_{X_{r-s}}\wedge\tau(\s^{s-1}\wedge\s^1,\s^1))\circ\Sigma q_{g_{r,s}}\\
&:\Sigma(C_{r,s}\cup_{g_{r,s}} C\Sigma^{s-1}X_{r-s})\to\Sigma\Sigma^{s-1}\Sigma X_{r-s}.
\end{align*}
The last equality is proved easily. 
Hence (6.3.4) is commutative when $r=3$. 

When $n=3$, $\{\mathscr{S}_r^*,\overline{f_r^*},\mathscr{A}_r^*\,|\,r=2,3\}$ is an $\ddot{s}_t$-presentation of $\Sigma\vec{\bm f}$ and  
\begin{align*}
\overline{f_3^*}\circ g^*_{3,2}&= \Sigma\overline{f_3}\circ e_{3,2}\circ g^*_{3,2}\\
&\simeq \Sigma\overline{f_3}\circ e_{3,2}\circ e_{3,2}^{-1}\circ \Sigma g_{3,2}\circ (1_{X_1}\wedge\tau(\s^1,\s^1))^{-1}\ (\text{by (6.3.2)})\\
&=\Sigma(\overline{f_3}\circ g_{3,2})\circ (1_{X_1}\wedge\tau(\s^1,\s^1))^{-1}. 
\end{align*}
Hence $\Sigma\{\vec{\bm f}\,\}^{(\ddot{s}_t)}\subset\{\Sigma\vec{\bm f}\}^{(\ddot{s}_t)}\circ(1_{X_1}\wedge\tau(\s^1,\s^1))$. 

When $n\ge 4$, we set $\mathscr{S}_4^*=\mathscr{S}^*_3(\overline{f_3^*},\mathscr{A}^*_3)$ and 
\begin{gather*}
e_{4,s}=\begin{cases} 1_{\{*\}} & s=0\\ \psi^1_{\overline{f_3}^{s-1}}\circ (1_{\Sigma X_4}\cup Ce_{3,s-1}) & s=1,2,3,4\end{cases}:C^*_{4,s}\approx \Sigma C_{4,s},\\
\overline{f_4^*}=\begin{cases} \Sigma\overline{f_4}\circ e_{4,3}:C_{4,3}^*\to \Sigma X_5 & n=4\\
 \Sigma\overline{f_4}\circ e_{4,4}:C^*_{4,4}\to \Sigma X_5 & n\ge 5\end{cases}.
\end{gather*}
Then $\mathscr{S}_4^*$ is an iterated mapping cone, $\overline{f_4^*}$ is an extension of $f_4^*$ to 
$C^*_{4,3}$ or $C_{4,4}^*$ according as $n=4$ or $n\ge 5$, and 
\begin{gather*}
e_{4,1}=1_{\Sigma X_4},\quad e_{4,s+1}\circ j^*_{4,s}=\Sigma j_{4,s}\circ e_{4,s}\ (s=0,1,2,3),\\
\overline{f^*_4}^s=\Sigma\overline{f_4}^s\circ e_{4,s}\ (s=1,2,3),\\
g^*_{4,s}=\begin{cases} f_3^* & s=1\\(\overline{f_3^*}^s\cup C1_{C^*_{3,s-1}})\circ{\omega_{3,s-1}^*}^{-1} & s=2,3\end{cases}: 
\Sigma\Sigma^{s-2}\Sigma X_{4-s}\to C^*_{4,s}
\end{gather*}
We can prove 
$$
g^*_{4,s}\simeq e_{4,s}^{-1}\circ\Sigma g_{4,s}\circ (1_{X_{4-s}}\wedge\tau(\s^{s-1},\s^1))^{-1}\ (s=1,2,3)
$$
by the method which was used to prove (6.3.2) for $r=3$.
Set 
\begin{align*}
a_{4,s}^*&=(e_{4,s}^{-1}\cup C(1_{X_{4-s}}\wedge\tau(\s^{s-1},\s^1))\circ(\psi^1_{g_{4,s}})^{-1}\circ\Sigma a_{4,s}\circ e_{4,s+1}\\
&:C^*_{4,s+1}\to C^*_{4,s}\cup_{g^*_{4,s}}C\Sigma^s X_{4-s}\ (s=1,2,3).
\end{align*}
Then the diagrams (6.3.3) and (6.3.4) are commutative for $r=4$, that is, $\mathscr{A}_4^*=\{a^*_{4,s}\,|\,1\le s\le 3\}$ is 
a reduced structure on $\mathscr{S}_4^*$ and, if we set $\Omega(\mathscr{A}_4^*)=\{\omega^*_{4,s}\,|\,1\le s\le 3\}$, then  
\begin{align*}
\omega_{4,s}^*&=q'_{g_{4,s}^*}\circ(a^*_{4,s}\cup C1_{C^*_{4,s}})\\
&=(1_{X_{4-s}}\wedge\tau(\s^{s-1}\wedge\s^1,\s^1))\circ\Sigma\omega_{4,s}\circ\psi^1_{j_{4,s}}\circ(e_{4,s+1}\cup Ce_{4,s}).
\end{align*} 

When $n=4$, $\{\mathscr{S}_r^*,\overline{f_r^*},\mathscr{A}^*_r\,|\,2\le r\le 4\}$ is an $\ddot{s}_t$-presentation of 
$\Sigma\vec{\bm f}$ and 
\begin{align*}
\overline{f_4^*}\circ g^*_{4,3}&\simeq \Sigma\overline{f_4}\circ e_{4,3}\circ e^{-1}_{4,3}\circ\Sigma g_{4,3}\circ(1_{X_1}\wedge\tau(\s^2,\s^1))^{-1} \\
&=\Sigma(\overline{f_4}\circ g_{4,3})\circ(1_{X_1}\wedge\tau(\s^2,s^1))^{-1}.
\end{align*}
Hence $\Sigma\{\vec{\bm f}\,\}^{(\ddot{s}_t)}\subset\{\Sigma\vec{\bm f}\,\}^{(\ddot{s}_t)}\circ(1_{X_1}\wedge\tau(\s^2,\s^1))$. 

By repeating the above process, we have an $\ddot{s}_t$-presentation 
$\{\mathscr{S}_r^*,\overline{f_r^*},\mathscr{A}_r^*|2\le r\le n\}$ of $\Sigma\vec{\bm f}$ 
such that $\overline{f_n^*}\circ g^*_{n,n-1}\simeq \Sigma(\overline{f_n}\circ g_{n,n-1})\circ(1_{X_1}\wedge\tau(\s^{n-2},\s^1))^{-1}$ 
so that $\Sigma\{\vec{\bm f}\,\}^{(\ddot{s}_t)}\subset\{\Sigma\vec{\bm f}\,\}^{(\ddot{s}_t)}\circ(1_{X_1}\wedge\tau(\s^{n-2},\s^1))$. 
This completes the proof of Theorem~6.3.1 for $\star=\ddot{s}_t$. 
 \end{proof}

\subsection{Homotopy invariance of higher Toda brackets}

We prove the following which is (1.6) (cf. \cite[Theorem~3.4]{G} for $\{\vec{\bm f}\,\}^{(q)}$) 
and allows us to use the notation $\{\vec{\bm \alpha}\,\}^{(\star)}$ instead of $\{\vec{\bm f}\,\}^{(\star)}$ 
for every $\star$. 

\begin{thm'}%Theorem 6.4.1
If $\vec{\bm f}, \vec{\bm f'}\in\mathrm{Rep}(\vec{\bm \alpha})$, then $\{\vec{\bm f}\,\}^{(\star)}=\{\vec{\bm f'}\}^{(\star)}$ for all $\star$. 
\end{thm'}

For $\vec{\bm f}=(f_n,\dots,f_1)\in\mathrm{Rep}(\vec{\bm \alpha})$ and $i\in\{1,2,\dots,n\}$, 
let $\vec{\bm f}_i\in\mathrm{Rep}(\vec{\bm \alpha})$ 
denote a sequence obtained from $\vec{\bm f}$ by replacing $f_i$ with $f_i'$ such that $f_i'\simeq f_i$, 
for example $\vec{\bm f}_2=(f_n,\dots,f_3,f_2',f_1)$ with $f_2'\simeq f_2$. 

\begin{lemma'}%Lemma 6.4.2
If $\vec{\bm f}\in\mathrm{Rep}(\vec{\bm\alpha})$, then $\{\vec{\bm f}\,\}^{(\star)}=\{\vec{\bm f}_i\,\}^{(\star)}$ 
for all $\star$ and $i$. 
\end{lemma'}

From the lemma, the theorem is proved as follows: 
\begin{align*}
\{\vec{\bm f}\,\}^{(\star)}&=\{f_n,\dots,f_2,f_1'\}^{(\star)}
=\{f_n,\dots,f_3,f_2',f_1'\}^{(\star)}=\dots\\
&=\{f_n',\dots,f_1'\}^{(\star)}=\{\vec{\bm f}'\}^{(\star)}.
\end{align*}

\begin{proof}[Proof of Lemma 6.4.2] 
By Theorem 6.2.1, it suffices to prove the lemma for the cases $\star=\ddot{s}_t, aq\ddot{s}_2, aq, q_2$. 
We consider only the case of $\star=\ddot{s}_t$, because other cases can be treated similarly or more easily. 
For simplicity we abbreviate $\{\ \}^{(\ddot{s}_t)}$ as $\{\ \}$. 
Let $\alpha\in\{\vec{\bm f}\,\}$, $\{\mathscr{S}_r,\overline{f_r},\mathscr{A}_r\,|\,2\le r\le n\}$ 
an $\ddot{s}_t$-presentation of $\vec{\bm f}=(f_n,\dots,f_1)$ 
such that $\alpha=\overline{f_n}\circ g_{n,n-1}$, $\mathscr{A}_r=\{a_{r,s}\,|\,1\le s<r\}$ with $a_{r,1}=1_{C_{r,2}}$, 
and $\Omega(\mathscr{A}_r)=\{\omega_{r,s}\,|\,1\le s<r\}$. 
We are going to construct an $\ddot{s}_t$-presentation $\{\mathscr{S}_r',\overline{f_r}',\mathscr{A}_r'\,|\,2\le r\le n\}$ of 
$\vec{\bm f}_i$ with $\overline{f_n}'\circ g_{n,n-1}'=\alpha$. 
If this is done, then $\{\vec{\bm f}\,\}\subset\{\vec{\bm f}_i\}$, and by interchanging $\vec{\bm f}$ with $\vec{\bm f}_i$ 
each other we have $\{\vec{\bm f}_i\}\subset \{\vec{\bm f}\,\}$ so that $\{\vec{\bm f}\,\}=\{\vec{\bm f}_i\}$. 

We divide the proof into three cases: $i=n$; $i=1$; $2\le i\le n-1$. 

First we consider the case: $i=n$. 
Let $\vec{\bm f}_n=(f_n',f_{n-1},\dots,f_1)$ with $J^n:f_n\simeq f_n'$ and set $j=j_{n,n-2}\circ \cdots\circ j_{n,2}\circ j_{n,1}$. 
Since $j$ is a free cofibration, there exists a map $H:C_{n,n-1}\times I\to X_{n+1}$ such that 
$H\circ i_0^{C_{n,n-1}}=\overline{f_n}$ and $H\circ(j\times 1_I)=J^n$. 
Let $\{\mathscr{S}_r',\overline{f_r}',\mathscr{A}_r'\,|\,2\le r\le n\}$ be the 
collection obtained from $\{\mathscr{S}_r,\overline{f_r},\mathscr{A}_r\,|\,2\le r\le n\}$ by replacing 
$\overline{f_n}$ with $H_1$. 
Then the new collection is an $\ddot{s}_t$-presentation of $\vec{\bm f}_n$ such that it represents 
$\alpha=H_1\circ g_{n,n-1}\in\{\vec{\bm f}_n\}$. 
Hence $\{\vec{\bm f}\,\}=\{\vec{\bm f}_n\}$. 

Secondly we consider the case: $i=1$. 
Let $\vec{\bm f}_1=(f_n,\dots,f_2,f_1')$ with $J^1:f_1\simeq f_1'$. 
Set 
\begin{gather*}
\mathscr{S}_2'=(X_1;X_2,X_2\cup_{f_1'}CX_1;f_1';i_{f_1'}),\\
e_{2,2}=\Phi(f_1',f_1,1_{X_1},1_{X_2};-J^1):C_{2,2}'=X_2\cup_{f_1'}CX_1\to C_{2,2}=X_2\cup_{f_1}CX_1, \\
e_{2,1}=1_{X_2},\quad \overline{f_2}'=\overline{f_2}\circ e_{2,2}:C_{2,2}'\to X_3, \quad a_{2,1}'=1_{C_{2,2}'}. 
\end{gather*}
Then $e_{2,2}\circ j_{2,1}'=j_{2,1}$, and 
$\omega_{2,1}'\simeq e_{1_{X_1}}\circ\omega_{2,1}'=\omega_{2,1}\circ(e_{2,2}\cup Ce_{2,1})$ 
by Proposition~3.3(2). 
Set $\mathscr{S}_3'=\mathscr{S}_2'(\overline{f_2}',\mathscr{A}_2')$. 
Then $C_{3,s}'=C_{3,s}$ for $s=1,2$ and $j_{3,1}'=j_{3,1}$. 
Set 
\begin{align*}
e_{3,3}&=1_{X_3}\cup Ce_{2,2}: C_{3,3}'=X_3\cup_{\overline{f_2}'}CC_{2,2}'\to C_{3,3}=X_3\cup_{\overline{f_2}}CC_{2,2},\\
e_{3,s}&=1_{C_{3,s}}\ (s=1,2),\quad  \overline{f_3}'=\begin{cases} \overline{f_3}\circ e_{3,2} & n=3\\ \overline{f_3}\circ e_{3,3} & n\ge 4\end{cases}.
\end{align*}
We have $e_{3,3}\circ j_{3,2}'=j_{3,2}\circ e_{3,2}$, $g_{3,1}'=g_{3,1}$, and 
\begin{align*}
g_{3,2}'&=(\overline{f_2}'\cup C1_{X_2})\circ\omega_{2,1}'^{-1}=(\overline{f_2}\cup C1_{X_2})\circ (e_{2,2}\cup C1_{X_2})\circ\omega_{2,1}'^{-1}\\
&\simeq (\overline{f_2}\cup C1_{X_2})\circ\omega_{2,1}^{-1}=g_{3,2}.
\end{align*}
Take $K^3:g_{3,2}\simeq g_{3,2}'$ and set 
\begin{gather*}
\Phi(K^3)=\Phi(g_{3,2},g_{3,2}',1_{\Sigma X_1},1_{C_{3,2}};K^3):C_{3,2}\cup_{g_{3,2}}C\Sigma X_1\to C_{3,2}'
\cup_{g_{3,2}'}C\Sigma X_1,\\
a_{3,2}'=\Phi(K^3)\circ a_{3,2}\circ e_{3,3}:C_{3,3}'\to C_{3,2}'\cup_{g_{3,2}'}C\Sigma X_1. 
\end{gather*}
Then $\mathscr{A}_3'=\{a_{3,1},a_{3,2}'\}$ is a reduced structure on $\mathscr{S}_3'$. 
We have $\omega_{3,1}'=\omega_{3,1}=\omega_{3,1}\circ(e_{3,2}\cup Ce_{3,1})$ and 
\begin{align*}
\omega_{3,2}'&=q'_{g_{3,2}'}\circ (a_{3,2}'\cup C1_{C_{3,2}'})\\
&=q'_{g_{3,2}'}\circ (\Phi(K^3)\cup C1_{C_{3,2}})\circ(a_{3,2}\cup C1_{C_{3,2}})\circ (e_{3,3}\cup Ce_{3,2})\\
&\simeq q'_{g_{3,2}}\circ(a_{3,2}\cup C1_{C_{3,2}})\circ (e_{3,3}\cup Ce_{3,2})\quad (\text{by Proposition 3.3(2)})\\
&=\omega_{3,2}\circ (e_{3,3}\cup Ce_{3,2}).
\end{align*}
When $n=3$, $\{\mathscr{S}_r',\overline{f_r}',\mathscr{A}_r'\,|\,r=2,3\}$ is an $\ddot{s}_t$-presentation of 
$(f_3,f_2,f_1')$ such that $\overline{f_3}'\circ g_{3,2}'\simeq \overline{f_3}\circ g_{3,2}$ and so 
$\{\vec{\bm f}\,\}=\{\vec{\bm f_1}\,\}$. 
When $n\ge 4$, set $\mathscr{S}_4'=\mathscr{S}_3'(\overline{f_3}',\mathscr{A}_3')$. 
By repeating the above process we obtain $\{\mathscr{S}_r',\overline{f_n}',\mathscr{A}_r'\,|\,2\le r\le n\}$, 
an $\ddot{s}_t$-presentation of $\vec{\bm f_1}$, and $e_{r,s}:C_{r,s}'\simeq C_{r,s}$ such that  
\begin{gather*}
C_{n,s}'=C_{n,s},\ e_{n,s}=1_{C_{n,s}}\ (1\le s\le n-1);\\
a_{n,s}'=a_{n,s},\ \omega_{n,s}'=\omega_{n,s}\ (1\le s\le n-2);\\
\omega_{n,n-1}'\simeq\omega_{n,n-1}\circ (e_{n,n}\cup Ce_{n,n-1});\\
\overline{f_n}'=\overline{f_n}:C_{n,n-1}'=C_{n,n-1}\to X_{n+1}; \quad 
g_{n,n-1}'\simeq g_{n,n-1}
\end{gather*}
so that $\overline{f_n}'\circ g_{n,n-1}'\simeq \overline{f_n}\circ g_{n,n-1}$ and 
hence $\{\vec{\bm f}\,\}=\{\vec{\bm f_1}\,\}$. 

Thirdly we consider the case: $2\le i\le n-1$. 
We prove only the case $i=2$, because other cases can be proved similarly. 
Let $\vec{\bm f}_2=(f_n,\dots,f_3,f_2',f_1)$ with $J^2:f_2\simeq f_2'$. 

Step 1: Set $\mathscr{S}_2'=\mathscr{S}_2$ and $\mathscr{A}_2'=\mathscr{A}_2$. 
Then $C_{2,s}'=C_{2,s}$, $j_{2,1}'=j_{2,1}=i_{f_1}$, and $\omega_{2,1}'=\omega_{2,1}=q_{f_1}'$. 
Set $e_{2,s}=1_{C_{2,s}}$ for $s=1,2$. 
Since $j_{2,1}=i_{f_1}:X_2\to C_{2,2}=X_2\cup_{f_1}CX_1$ is a free cofibration, there exists $H^2:C_{2,2}\times I\to X_3$ 
such that $H^2\circ i_0^{C_{2,2}}=\overline{f_2}$ and $H^2\circ (i_{f_1}\times 1_I)=J^2$. 
Set $\overline{f_2}'=H^2_1:C_{2,2}'=C_{2,2}\to X_3$. 

Step 2: Set 
\begin{align*}
&\mathscr{S}_3'=\mathscr{S}_2'(\overline{f_2}',\mathscr{A}_2'),\quad a_{3,1}'=1_{C_{3,2}'},\\
&e_{3,3}=\Phi(\overline{f_2}',\overline{f_2},1_{C_{2,2}},1_{X_3};-H^2):C_{3,3}'=X_3\cup_{\overline{f_2}'}CC_{2,2}'
\to \\
&\hspace{3cm}C_{3,3}=X_3\cup_{\overline{f_2}}CC_{2,2},\\
&e_{3,2}=\Phi(f_2',f_2,1_{X_2},1_{X_3};-J^2):C_{3,2}'=X_3\cup_{f_2'}CX_2\to C_{3,2}=X_3\cup_{f_2}CX_2,\\
&e_{3,1}=1_{X_3}:C_{3,1}'\to C_{3,1},\quad 
\overline{f_3}'=\begin{cases} \overline{f_3}\circ e_{3,2}:C_{3,2}'\to X_4 & n=3\\ \overline{f_3}\circ e_{3,3}:C_{3,3}'\to X_4 & n\ge 4\end{cases}.
\end{align*}
Then $e_{3,3}\circ j_{3,2}'=j_{3,2}\circ e_{3,2}$, $e_{3,2}\circ j_{3,1}'=j_{3,1}\circ e_{3,1}$, and $e_{3,1}\circ g_{3,1}'=f_2'\simeq f_2=g_{3,1}$. 
We will prove 
$$
e_{3,2}\circ g_{3,2}'\simeq g_{3,2}\quad  i.e.\quad 
e_{3,2}\circ(\overline{f_2}'\cup C1_{X_2})\circ\omega_{2,1}'^{-1}\simeq (\overline{f_2}\cup C1_{X_2})\circ\omega_{2,1}^{-1}.
$$
It suffices to prove $e_{3,2}\circ(\overline{f_2}'\cup C1_{X_2})\simeq \overline{f_2}\cup C1_{X_2}$, 
since $\omega_{2,1}'=\omega_{2,1}$. 
We have
\begin{align*}
&e_{3,2}\circ (\overline{f_2}'\cup C1_{X_2})\\
&\simeq\Phi (f_2',f_2,1_{X_2},1_{X_3};-J^2)\circ \Phi (i_{f_1},f_2',1_{X_2},\overline{f_2}';1_{f_2'})\ (\text{by Proposition 3.3(3)})\\
&\simeq\Phi (i_{f_1},f_2,1_{X_2},\overline{f_2}';((-J^2)\bar\circ 1_{1_{X_2}})\bullet(1_{1_{X_3}}\bar\circ 1_{f_2'}))\ 
(\text{by Proposition 3.3(1)(d)})\\
&\simeq \Phi(i_{f_1},f_2,1_{X_2},\overline{f_2}';-J^2)\\
&\hspace{2cm} (\text{by Proposition 3.3(5) 
and $((-J^2)\bar\circ 1_{1_{X_2}})\bullet(1_{1_{X_3}}\bar\circ 1_{f_2'})\overset{X_2}{\simeq}-J^2$}).
\end{align*}
We define $\widetilde{J}^t: X_2\times I\to X_3$ for $t\in I$ by $\widetilde{J}^t(x_2,u)=(-J^2)(x_2,t+u-tu)$. 
Then $\widetilde{J}^t:(-H^2)_t\circ i_{f_1}\simeq f_2\circ 1_{X_2}$ and so 
\begin{align*}
\Phi(i_{f_1},f_2,1_{X_2},\overline{f_2}';-J^2)&\simeq\Phi(i_{f_1},f_2,1_{X_2},\overline{f_2};1_{f_2})\ (\text{by Proposition 3.3(4)})\\
&\simeq \overline{f_2}\cup C1_{X_2}\ (\text{by Proposition 3.3(3)}).
\end{align*}
Therefore $e_{3,2}\circ(\overline{f_2}'\cup C1_{X_2})\simeq \overline{f_2}\cup C1_{X_2}$ as desired. 

Let $e_{3,2}^{-1}:C_{3,2}\to C_{3,2}'$ be a homotopy inverse of $e_{3,2}$. 
Take $N:e_{3,2}^{-1}\circ e_{3,2}\simeq 1_{C_{3,2}'}$ and set $L=1_{i_{g_{3,2}'}}\bar\circ N:i_{g_{3,2}'}\circ e_{3,2}^{-1}\circ e_{3,2}\simeq i_{g_{3,2}'}$. 
Take $K^{3,2}:e^{-1}_{3,2}\circ g_{3,2}\simeq g_{3,2}'$ and set 
$\Phi(K^{3,2})=\Phi(g_{3,2},g_{3,2}',1_{\Sigma X_1},e^{-1}_{3,2};K^{3,2})$ and 
$$
a_{3,2}''=\Phi(K^{3,2})\circ a_{3,2}\circ e_{3,3}
:C_{3,3}'\to C_{3,2}'\cup_{g_{3,2}'}C\Sigma X_1.
$$
Then $L:a_{3,2}''\circ j_{3,2}'=i_{g_{3,2}'}\circ e_{3,2}^{-1}\circ e_{3,2}\simeq i_{g_{3,2}'}$. 
Since $j_{3,2}'$ is a free cofibration, there exists a homotopy $M:C_{3,3}'\times I\to C_{3,2}'\cup_{g_{3,2}'}C\Sigma X_1$ 
such that $M_0=a_{3,2}''$ and $M\circ (j_{3,2}'\times 1_I)=L$. 
Set $a_{3,2}'=M_1$, $\mathscr{A}_3'=\{1_{C_{3,2}'}, a_{3,2}'\}$, and $\Omega_3'=\Omega(\mathscr{A}_3')$. 
Then $\mathscr{A}_3'$ is a reduced structure on $\mathscr{S}_3'$. 
We will prove
$$
\omega_{3,s}'\simeq\omega_{3,s}\circ(e_{3,s+1}\cup Ce_{3,s})\ (s=1,2).
$$
Since $\omega_{3,1}\circ(e_{3,2}\cup Ce_{3,1})=e_{1_{X_1}}\circ\omega_{3,1}'\simeq\omega_{3,1}'$ by Proposition 3.3(2), 
the case of $s=1$ holds. 
By the definition, we have $M_t\circ j_{3,2}'=L_t=i_{g_{3,2}'}\circ N_t$ for every $t\in I$. 
Set $H^t=1_{M_t\circ j_{3,2}'}:M_t\circ j_{3,2}'=i_{g_{3,2}'}\circ N_t$. 
Since the function $C_{3,2}'\times I\times I\to C_{3,2}'\cup_{g_{3,2}'} C\Sigma X_1,\ (z,s,t)\mapsto H^t(z,s)=M_t(j_{3,2}'(z))$, 
is continuous, it follows from Proposition 3.3(4) that 
$\Phi(j_{3,2}',i_{g_{3,2}'},N_0,M_0;H^0)\simeq \Phi(j_{3,2}',i_{g_{3,2}'},N_1,M_1;H^1)$. 
By Proposition 3.3(3), we have 
\begin{gather*}
\Phi(j_{3,2}',i_{g_{3,2}'},N_0,M_0;H^0)\simeq M_0\cup CN_0=a_{3,2}''\cup C(e_{3,2}^{-1}\circ e_{3,2}),\\
\Phi(j_{3,2}',i_{g_{3,2}'},N_1,M_1;H^1)\simeq M_1\cup CN_1=a_{3,2}'\cup C1_{C_{3,2}'}
\end{gather*}
and so 
\begin{align*}
\omega_{3,2}'&=q'_{g_{3,2}'}\circ (a_{3,2}'\cup C1_{C_{3,2}'})\simeq q'_{g_{3,2}'}\circ (a_{3,2}''\cup C(e_{3,2}^{-1}\circ e_{3,2}))\\
&=q'_{g_{3,2}'}\circ(\Phi(K^{3,2})\cup Ce_{3,2}^{-1})\circ (a_{3,2}\cup C1_{C_{3,2}})\circ(e_{3,3}\cup Ce_{3,2})\\
&\simeq q'_{g_{3,2}}\circ (a_{3,2}\circ C1_{C_{3,2}})\circ (e_{3,3}\cup Ce_{3,2})\quad(\text{by Proposition 3.3(2)})\\
&=\omega_{3,2}\circ (e_{3,3}\cup Ce_{3,2}).
\end{align*}

Step 3: When $n=3$, $\{\mathscr{S}_r',\overline{f_r}',\mathscr{A}_r'\,|\,r=2,3\}$ is 
an $\ddot{s}_t$-presentation of $(f_3,f_2',f_1)$ such that 
$\overline{f_3}'\circ g_{3,2}'=\overline{f}_3\circ e_{3,2}\circ g_{3,2}'\simeq \overline{f_3}\circ g_{3,2}$ 
so that $\{\vec{\bm f}\,\}=\{\vec{\bm f_2}\,\}$. 
When $n\ge 4$, by repeating the above process, we have 
$\mathscr{S}_r'$, $\overline{f_r}'$, $\mathscr{A}_r'$, 
and $e_{r,s}:C_{r,s}'\simeq C_{r,s}$ for $4\le r\le n$ such that 
$\mathscr{A}_r'$ is a reduced structure on $\mathscr{S}_r'$ and 
$$
\left\{\begin{array}{@{\hspace{0.6mm}}l}
\mathscr{S}_r'=\mathscr{S}_{r-1}'(\overline{f_{r-1}}',\mathscr{A}_{r-1}'),\\ 
e_{r,1}=1_{X_r},\ 
e_{r,s}=1_{X_r}\cup Ce_{r-1,s-1}:C_{r,s}'\simeq C_{r,s}\ (2\le s\le r),\\
e_{r,s+1}\circ j_{r,s}'=j_{r,s}\circ e_{r,s},\ \omega_{r,s}'\simeq \omega_{r,s}\circ(e_{r,s+1}\cup Ce_{r,s}),\\
\hspace{2cm} e_{r,s}\circ g_{r,s}'\simeq g_{r,s}\ (1\le s<r),\\
\overline{f_r}'=\begin{cases} \overline{f_n}\circ e_{n,n-1}:C_{n,n-1}'\to X_{n+1} & r=n
\\ \overline{f_r}\circ e_{r,r}:C_{r,r}'\to X_{r+1} & r<n\end{cases}.
\end{array}\right.
$$
Then $\{\mathscr{S}_r',\overline{f_r}', \mathscr{A}_r'\,|\,2\le r\le n\}$ is an $\ddot{s}_t$-presentation 
of $\vec{{\bm f}_2}$ such that 
$\overline{f_n}'\circ g_{n,n-1}'=\overline{f_n}\circ e_{n,n-1}\circ g_{n,n-1}'\simeq \overline{f_n}\circ g_{n,n-1}$. 
Therefore $\{\vec{\bm f}\}=\{\vec{{\bm f}_2}\}$. 
\end{proof}

\subsection{Relations with the J. Cohen's higher Toda bracket}

Let $\langle\vec{\bm f}\,\rangle$ be the Cohen's $n$-fold bracket of $\vec{\bm f}$ which shall be recalled 
and denoted by $\langle\vec{\bm f}\,\rangle^w$ in Appendix~B. 
The purpose of this subsection is to prove the following which is (1.7). 

\begin{thm'}%6.5.1
$\{\vec{\bm f}\,\}^{(aq\ddot{s}_2)}\cup\{\vec{\bm f}\,\}^{(\ddot{s}_t)}\subset\langle \vec{\bm f}\,\rangle$. 
\end{thm'}

Let $\alpha\in \{\vec{\bm f}\,\}^{(aq\ddot{s}_2)}\cup\{\vec{\bm f}\,\}^{(\ddot{s}_t)}$. 
Let $\{\mathscr{S}_r,\overline{f_r},\Omega_r|2\le r\le n\}$ and 
$\{\mathscr{S}_r,\overline{f_r},\mathscr{A}_r\, |\,2\le r\le n\}$ be 
an $aq\ddot{s}_2$-presentation of $\vec{\bm f}$ and an $\ddot{s}_t$-presentation of $\vec{\bm f}$ 
with $\alpha=\overline{f_n}\circ g_{n,n-1}$ according as $\alpha\in\{\vec{\bm f}\,\}^{(aq\ddot{s}_2)}$ 
or $\alpha\in\{\vec{\bm f}\,\}^{(\ddot{s}_t)}$. 
We prepare two lemmas. 

\begin{lemma'}%6.5.2
$C_{n,n-1}$ is a finitely filtered space of type $(f_{n-1},\dots,f_2)$ (see Appendix B for the definition).
\end{lemma'}
\begin{proof}
We can assume that the free cofibration $j_{n,s} : C_{n,s}\to C_{n,s+1}$ is an inclusion map.  
Then by setting $F_k=C_{n,k}\,(0\le k\le n-1)$, where $C_{n,0}=\{*\}$, we obtain the filtration of $C_{n,n-1}$: 
$F_0=\{*\}\subset F_1\subset F_2\subset \cdots\subset F_{n-1}=C_{n,n-1}$. 
By Proposition~6.1.4 and Lemma B.5, we have the canonical homeomorphism 
$$
F_{k+1}/F_k=C_{n,k+1}/C_{n,k}\approx \Sigma^k X_{n-k}\quad (0\le k\le n-2).
$$
By this homeomorphism, we identify $F_{k+1}/F_k$ with $\Sigma^kX_{n-k}$. 
We set 
$$
g_k=(-1)^k 1_{\Sigma^kX_{n-k}}: \Sigma^kX_{n-k}\to C_{n,k+1}/C_{n,k}=\Sigma^kX_{n-k}.
$$ 
The assertion we must prove is that the skew pentagon of the following diagram is homotopy commutative for $1\le k\le n-2$.   
$$
\xymatrix{
\Sigma^kX_{n-k+1} \ar@/_6mm/[dd]_{\Sigma g_{k-1}} &  \Sigma^kX_{n-k} \ar[d]_-{g_k} \ar@/_/[l]_-{\Sigma^k f_{n-k}} &   \\
\Sigma F_k \ar[d]^-{\Sigma(/F_{k-1})}  & F_{k+1}/F_k \ar@{=}[d] \ar[l]_-\delta & F_{k+1} \ar[d]_-{/F_{k-1}} \ar[l]_-q 
& F_k \ar[d]_-{/F_{k-1}} \ar[l]_-{j_{n,k}}  \\
\Sigma^k X_{n-k+1}  & \Sigma^kX_{n-k} \ar@/^/[l]^-{-\Sigma^kf_{n-k}} & \Sigma^{k-1}C_{f_{n-k}}  
\ar@/^/[l]^-{\Sigma^{k-1}q_{f_{n-k}}} & \Sigma^{k-1}X_{n-k+1} \ar@/^/[l]^-{\Sigma^{k-1}i_{f_{n-k}}} 
}
$$
Since lower two rows are cofibre sequences and three squares are homotopy commutative, we have 
$\Sigma(/F_{k-1})\circ\delta\circ g_k\simeq (-\Sigma^k f_{n-k})\circ g_k=(-1)^{k+1}\Sigma^k f_{n-k}
=(-1)^{k-1}\Sigma^k f_{n-k}=\Sigma g_{k-1}\circ \Sigma^k f_{n-k}$. 
Hence the skew pentagon of the above diagram is homotopy commutative. 
This proves Lemma~6.5.2. 
\end{proof}

\begin{lemma'}%6.5.3
The next square is homotopy commutative for $2\le r\le n$, where $q_r$ is the quotient. 
$$
\xymatrix{
\Sigma^{r-2}X_1 \ar[d]_-{\Sigma^{r-2}f_1} \ar[r]^-{g_{r,r-1}} & C_{r,r-1} \ar[d]^-{q_r=/C_{r,r-2}}\\
\Sigma^{r-2}X_2 \ar[r]^-{(-1)^r} & \Sigma^{r-2}X_2=C_{r,r-1}/C_{r,r-2}
}
$$
\end{lemma'}
\begin{proof}
For $r=2$, the square is commutative. 
We use an induction on $r\ge 3$. 
The next diagram is homotopy commutative and 
$g_{3,2}\simeq (\overline{f_2}\cup C1_{X_2})\circ\omega_{2,1}^{-1}$. 
$$
\xymatrix{
&& \Sigma X_1 \ar[r]^-{-\Sigma f_1}  & \Sigma X_2 \ar@/^3cm/[ddd]^-=  \\
 X_2 \ar[r]^-{i_{f_1}} \ar@{=}[d] & C_{f_1} \ar[ur]^-{q_{f_1}} \ar[r]_-{i_{i_{f_1}}} \ar[d]_-{\overline{f_2}} & 
C_{f_1}\cup CX_2 \ar[u]^-{\omega_{2,1}}_-\simeq \ar[r]^-{i_{i_{i_{f_1}}}} \ar[d]_-{\overline{f_2}\cup C1_{X_2}} & 
(C_{f_1}\cup CX_2)\cup CC_{f_1} \ar[u]^-{q_{i_{f_1}}'}_-\simeq \ar[d]_-{(\overline{f_2}\cup C1_{X_2})\cup C\overline{f_2}} \\
 X_2 \ar[r]^-{f_2} & X_3 \ar[r]^-{i_{f_2}} & X_3\cup_{f_2}CX_2 \ar[r]^-{i_{i_{f_2}}} \ar[dr]_-{q_{f_2}} & 
(X_3\cup CX_2)\cup CX_3 \ar[d]_-{q_{f_2}'}^-\simeq \\
&&& \Sigma X_2 
}
$$
Then $q_{f_2}\circ g_{3,2}\circ\omega_{2,1}\simeq q_{f_2}\circ(\overline{f_2}\cup C1_{X_2})
\simeq (-\Sigma f_1)\circ\omega_{2,1}$ and so 
$q_{f_2}\circ g_{3,2}\simeq -\Sigma f_1$. 
This proves the assertion for $r=3$. 
Suppose that the assertion is true for some $r$ with $3\le r<n$. 

First we consider the case of $\ddot{s}_t$-presentation. 
Set $Y=C_{r,r-1}\cup_{g_{r,r-1}} C\Sigma^{r-2}X_1$ and consider the following diagram. 
$$
\xymatrix{
\Sigma^{r-1}X_1 \ar[r]^-{-\Sigma g_{r,r-1}} & \Sigma C_{r,r-1} \ar@/^8mm/[ddr]^-=& &\\
Y\cup_{i_{g_{r,r-1}}} CC_{r,r-1} \ar[u]_-{q_{g_{r,r-1}}'}^-\simeq \ar[r]^-{i_{i_{i_{g_{r,r-1}}}}}  
& (Y\cup CC_{r,r-1})\cup CY \ar[u]_-{q_{i_{i_{g_{r,r-1}}}}'}^-\simeq &\\
C_{r,r}\cup_{j_{r,r-1}} CC_{r,r-1} \ar[u]_-{a_{r,r-1}\cup C1_{C_{r,r-1}}}^-\simeq  \ar[r]^-{i_{i_{j_{r,r-1}}}} 
\ar[d]^-{\overline{f_r}\cup C1_{C_{r,r-1}}}
& (C_{r,r}\cup CC_{r,r-1})\cup CC_{r,r} \ar[u]_-{(a_{r,r-1}\cup C1_{C_{r,r-1}})\cup Ca_{r,r-1}}^-\simeq 
\ar[r]^-{q_{j_{r,r-1}}'}_-\simeq \ar[d]^-{(\overline{f_r}\cup C1_{C_{r,r-1}})\cup C\overline{f_r}}
& \Sigma C_{r,r-1} \ar@{=}[d]\\
X_{r+1}\cup_{\overline{f_r}^{r-1}} CC_{r,r-1} \ar[r]^-{i_{i_{\overline{f_r}^{r-1}}}} \ar@{=}[d]
& (X_{r+1}\cup CC_{r,r-1})\cup CX_{r+1} \ar[r]^-{q_{\overline{f_r}^{r-1}}'}_-\simeq & 
\Sigma C_{r,r-1} \ar[d]_-{\Sigma q_r} \\
C_{r+1,r} \ar[rr]^-{q_{r+1}=/C_{r+1,r-1}} & & \Sigma^{r-1}X_2
}
$$
The diagram is commutative except the first square which is homotopy commutative. 
It follows from definitions that $\omega_{r,r-1}=q_{g_{r,r-1}}'\circ(a_{r,r-1}\cup C1_{C_{r,r-1}})$ and 
$g_{r+1,r}\circ \omega_{r,r-1}\simeq \overline{f_r}\cup C1_{C_{r,r-1}}$ so that 
\begin{align*}
(\ /C_{r+1,r-1})\circ & g_{r+1,r}\circ\omega_{r,r-1}\simeq (\ /C_{r+1,r-1})\circ(\overline{f_r}\cup C1_{C_{r,r-1}})
=\Sigma q_r\circ q_{j_{r,r-1}}\\
&\simeq \Sigma q_r\circ (-\Sigma g_{r,r-1})\circ\omega_{r,r-1}\\
&\simeq (-1)^{r+1} \Sigma^{r-1}f_1\circ\omega_{r,r-1}\ (\text{by the inductive assumption}).
\end{align*}
Hence $(\ /C_{r+1,r-1})\circ g_{r+1,r}\simeq (-1)^{r+1}\Sigma^{r-1}f_1$. 
This completes the induction and proves Lemma 6.5.3 for the case of $\ddot{s}_t$-presentation. 

Secondly we consider the case of $aq\ddot{s}_2$-presentation. 
It suffices to prove 
\begin{equation}
q_{r+1}\circ g_{r+1,r}\circ\omega_{r,r-1}\simeq (-1)^{r+1}1_{\Sigma^{r-1}X_2}\circ \Sigma^{r-1}f_1\circ\omega_{r,r-1}.
\end{equation} 
We have 
$$
q_{r+1}\circ g_{r+1,r}\circ\omega_{r,r-1}\simeq q_{r+1}\circ (\overline{f_r}\cup C1_{C_{r,r-1}})
=\Sigma q_r\circ q_{j_{r,r-1}}'\circ i_{i_{j_{r,r-1}}}
=\Sigma q_r\circ q_{j_{r,r-1}}
$$
and 
\begin{align*}
&(-1)^{r+1}1_{\Sigma^{r-1}X_2}\circ \Sigma^{r-1}f_1\circ\omega_{r,r-1}\\
&\simeq \Sigma q_r\circ(-\Sigma g_{r,r-1})\circ\omega_{r,r-1}\quad(\text{by the inductive assumption})\\
&\simeq \Sigma q_r\circ (-\Sigma(\overline{f_{r-1}}\cup C1_{C_{r-1,r-2}})\circ \Sigma\omega_{r-1,r-2}^{-1})\circ\omega_{r,r-1}\\
&=\Sigma q_r\circ (-\Sigma(\overline{f_{r-1}}\cup C1_{C_{r-1,r-2}})\circ \Sigma\omega_{r-1,r-2}^{-1})\circ\widetilde{\omega_{r-1,r-2}}\\
&=\Sigma q_r\circ (-\Sigma(\overline{f_{r-1}}\cup C1_{C_{r-1,r-2}})\circ \Sigma\omega_{r-1,r-2}^{-1})\circ \Sigma\omega_{r-1,r-2}\\
&\hspace{3cm}\circ 
q_{\overline{f_{r-1}}\cup C\overline{f_{r-1}}^{r-2}}\circ\xi\\
&\simeq \Sigma q_r\circ (-\Sigma(\overline{f_{r-1}}\cup C1_{C_{r-1,r-2}}))\circ q_{\overline{f_{r-1}}\cup C\overline{f_{r-1}}^{r-2}}\circ\xi,
\end{align*}
where 
\begin{align*}
\xi:\ &C_{r,r}\cup_{j_{r,r-1}}CC_{r,r-1}\\
&=(X_r\cup_{\overline{f_{r-1}}} CC_{r-1,r-1})
\cup_{1_{X_r}\cup Cj_{r-1,r-2}} C(X_r\cup_{\overline{f_{r-1}}^{r-2}} CC_{r-1,r-2})\\
&\to  (X_r\cup_{1_{X_r}} CX_r)\cup_{\overline{f_{r-1}}\cup C\overline{f_{r-1}}^{r-2}} 
C(C_{r-1,r-1}\cup_{j_{r-1,r-2}}CC_{r-1,r-2})
\end{align*}
is the homeomorphism defined in (5.3). 
Set 
\begin{align*}
M&=\Sigma q_r\circ q_{j_{r,r-1}},\\ 
N&=\Sigma q_r\circ(-\Sigma(\overline{f_{r-1}}\cup C1_{C_{r-1,r-2}}))\circ q_{\overline{f_{r-1}}\cup C\overline{f_{r-1}}^{r-2}}\circ\xi.
\end{align*} 
Then (6.5.1) is equivalent to $M\simeq N$. 
Let $\pi$ denote the composite of 
\begin{align*}
&C_{r,r}\cup_{j_{r,r-1}}CC_{r,r-1}\overset{q_{j_{r,r-1}}}{\longrightarrow} \Sigma C_{r,r-1}
=\Sigma(X_r\cup_{\overline{f_{r-1}}^{r-2}}CC_{r-1,r-2})\\
&\hspace{3cm}\overset{\Sigma q_{\overline{f_{r-1}}^{r-2}}}{\longrightarrow} 
\Sigma\Sigma C_{r-1,r-2}.
\end{align*}
Define $M',N' : \Sigma\Sigma C_{r-1,r-2}\to \Sigma^{r-1}X_2$ by 
$$
M'(y\wedge\overline{t}\wedge\overline{u})=q_{r-1}(y)\wedge\overline{t}\wedge\overline{u},\quad 
N'(y\wedge\overline{t}\wedge\overline{u})=q_{r-1}(y)\wedge\overline{u}\wedge\overline{1-t}.
$$
Since $q_r=\Sigma q_{r-1}\circ q_{\overline{f_{r-1}}^{r-2}}$, we have $M=M'\circ\pi$ and $N=N'\circ\pi$. 
As is easily seen, $M'\simeq N'$ so that we have $M\simeq N$ as desired. 
This completes the induction and the proof of Lemma~6.5.3. 
\end{proof}

\begin{proof}[Proof of Theorem 6.5.1]
Under the notations of \cite{C,OO1} and Lemma 6.5.2, we have
\begin{gather*}
j_{C_{n,n-1}}:X_n=F_1\subset F_{n-1}=C_{n,n-1},\\
 \sigma_{C_{n,n-1}}: C_{n,n-1}=F_{n-1}\overset{/F_{n-2}}{\longrightarrow}
 F_{n-1}/F_{n-2}=\Sigma^{n-2}X_2\overset{(-1)^n}{\longrightarrow}\Sigma^{n-2}X_2 .
\end{gather*}
By Lemma 6.5.3, we have the following homotopy commutative diagram. 
$$
\xymatrix{
& \Sigma^{n-2}X_1 \ar[dl]_-{\Sigma^{n-2}f_1} \ar[d]^-{g_{n,n-1}} &\\
\Sigma^{n-2}X_2 & C_{n,n-1} \ar[l]^-{\sigma_{C_{n,n-1}}} \ar[d]_-{\overline{f_n}} & X_n \ar[l]_-{j_{C_{n,n-1}}} \ar[dl]^-{f_n}\\
& X_{n+1} &
}
$$
Hence $\alpha=\overline{f_n}\circ g_{n,n-1}\in\langle \vec{\bm f}\,\rangle$ and so 
$\{\vec{\bm f}\,\}^{(aq\ddot{s}_2)}\cup\{\vec{\bm f}\,\}^{(\ddot{s}_t)}\subset\langle \vec{\bm f}\,\rangle$. 
This proves Theorem~6.5.1. 
\end{proof}

\subsection{3-fold brackets} 

We denote the classical unstable Toda bracket of $\vec{\bm f}=(f_3,f_2,f_1)$ by $\{\vec{\bm f}\,\}$ or $\{f_3,f_2,f_1\}$ 
(see the end of Second 2). 
We have (1.8) from Theorem 6.6.1 and Example 6.6.2 below.  
 
\begin{thm'}%Theorem 6.6.1
When $n=3$, we have 
\begin{align*}
\{\vec{\bm f}\,\}^{(\ddot{s}_t)}&=\{\vec{\bm f}\,\}^{(aq\ddot{s}_2)}=\{\vec{\bm f}\,\}^{(q\ddot{s}_2)}=\{\vec{\bm f}\,\}\\
&\subset \{\vec{\bm f}\,\}^{(aqs_2)}=\{\vec{\bm f}\,\}^{(qs_2)}=\{\vec{\bm f}\,\}^{(q\dot{s}_2)}
=\{\vec{\bm f}\,\}^{(aq\ddot{s}_2)}\circ\mathscr{E}(\Sigma X_1)\\
&\subset \{\vec{\bm f}\,\}^{(aq)}=\{\vec{\bm f}\,\}^{(q_2)}=\{\vec{\bm f}\,\}^{(q)},
\end{align*}
so that systems $\{\ \}^{(\star)}$ of unstable higher Toda brackets for $\star=\ddot{s}_t, aq\ddot{s}_2,q\ddot{s}_2$ are normal, 
and there exist $\vec{\bm f}$ and $\vec{\bm f'}$ such that $\{\vec{\bm f}\,\}^{(q)}$ is empty, 
$\langle\vec{\bm f}\,\rangle$ is not empty, and $\{\vec{\bm f'}\}^{(q)}$ is a non-empty proper subset of 
$\langle\vec{\bm f'}\,\rangle$ so that the Cohen's system of unstable higher Toda brackets is not normal.
\end{thm'}
\begin{proof}
The relations 
$
\{\vec{\bm f}\,\}^{(\ddot{s}_t)}=\{\vec{\bm f}\,\}^{(aq\ddot{s}_2)}
=\{\vec{\bm f}\,\}^{(q\ddot{s}_2)}\subset\{\vec{\bm f}\,\}^{(aq)}=\{\vec{\bm f}\,\}^{(q_2)}=\{\vec{\bm f}\,\}^{(q)}$
$\supset \{\vec{\bm f}\,\}^{(qs_2)}
\supset \{\vec{\bm f}\,\}^{(aqs_2)}
\supset \{\vec{\bm f}\,\}^{(aq\ddot{s}_2)}
$ 
hold immediately from the definitions. 
We have $\{\vec{\bm f}\,\}^{(aqs_2)}=\{\vec{\bm f}\,\}^{(aq\ddot{s}_2)}\circ\mathscr{E}(\Sigma X_1)
=\{\vec{\bm f}\,\}^{(qs_2)}=\{\vec{\bm f}\,\}^{(q\dot{s}_2)}$ by Theorem~6.2.1(2). 
The equality $\{\vec{\bm f}\,\}^{(q\ddot{s}_2)}=\{\vec{\bm f}\,\}$ follows easily from (4.2). 

Let $p^1:\s^7\to\s^4=\mathbb{H}P^1$ be the projection, where $\mathbb{H}P^m$ is 
the quaternionic projective $m$-space, and $*^1_3$ the trivial map $\s^4\to\s^3$. 
Set $\vec{\bm f}=(*^1_3,p^1,2\Sigma^3p^1)$. 
Then it follows from \cite[Remark B.5]{OO1} that $\langle \vec{\bm f}\,\rangle$ contains $0$ and 
$\{\vec{\bm f}\,\}$ is empty so that $\{\vec{\bm f}\,\}^{(q)}$ is empty by Corollary 6.2.2(4). 

Set $\vec{\bm f'}=(2\iota_5,\nu_5\eta_8,2\iota_9)$ (see \cite{T} for notations). 
Then, since $\pi_{10}(\s^5)=\bZ_2\{\nu_5\eta_8^2\}$ and $\{\vec{\bm f'}\,\}=\{\nu_5\eta_8^2\}$ by 
\cite[Proposition 5.9,  Theorem 13.4, Corollary 3.7]{T}, we have 
$\{\vec{\bm f'}\,\}^{(q)}=\{\nu_5\eta_8^2\}$ by Theorem~6.2.1(5). 
Also we have $\langle\vec{\bm f'}\,\rangle=\pi_{10}(\s^5)$ by Remark~B.5(1) of \cite{OO1}. 
Hence $\{\vec{\bm f'}\,\}^{(q)}$ is a non-empty proper subset of $\langle\vec{\bm f'}\,\rangle$ 
so that the Cohen's system of unstable higher Toda brackets is not normal. 
This completes the proof of the theorem. 
\end{proof}

\begin{exam'}%Example 6.6.2
We use freely notations and results in \cite{T}. 
\begin{enumerate}
\item If $\ell\ge 8$, then 
$\{\Sigma^\ell\nu_5,8\iota_{\ell+8},\Sigma^{\ell+1}\sigma'\}=(-1)^\ell\zeta_{\ell+5}$,  
$$
\{\Sigma^\ell\nu_5,8\iota_{\ell+8},\Sigma^{\ell+1}\sigma'\}^{(qs_2)}
=\{\zeta_{\ell+5},-\zeta_{\ell+5}\}\subset\pi_{\ell+16}(\s^{\ell+5})=\bZ_8\{\zeta_{\ell+5}\}\oplus\bZ_{63},
$$
and the order of any element of $\{\Sigma^\ell\nu_5,8\iota_{\ell+8},\Sigma^{\ell+1}\sigma'\}^{(q)}$ is a multiple of $8$.  
\item If $\ell\ge 7$, then 
$\{\Sigma^\ell\alpha_1(5), \Sigma^\ell\alpha_1(9),\Sigma^\ell\alpha_1(12)\}=(-1)^\ell\beta_1(\ell+5)$, 
\begin{align*}
\{\Sigma^\ell\alpha_1(5), \Sigma^\ell\alpha_1(9),\Sigma^\ell\alpha_1(12)\}^{(qs_2)}&=\{\beta_1(\ell+5),-\beta_1(\ell+5)\}\\
&\subset\pi_{\ell+15}(\s^{\ell+5})=\bZ_3\{\beta_1(\ell+5)\}\oplus\bZ_2,
\end{align*}
and the order of any element of $\{\Sigma^\ell\alpha_1(5), \Sigma^\ell\alpha_1(9),\Sigma^\ell\alpha_1(12)\}^{(q)}$ 
is a multiple of $3$.  
\end{enumerate}
\end{exam'}

\subsection{A 4-fold bracket} %6.7
When $\vec{\bm \alpha}=(\alpha_4,\alpha_3,\alpha_2,\alpha_1)$ is admissible, 
we have considered the next four non-empty 
subsets of $[\Sigma^2X_1,X_5]$ which are called tertiary compositions \cite{OO1}: 
$
\{\vec{\bm \alpha}\,\}^{(0)}\subset\{\vec{\bm \alpha}\,\}^{(1)}
\subset\{\vec{\bm \alpha}\,\}^{(2)}\subset\{\vec{\bm \alpha}\,\}^{(3)}.
$
Here $\{\vec{\bm \alpha}\,\}^{(1)}=\bigcup_{\vec{\bm f}\in\mathrm{Rep}(\vec{\bm \alpha})}\{\vec{\bm f}\,\}^{(1)}$ 
is the tertiary composition of \^{O}guchi \cite{Og, OO1}, and 
\begin{gather}
\begin{split}
\{\vec{\bm \alpha}\,\}^{(2)}
&=\bigcup_{\vec{\bm A}}\big(\{f_4,[f_3,A_2,f_2],(f_2,A_1,f_1)\}\\
&\hspace{2cm}\cap\{[f_4,A_3,f_3],(f_3,A_2,f_2),-\Sigma f_1\}\big),
\end{split}\\
\{\vec{\bm \alpha}\,\}^{(3)}=\bigcup_{\vec{\bm A}}\{[f_4,A_3,f_3],i_{f_3}\circ[f_3,A_2,f_2], (f_2,A_1,f_1)\},
\end{gather}
where $\vec{\bm f}=(f_4,f_3,f_2,f_1)$ is a representative of $\vec{\bm \alpha}$ and 
unions $\bigcup_{\vec{\bm A}}$ are taken over all $\vec{\bm A}=(A_3,A_2,A_1)$ such that 
$(\vec{\bm f};\vec{\bm A})$ is admissible. 
As remarked in \cite[p.56]{OO1}, the terms on right hand sides of (6.7.1) and (6.7.2) do not depend on the choice 
of a representative $\vec{\bm f}$. 
The following theorem implies (1.9). 

\begin{thm'}%Theorem 6.7.1
The sequence $\vec{\bm \alpha}=(\alpha_4,\alpha_3,\alpha_2,\alpha_1)$ is admissible if and only 
if it is $\star$-presentable for some and hence all $\star$. 
If $\vec{\bm \alpha}$ is admissible and $\vec{\bm f}=(f_4,f_3,f_2,f_1)\in\mathrm{Rep}(\vec{\bm \alpha})$, then 
\begin{equation}
\{\vec{\bm \alpha}\,\}^{(2)}\subset \{\vec{\bm \alpha}\,\}^{(\ddot{s}_t)}=\{\vec{\bm f}\,\}^{(\ddot{s}_t)}
=\bigcup\{f_4,[f_3,A_2,f_2],(f_2,A_1,f_1)\}\subset\{\vec{\bm \alpha}\,\}^{(3)},
\end{equation}
where $\bigcup$ is taken over all $\vec{\bm A}=(A_3,A_2,A_1)$ such that $(\vec{\bm f};\vec{\bm A})$ is admissible. 
\end{thm'}

\begin{cor'}[cf.\,Theorem 2.7 of \cite{W2}]%6.7.2
If a map $f_0:X_0\to X_1$ satisfies $\{f_2,f_1,f_0\}=\{0\}$, then 
$
\{f_4,f_3,f_2,f_1\}^{(\ddot{s}_t)}\circ\Sigma^2 f_0\subset f_4\circ\{f_3,f_2,f_1,f_0\}^{(\ddot{s}_t)}. 
$
\end{cor'}

\begin{proof}[Proof of Theorem 6.7.1] 
First we suppose that $\vec{\bm \alpha}$ is admissible. 
Let $\vec{\bm f}\in\mathrm{Rep}(\vec{\bm \alpha})$.  
We are going to show
\begin{equation}
\bigcup\{f_4,[f_3,A_2,f_2],(f_2,A_1,f_1)\}\subset\{\vec{\bm f}\,\}^{(\ddot{s}_t)}
\end{equation}
where $\bigcup$ is taken over all $\vec{\bm A}=(A_3,A_2,A_1)$ such that $(\vec{\bm f};\vec{\bm A})$ is admissible. 

Let $(\vec{\bm f};\vec{\bm A})$ be an admissible representative of $\vec{\bm \alpha}$. 
Then $[f_4,A_3,f_3]\circ(f_3,A_2,f_2)\simeq *$ and $[f_3,A_2,f_2]\circ (f_2,A_1,f_1)\simeq *$, 
and it follows from \cite[Proposition (5.11)]{Og} (or \cite[Lemma 3.6]{OO1}) that 
\begin{align*}
f_4\circ [f_3,A_2,f_2]&=[f_4,A_3,f_3]\circ i_{f_3}\circ [f_3,A_2,f_2]\\
&\simeq [f_4,A_3,f_3]\circ (f_3,A_2,f_2)\circ q_{f_2}\simeq *.
\end{align*}
Hence $\{f_4,[f_3,A_2,f_2],(f_2,A_1,f_1)\}$ is non-empty. 
Take $\alpha$ from it. 
Then there exist 
$F:f_4\circ[f_3,A_2,f_2]\simeq *$ and $B_2:[f_3,A_2,f_2]\circ(f_2,A_1,f_1)\simeq *$ with 
$$
\alpha=[f_4,F,[f_3,A_2,f_2]]\circ([f_3,A_2,f_2],B_2,(f_2,A_1,f_1)).
$$ 
We are going to construct an $\ddot{s}_t$-presentation 
$\{\mathscr{S}_r,\overline{f_r},\mathscr{A}_r\,|\,r=2,3,4\}$ 
of $\vec{\bm f}$ with $\overline{f_4}\circ g_{4,3}=\alpha$. 
If this is done, then $\vec{\bm f}$ is $\ddot{s}_t$-presentable and hence $\star$-presentable for every $\star$ 
by Corollary 6.2.2(4), and (6.7.4) is proved. 

Set 
\begin{align*}
&\mathscr{S}_2=\big(X_1; X_2, X_2\cup_{f_1}CX_1; f_1; i_{f_1}), \ \overline{f_2}=[f_2,A_1,f_1], \ 
\mathscr{A}_2=\{1_{X_2\cup_{f_1}CX_1}\},\\
&\mathscr{S}_3=\mathscr{S}_2(\overline{f_2},\mathscr{A}_2).
\end{align*}
Then $g_{3,2}=(\overline{f_2}\cup C1_{X_2})\circ q_{f_1}'^{-1}\simeq (f_2,A_1,f_1)$ by (4.2). 
Since $j_{3,2}$ is a homotopy cofibre of $g_{3,2}$, there exists $a_{3,2}:C_{3,3}\simeq C_{3,2}\cup_{g_{3,2}}C\Sigma X_1$ 
such that $a_{3,2}\circ j_{3,2}=i_{g_{3,2}}$.
Set $\mathscr{A}_3=\{1_{C_{3,2}}, a_{3,2}\}$. 
Take $H:g_{3,2}\simeq (f_2,A_1,f_1)$ and set 
\begin{align*}
&\Phi(H)=\Phi(g_{3,2},(f_2,A_1,f_1),1_{\Sigma X_1},1_{C_{3,2}};H)\\
&\qquad :C_{3,2}\cup_{g_{3,2}}C\Sigma X_1\to C_{3,2}\cup_{(f_2,A_1,f_1)}C\Sigma X_1,\\
&\overline{f_3}=[[f_3,A_2,f_2],B_2,(f_2,A_1,f_1)]\circ\Phi(H)\circ a_{3,2}:C_{3,3}\to X_4,\\
&\mathscr{S}_4=\mathscr{S}_3(\overline{f_3},\mathscr{A}_3).
\end{align*} 
Since $\overline{f_3}^2=\overline{f_3}\circ j_{3,2}=[f_3,A_2,f_2]$, we can set 
$$
\overline{f_4}=[f_4,F,[f_3,A_2,f_2]]:C_{4,3}=X_4\cup_{\overline{f_3}^2}CC_{3,2}\to X_5.
$$  
Let $\mathscr{A}_4$ be any reduced structure on $\mathscr{S}_4$. 
Then $\{\mathscr{S}_r,\overline{f_r},\mathscr{A}_r\,|\,r=2,3,4\}$ is an $\ddot{s}_t$-presentation of $\vec{\bm f}$ and  
\begin{align*}
&g_{4,3}=(\overline{f_3}\cup C1_{C_{3,2}})\circ \omega_{3,2}^{-1}\\
&\simeq ([[f_3,A_2,f_2],B_2,(f_2,A_1,f_1)]\cup C1_{C_{3,2}})\circ(\Phi(H)\cup C1_{C_{3,2}})\\
&\hspace{7cm}\circ (a_{3,2}\cup C1_{C_{3,2}})\circ\omega_{3,2}^{-1}\\
&\simeq ([[f_3,A_2,f_2],B_2,(f_2,A_1,f_1)]\cup C1_{C_{3,2}})\circ(\Phi(H)\cup C1_{C_{3,2}})\circ q_{g_{3,2}}'^{-1}\\
&\simeq ([[f_3,A_2,f_2],B_2,(f_2,A_1,f_1)]\cup C1_{C_{3,2}})\circ q_{(f_2,A_1,f_1)}'^{-1}\\ &\hspace{6cm}(\text{by Proposition 3.3(2)})\\
&\simeq ([f_3,A_2,f_2],B_2,(f_2,A_1,f_1))\quad (\text{by (4.2)}).
\end{align*}
Hence $\overline{f_4}\circ g_{4,3}\simeq [f_4,F,[f_3,A_2,f_2]]\circ ([f_3,A_2,f_2],B_2,(f_2,A_1,f_1)$. 
Thus (6.7.4) is proved. 

Secondly suppose that $\vec{\bm \alpha}$ is $\ddot{s}_t$-presentable. 
Then it has an $\ddot{s}_t$-presentable representative  $\vec{\bm f}=(f_4,f_3,f_2,f_1)$. 
We are going to show that $\vec{\bm f}$ and $\vec{\bm \alpha}$ are admissible, and 
\begin{equation}
\bigcup\{f_4,[f_3,A_2,f_2],(f_2,A_1,f_1)\}\supset\{\vec{\bm f}\,\}^{(\ddot{s}_t)},
\end{equation}
where $\bigcup$ is taken over all $\vec{\bm A}=(A_3,A_2,A_1)$ such that $(\vec{\bm f};\vec{\bm A})$ is admissible.
If this is done, then 
\begin{equation}
\{\vec{\bm f}\,\}^{(\ddot{s}_t)}=\bigcup\{f_4,[f_3,A_2,f_2],(f_2,A_1,f_1)\}
\end{equation}
by (6.7.4) and (6.7.5). 

Let $\alpha\in\{\vec{\bm f}\,\}^{(\ddot{s}_t)}$ and $\{\mathscr{S}_r,\,  \overline{f_r},\,\mathscr{A}_r\,|\, r=2,3,4\}$ 
an $\ddot{s}_t$-presentation of $\vec{\bm f}$ with $\alpha=\overline{f_4}\circ g_{4,3}$. 
Take $A_1:f_2\circ f_1\simeq *$ and $A_2:f_3\circ f_2\simeq *$ such that 
$\overline{f_2}=[f_2,A_1,f_1]$ and $\overline{f_3}^2=[f_3,A_2,f_2]$. 
We have $g_{3,2}=(\overline{f_2}\cup C1_{X_2})\circ q_{f_1}'^{-1}\simeq (f_2,A_1,f_1)$ and 
$g_{4,2}=(\overline{f_3}^2\cup C1_{X_3})\circ q_{f_2}'^{-1}\simeq (f_3,A_2,f_2)$ by (4.2). 
Since $j_{3,2}$ is a homotopy cofibre of $g_{3,2}$ and $\overline{f_3}$ is an extension of $\overline{f_3}^2$ on $C_{3,3}$, 
we have $\overline{f_3}^2\circ g_{3,2}\simeq *$ by Lemma 4.3(7). 
Hence 
$$
[f_3,A_2,f_2]\circ (f_2,A_1,f_1)\simeq *.
$$
Since $\overline{f_4}^2:X_4\cup_{f_3}CX_3\to X_5$ is an extension of $f_4$, there exists $A_3:f_4\circ f_3\simeq *$ 
such that $\overline{f_4}^2=[f_4,A_3,f_3]$. 
Since $j_{4,2}$ is a homotopy cofibre of $g_{4,2}$ and $\overline{f_4}$ is an extension of $\overline{f_4}^2$ on $C_{4,3}$, 
we have $\overline{f_4}^2\circ g_{4,2}\simeq *$ by Lemma 4.3(7), that is, 
$$
[f_4,A_3,f_3]\circ (f_3,A_2,f_2)\simeq *.
$$
Thus $(\vec{\bm f};\vec{\bm A})$ is admissible. 
Since $\overline{f_4}$ is an extension of $f_4$ on $X_4\cup_{\overline{f_3}^2}CC_{3,2}$, 
there exists a homotopy $D:f_4\circ \overline{f_3}^2\simeq *$ such that $\overline{f_4}=[f_4,D,\overline{f_3}^2]$. 
Let $B:\overline{f_3}^2\circ g_{3,2}\simeq *$ be a homotopy such that $\overline{f_3}\circ a_{3,2}^{-1}=[\overline{f_3}^2,B,g_{3,2}]$.  
Then 
\begin{align*}
g_{4,3}&=(\overline{f_3}\cup C1_{C_{3,2}})\circ\omega_{3,2}^{-1}\\
&\simeq (\overline{f_3}\cup C1_{C_{3,2}})\circ(a_{3,2}^{-1}\cup C1_{C_{3,2}})\circ(a_{3,2}\cup C1_{C_{3,2}})\circ\omega_{3,2}^{-1}\\
&=(\overline{f_3}\circ a_{3,2}^{-1}\circ a_{3,2}\cup C1_{C_{3,2}})\circ\omega_{3,2}^{-1}
=([\overline{f_3}^2,B,g_{3,2}]\circ a_{3,2}\cup C1_{C_{3,2}})\circ\omega_{3,2}^{-1}\\
&\simeq (\overline{f_3}^2,B,g_{3,2})\quad(\text{by (4.2)}).
\end{align*}
Hence we have 
\begin{align*}
\alpha&=\overline{f_4}\circ g_{4,3}=[f_4,D,\overline{f_3}^2]\circ (\overline{f_3}^2,B,g_{3,2})\\
&\in\{f_4,\overline{f_3}^2,g_{3,2}\}=\{f_4,[f_3,A_2,f_2],(f_2,A_1,f_1)\}.
\end{align*}
This proves (6.7.5) and (6.7.6) holds. 

Therefore $\vec{\bm \alpha}$ is admissible if and only if it is $\ddot{s}_t$-presentable and 
hence $\star$-presentable for every $\star$. 
Also if $\vec{\bm \alpha}$ is admissible, then (6.7.6) holds for every representative $\vec{\bm f}$ of $\vec{\bm \alpha}$.  

In the rest of the proof we suppose that $\vec{\bm \alpha}$ is admissible and $\vec{\bm f}\in\mathrm{Rep}(\vec{\bm \alpha})$. 
Then $\{\vec{\bm \alpha}\,\}^{(2)}\subset\{\vec{\bm \alpha}\,\}^{(\ddot{s}_t)}=\{\vec{\bm f}\,\}^{(\ddot{s}_t)}
=\bigcup_{\vec{\bm A}}\{f_4,[f_3,A_2,f_2],(f_2,A_1,f_1)\}$ by (6.7.1), (6.7.6), and Theorem 6.5.1. 
Since $f_4=[f_4,A_3,f_3]\circ i_{f_3}$, it follows from \cite[Proposition~1.2~III)]{T} that 
$$
\{f_4,[f_3,A_2,f_2],(f_2,A_1,f_1)\}\subset\{[f_4,A_3,f_3],i_{f_3}\circ[f_3,A_2,f_2],(f_2,A_1,f_1)\}
$$
so that $\{\vec{\bm f}\,\}^{(\ddot{s}_t)}\subset\{\vec{\bm \alpha}\,\}^{(3)}$ by (6.7.2). 
This completes the proof of Theorem~6.7.1. 
\end{proof}

\begin{proof}[Proof of Corollary 6.7.2] 
By Theorem 6.7.1, we have
$$
\{f_4,f_3,f_2,f_1\}^{(\ddot{s}_t)}\circ\Sigma^2 f_0=\bigcup\big(\{f_4,[f_3,A_2,f_2],(f_2,A_1,f_1)\}\circ\Sigma^2f_0\big),
$$
where $\bigcup$ is taken over all $(A_3,A_2,A_1)$ such that $(f_4,f_3,f_2,f_1;A_3,A_2,A_1)$ is  admissible. 
Take any such $(A_3,A_2,A_1)$. 
By the assumption, $f_1\circ f_0\simeq *$ and for any $A_0:f_1\circ f_0\simeq *$, 
we have $[f_2,A_2,f_1]\circ(f_1,A_0,f_0)\simeq *$ so that $(f_3,f_2,f_1,f_0;A_2,A_1,A_0)$ is admissible. 
Hence $\{f_3,[f_2,A_1,f_1],(f_1,A_0,f_0)\}\subset\{f_3,f_2,f_1,f_0\}^{(\ddot{s}_t)}$ by Theorem~6.7.1. 
Now we have
\begin{align*}
&\{f_4,[f_3,A_2,f_2],(f_2,A_1,f_1)\}\circ\Sigma^2 f_0\\
&=-\{f_4,[f_3,A_2,f_2],(f_2,A_1,f_1)\}\circ\Sigma q_{f_1}\circ\Sigma (f_1,A_0,f_0)\\
&\hspace{4cm} (\text{since $q_{f_1}\circ(f_1,A_0,f_0)\simeq -\Sigma f_0$})\\
&\subset-\{f_4,[f_3,A_2,f_2],(f_2,A_1,f_1)\circ q_{f_1}\}\circ\Sigma(f_1,A_0,f_0)\\
&\hspace{4cm} (\text{by \cite[Proposition 1.2(i)]{T}})\\
&=-\{f_4,[f_3,A_2,f_2],i_{f_2}\circ[f_2,A_1,f_1]\}\circ\Sigma(f_1,A_0,f_0)\ 
(\text{by \cite[(5.11)]{Og}})\\
&\subset-\{f_4,[f_3,A_2,f_2]\circ i_{f_2},[f_2,A_1,f_1]\}\circ\Sigma(f_1,A_0,f_0)\\
&\hspace{4cm} (\text{by \cite[Proposition 1.2(ii)]{T}})\\
&=-\{f_4,f_3,[f_2,A_1,f_1]\}\circ\Sigma(f_1,A_0,f_0)\\
&=f_4\circ\{f_3,[f_2,A_1,f_1],(f_1,A_0,f_0)\}\ (\text{by \cite[Proposition 1.4]{T}}).
\end{align*}
Hence we have the assertion.
\end{proof}

\begin{rem'}%6.7.3
{\rm(1)} It follows from Theorem 6.7.1, Theorem 6.5.1 and \cite[Proposition~B.6]{OO1} that 
if $\vec{\bm f}=(f_4,f_3,f_2,f_1)$ is admissible, then 
$$
\langle \vec{\bm f}\,\rangle
\supset \bigcup\Big[\{f_4,[f_3,A_2,f_2],(f_2,A_1,f_1)\}\cup\{[f_4,A_3,f_3],(f_3,A_2,f_2),-\Sigma f_1\}\Big],
$$
where $\bigcup$ is taken over all $\vec{\bm A}=(A_3,A_2,A_1)$ such that $(\vec{\bm f};\vec{\bm A})$ is admissible. 

{\rm(2)} When we work in $\mathrm{TOP}^{clw}$, it can be shown that the Walker's $4$-fold product of 
$\vec{\bm \alpha}=(\alpha_4,\alpha_3,\alpha_2,\alpha_1)$ is not empty if and only if $\vec{\bm \alpha}$ is admissible. 
\end{rem'}

\subsection{Proof of (1.10)}%6.8
We prove the following which is the same as (1.10). 

\begin{prop'}
\begin{enumerate}
\item[\rm(1)] If $\{f_{n-1},\dots,f_1\}^{(q)}\ni 0$ and $\{f_n,\dots,f_k\}^{(aq\ddot{s}_2)}
=\{0\}$ for 
all $k$ with $2\le k< n$, then 
$\{f_n,\dots,f_1\}^{(\star)}$ is not empty for all $\star$. 
\item[\rm(2)]$($cf.\,\cite[Lemma 2.2]{W2}$)$ If $\{f_n,\dots,f_2\}^{(q)}\ni 0$ and 
$\{f_k,\dots,f_1\}^{(aq\ddot{s}_2)}=\{0\}$ for all $k$ with $2\le k< n$, 
then $\{f_n,\dots,f_1\}^{(\star)}$ is not empty for all $\star$.
\end{enumerate}
\end{prop'}
\begin{proof}
When $n=3$, the assertions hold by Theorem~6.6.1 and Corollary~6.2.2(4). 
Hence we suppose $n\ge 4$. 

(1) By Corollary 6.2.2(5) and the assumptions, there exists an $aq\ddot{s}_2$-presentation 
$\{\mathscr{S}_r,\overline{f_r}, \Omega_r\,|\,2\le r<n\}$ of $(f_{n-1},\dots,f_1)$ such that 
$\overline{f_{n-1}}\circ g_{n-1,n-2}\simeq *$. 
Since $j_{n-1,n-2}$ is a homotopy cofibre of $g_{n-1,n-2}$, there is a homotopy equivalence 
$e:C_{n-1,n-1}\to C_{n-1,n-2}\cup_{g_{n-1,n-2}}C\Sigma^{n-3}X_1$ such that $e\circ j_{n-1,n-2}=i_{g_{n-1,n-2}}$. 
Let $\overline{f_{n-1}}':C_{n-1,n-2}\cup_{g_{n-1,n-2}}C\Sigma^{n-3}X_1\to X_n$ 
be an extension of $\overline{f_{n-1}}$. 
Set $\overline{f_{n-1}}^*=\overline{f_{n-1}}'\circ e:C_{n-1,n-1}\to X_n$, 
$\mathscr{S}_n=\mathscr{S}_{n-1}(\overline{f_{n-1}}^*,\Omega_{n-1})$, 
and $\Omega_n=\widetilde{\Omega_{n-1}}$. 
Since $f_n\circ f_{n-1}\simeq *$ by the assumption $\{f_n,f_{n-1}\}^{(aq\ddot{s}_2)}=\{0\}$, 
$f_n$ has an extension $\overline{f_n}^2:C_{n,2}=X_n\cup_{f_{n-1}}CX_{n-1}\to X_{n+1}$. 
Since $\overline{f_n}^2\circ g_{n,2}$ represents an element of 
$\{f_n,f_{n-1},f_{n-2}\}^{(aq\ddot{s}_2)}=\{0\}$, 
and since $j_{n,2}$ is a homotopy cofibre of $g_{n,2}$, $\overline{f_n}^2$ has 
an extension $\overline{f_n}^3:C_{n,3}=X_n\cup_{\overline{f_{n-1}}^{*2}}CC_{n-1,2}\to X_{n+1}$. 
We inductively have a map $\overline{f_n}:C_{n,n-1}=X_n\cup_{\overline{f_{n-1}}^{*n-2}}CC_{n-1,n-2}\to X_{n+1}$ 
which is an extension of $f_n$. 
Then the collection obtained from $\{\mathscr{S}_r,\overline{f_r},\Omega_r\,|\,2\le r\le n\}$ 
by replacing $\overline{f_{n-1}}$ with $\overline{f_{n-1}}^*$ is an $aq\ddot{s}_2$-presentation of $(f_n,\dots,f_1)$. 
Thus $\{f_n,\dots,f_1\}^{(aq\ddot{s}_2)}$ is not empty and so $\{f_n,\dots,f_1\}^{(\star)}$ is not empty 
for all $\star$ by Corollary 6.2.2(4). 

(2) We set $X_r'=X_{r+1}\ (1\le r\le n)$ and $f_r'=f_{r+1}\ (1\le r<n)$. 
By the assumptions and Corollary 6.2.2(5), $\{f_{n-1}',\dots,f_1'\}^{(aq\ddot{s}_2)}$ contains $0$. 
Let $\{\mathscr{S}_r',\overline{f_r'},\Omega_r'\,|\,2\le r<n\}$ be an $aq\ddot{s}_2$-presentation of 
$(f_{n-1}',\dots,f_1')$ with $\overline{f_{n-1}'}\circ g_{n-1,n-2}'\simeq *$.  
Set $\mathscr{S}_2=(X_1;X_2,X_2\cup_{f_1}CX_1;f_1;i_{f_1})$ and $\Omega_2=\{q_{f_1}'\}$. 
Let $\overline{f_2}:C_{2,2}=X_2\cup_{f_1}CX_1\to X_3$ be an extension of $f_2$. 
Set $\mathscr{S}_3=\mathscr{S}_2(\overline{f_2},\Omega_2)$ and $\Omega_3=\widetilde{\Omega_2}$. 
Then $C_{3,s}=C'_{2,s}\ (1\le s\le 2)$, $j_{3,1}=j_{2,1}'$, $g_{3,1}=g_{2,1}'$, and $\Omega_3$ contains $\Omega_2'$. 
Set $\overline{f_3}^2=\overline{f_2'}:C_{3,2}\to X_4$. 
By the assumptions, $\overline{f_3}^2\circ g_{3,2}\simeq *$ so that $\overline{f_3}^2$ 
can be extended to a map $\overline{f_3}:C_{3,3}\to X_4$ which is an extension of $f_3$.   
Set $\mathscr{S}_4=\mathscr{S}_3(\overline{f_3},\Omega_3)$ and $\Omega_4=\widetilde{\Omega_3}$. 
Then $C_{4,s}=C_{3,s}'\ (1\le s\le 3)$, $j_{4,s}=j'_{3,s}\ (1\le s\le 2)$, $g_{4,s}=g'_{3,s}\ 
(1\le s\le 2)$, and $\Omega_4$ contains $\Omega_3'$. 
When $n=4$, $\overline{f_4}:=\overline{f_3'}:C_{4,3}\to X_5$ is an extension of $f_4$ and 
we obtain $aq\ddot{s}_2$-presentation $\{\mathscr{S}_r,\overline{f_r},\Omega_r\,|\,2\le r\le 4\}$ of $(f_4,\dots,f_1)$ 
so that $\{f_4,\dots,f_1\}^{(\star)}$ is not empty for all $\star$. 
When $n\ge 5$, by repeating the above process, we have an $aq\ddot{s}_2$-presentation of $(f_n,\dots,f_1)$. 
Hence $\{f_n,\dots,f_1\}^{(\star)}$ is not empty for all $\star$. 
This completes the proof of Proposition 6.8.1.
\end{proof}

\subsection{Stable higher Toda brackets}%6.9

A stable $n$-fold bracket for $n\ge 3$ was defined in \cite{C,W2} (cf.\,\cite{K,Po}).  
We will give another definition which is a generalization of \cite[p.32]{T}. 
We set $\{X,Y\}=\varinjlim_{k}[\Sigma^k X,\Sigma^k Y]$. 
Given $\beta_i\in\{X_i,X_{i+1}\}\ (1\le i\le n)$, we will define  
$\{\beta_n,\dots,\beta_1\,\}^{(\star)}\subset\{\Sigma^{n-2}X_1,X_{n+1}\}$. 
Take a non-negative integer $m$ such that $\beta_i$ is represented by $f^m_i:\Sigma^mX_i\to\Sigma^mX_{i+1}$ 
for all $i$ and set $\vec{{\bm f}^m}=(f^m_n,\dots,f^m_1)$. 
The following square is commutative for every integer $M\ge 0$. 
$$
\xymatrix{
[\Sigma^m\Sigma^{n-2}X_1,\Sigma^mX_{n+1}] \ar[r]^-{\Sigma^M}
 & [\Sigma^M\Sigma^m\Sigma^{n-2}X_1,\Sigma^M\Sigma^mX_{n+1}]\\
[\Sigma^{n-2}\Sigma^m X_1,\Sigma^mX_{n+1}] \ar[r]^-{\Sigma^M}  
\ar[u]^-{(1_{X_1}\wedge\tau(\s^{n-2},\s^m))^*}
&[\Sigma^M\Sigma^{n-2}\Sigma^mX_1,\Sigma^M\Sigma^mX_{n+1}] \ar[u]_-{(1_{X_1}\wedge\tau(\s^{n-2},\s^m)\wedge 1_{\s^{M}})^*}
}
$$
We have
\begin{align*}
&\Sigma^M(\{\vec{{\bm f}^m}\}^{(\star)}\circ(1_{X_1}\wedge\tau(\s^{n-2},\s^m)))\\
&=\Sigma^M\{\vec{{\bm f}^m}\}^{(\star)}\circ(1_{X_1}\wedge\tau(\s^{n-2},\s^m)\wedge 1_{\s^M})\\
&\subset\{\Sigma^M\vec{{\bm f}^m}\}^{(\star)}\circ(1_{\Sigma^mX_1}\wedge\tau(\s^{n-2},\s^M))\circ 
(1_{X_1}\wedge\tau(\s^{n-2},\s^m)\wedge 1_{\s^M})  \\ &\hspace{8cm}(\text{by (6.3.1)})\\
&=\{\Sigma^M\vec{{\bm f}^m}\}^{(\star)}\circ(1_{X_1}\wedge\tau(\s^{n-2},\s^m\wedge\s^M)).
\end{align*}
Hence the sequence 
$\{\{\Sigma^M\vec{{\bm f}^m}\}^{(\star)}\circ(1_{X_1}\wedge\tau(\s^{n-2},\s^m\wedge\s^M))\}_{M\ge 0}$ 
defines a subset of $\{\Sigma^{n-2}X_1,X_{n+1}\}$. 
We denote it by $\{\beta_n,\dots,\beta_1\}^{(\star)}$. 
It does not depend on the choice of $\vec{{\bm f}^m}$. 
For another $\vec{{\bm f}^k}$, there exist $M,K$ such that $\Sigma^Mf^m_i\simeq\Sigma^Kf^k_i$ for all $i$. 
In this case $m+M=k+K$ and $\{\Sigma^M\vec{{\bm f}^m}\}^{(\star)}=\{\Sigma^K\vec{{\bm f}^k}\}^{(\star)}$ by Theorem 6.4.1. 
Hence $\{\Sigma^M\vec{{\bm f}^m}\}^{(\star)}\circ(1_{X_1}\wedge\tau(\s^{n-2},\s^m\wedge\s^M))=
\{\Sigma^K\vec{{\bm f}^k}\}^{(\star)}\circ(1_{X_1}\wedge\tau(\s^{n-2},\s^k\wedge\s^K))$. 
Thus $\{\beta_n,\dots,\beta_1\}^{(\star)}$ is well-defined.  

\begin{appendix}
\numberwithin{equation}{section}

\section{Proof of Proposition 2.2}
Given a free space $X$, we set $\Gamma X=(X\times I)/(X\times\{1\})$ which is called 
the {\it unpointed cone} on $X$ and whose point represented by $(x,t)\in X\times I$ is denoted by $x\wedge t$. 
We regard $X$ as a subspace of $\Gamma X$ by the identification $x=x\wedge 0\ (x\in X)$. 
We set $SX=\Gamma X/X$ which is called the {\it unpointed suspension} of $X$ and 
whose point represented by $(x,t)\in X\times I$ is denoted by $x\wedge\overline{t}$. 
For a map $f:X\to Y$, we define $\Gamma f:\Gamma X\to \Gamma Y$ and $Sf:SX\to SY$ by 
$\Gamma f(x\wedge t)=f(x)\wedge t$ and $Sf(x\wedge\overline{t})=f(x)\wedge\overline{t}$, 
and we denote by $Y\cup_f\Gamma X$ the quotient space of $Y+\Gamma X$ by the equivalence relation generated by the relation 
$f(x)\sim x\wedge 0\ (x\in X)$. 

Given two free maps $X\overset{u}{\longleftarrow}A\overset{v}{\longrightarrow}Y$, 
let $X\overset{uAv}{+}Y$ denote 
the quotient space of $X+Y$ by the equivalence relation generated by the relation $u(a)\sim v(a)\ (a\in A)$. 
Let $X\overset{i_X}{\longrightarrow} X+Y \overset{i_Y}{\longleftarrow} Y$ be the inclusion maps 
and $q:X+Y\to X\overset{uAv}{+}Y$ the quotient map. 
Then the following is a push-out diagram in $\mathrm{TOP}$. 
\begin{equation}
\begin{CD}
A @>v>> Y\\
@VuVV @ VVq\circ i_YV\\
X @>q\circ i_X>> X\overset{uAv}{+}Y
\end{CD}
\end{equation}
The space $X\overset{uAi}{+}\Gamma A$ which is induced from $X\overset{u}{\longleftarrow}A\overset{i}{\subset} \Gamma A$ is denoted by 
 $X\cup_u\Gamma A$ and called the {\it unpointed mapping cone} of $u$. 

\begin{lemma} %A.1
Given the push-out diagram (A.1), 
if $u$ is a free (resp.\,closed free) cofibration, then $q\circ i_Y$ is a free (resp.\,closed free) cofibration. 
\end{lemma}
\begin{proof}
Suppose that $u$ is a free cofibration. 
Then, as is well-known (for example \cite[(5.1.8)]{tD}), $q\circ i_Y$ is a free cofibration. 
Since $u$ is injective, the equality $q^{-1} (q\circ i_Y(B))=u(v^{-1}(B))+B$ holds for every subset $B$ of $Y$. 
Hence if $u$ is closed then $q\circ i_Y$ is closed. 
\end{proof}

Note that, given a pointed map $f:X\to Y$, we have the following commutative diagram in which all maps are quotient maps. 
$$
\xymatrix{
Y+X\times I \ar[r]^-{1_Y+q} \ar[dr]_-{1_Y+q} & Y+\Gamma X \ar[r]^-q \ar[d]^-{1_Y+q} & Y\cup_f\Gamma X \ar[d]^-{1_Y\cup q}\\
& Y+CX \ar[r]^-q & Y\cup_f CX
}
$$

\begin{proof}[Proof of Proposition 2.2]
We can suppose that $j:A\subset X$ by \cite[Theorem 1]{S1}. 
Consider the following commutative diagram. 
$$
\xymatrix{
Y+A\times I \ar[d]_i \ar[r]^-{\pi_1} & Y\cup_{f\circ j} \Gamma A \ar[dd]^-{1_Y\cup\Gamma j} \ar[r]^-{q_1} & Y\cup_{f\circ j} CA \ar[dd]^-{1_Y\cup Cj}\\
Y+X\times\partial I\cup A\times I \ar[d]_-{i'} \ar@{.>}[ur]_-{\varphi_f} &\\
Y+X\times I\ar[r]^-{\pi_2} & Y\cup_f \Gamma X \ar[r]^-{q_2} & Y\cup_f CX
}
$$
where $\pi_1,\pi_2, q_1, q_2$ are quotient maps, $i,i'$ are inclusions, and $\varphi_f$ is defined by 
\begin{gather*}
\varphi_f|_Y=\pi_1|_Y,\quad \varphi_f(x,0)=\pi_1f(x),\\
 \varphi_f(x,1)=\pi_1(a,1)\ (a\in A),\quad \varphi_f|_{A\times I}=\pi_1|_{A\times I}.
\end{gather*}
First we show that $\varphi_f$ is continuous. 
By \cite[Theorem 2]{S2}, $X\times\{0\}\cup A\times I$ and $X\times\{1\}\cup A\times I$ are retracts of $X\times I$. 
Therefore, since $\varphi_f$ is continuous on the subspaces $X\times\{0\}$, $X\times\{1\}$, $A\times I$ of the space $X\times I$, 
it follows from \cite[Lemma~3]{S2} that $\varphi_f$ is continuous on the subspaces $X\times\{0\}\cup A\times I$, 
$X\times\{1\}\cup A\times I$ of the space $X\times I$. 
Hence $\varphi_f$ is continuous on 
the open subspaces $X\times\{0\}\cup A\times[0,1)$, $X\times\{1\}\cup A\times(0,1]$ of the space $X\times\partial I\cup A\times I$. 
Therefore $\varphi_f$ is continuous on 
$X\times\{0\}\cup A\times [0,1)\cup X\times\{1\}\cup A\times (0,1]=X\times\partial I\cup A\times I$ 
so that $\varphi_f$ is continuous. 

Secondly we show that $(\pi_2,1_Y\cup\Gamma j)$ is a push-out of $(i',\varphi_f)$ 
in $\mathrm{TOP}$. 
For any space $Z$ and any maps $Y+X\times I\overset{g}{\rightarrow} Z\overset{h}{\leftarrow} Y\cup_{f\circ j}\Gamma A$ 
such that $g\circ i'=h\circ\varphi_f$, 
there exists only one map $k:Y\cup_f\Gamma X\to Z$ with $k\circ \pi_2=g$. 
It is obvious that $k\circ(1_Y\cup\Gamma j)=h$. 
Hence $(\pi_2,1_Y\cup\Gamma j)$ is a push-out of $(i',\varphi_f)$ 
in $\mathrm{TOP}$. 

Thirdly we show that $1_Y\cup\Gamma j$ is a free (resp.\,closed free) cofibration. 
By the last assertion and Lemma~A.1, it suffices to show that $i'$ is a free (resp.\,closed free) cofibration. 
Since the inclusion $\partial I\subset I$ is a closed free cofibration, it follows from \cite[Theorem 6]{S2} 
that the inclusion $X\times\partial I\cup A\times I\subset X\times I$ is a free 
(resp.\,closed free) cofibration so that $i'$ is a free (resp.\,closed free) cofibration. 

Fourthly we show that $1_Y\cup Cj:Y\cup_{f\circ j}CA\to Y\cup_f CX$ is a free (resp.\,closed free) cofibration. 
By Lemma A.1 it suffices to show that the last square of the above diagram is a push-out in $\mathrm{TOP}$. 
Let $Z$ be any space and $Y\cup_f \Gamma X\overset{g}{\rightarrow}Z\overset{h}{\leftarrow}Y\cup_{f\circ j}CA$ 
any maps such that $g\circ(1_Y\cup\Gamma j)=h\circ q_1$. 
Then there is only one map $k:Y\cup_f CX\to Z$ with $k\circ q_2=g$. 
It is obvious that $k\circ (1_Y\cup Cj)=h$. 
Hence the last square of the above diagram is a push-out in $\mathrm{TOP}$. 
\end{proof} 

The following corollary overlaps with \cite[(6.13)]{J}. 

\begin{cor}%A.2
If $j:A\to X$ is a free (resp.\,closed free) cofibration, 
then $\Gamma j:\Gamma A\to \Gamma X$ 
and $Sj:SA\to SX$ are free (resp.\,closed free) cofibrations. 
If in addition $j$ is pointed, then $Cj:CA\to CX$ is a free (resp.\,closed free) cofibration. 
\end{cor}
\begin{proof}
We can suppose that $j:A\subset X$. 
Consider the commutative diagram
$$
\xymatrix{
A+A\times I \ar[d]_-{i_1} \ar[r]^-{q_1} &A\cup_{1_A}\Gamma A=\Gamma A \ar[d]^-{j\cup\Gamma 1_A} &\\ 
X+A\times I \ar[d]_-{i_2} \ar[r]^-{q_2} & X\cup_j\Gamma A \ar[dd]^-{1_X\cup\Gamma j} \ar[r]^-{/X} & SA \ar[dd]^-{Sj} &\\ 
X+X\times\partial I\cup A\times I \ar@{.>}[ur]_-{\varphi_{1_X}} \ar[d]_-{i_3} & & \\
X+X\times I \ar[r]^-{q_3} & X\cup_{1_X}\Gamma X=\Gamma X \ar[r]^-{/X} & SX
}
$$
where $i_1,i_2,i_3$ are inclusions, $q_1,q_2,q_3$ are quotients, 
and $\varphi_{1_X}$ is defined as in the proof of Proposition~2.2. 
When we take off $i_2$, the remaining three squares of the diagram are push-outs in $\mathrm{TOP}$. 
Since $i_1, i_3$ are free (resp.\,closed free) cofibrations, it follows that $j\cup\Gamma 1_A, 1_X\cup\Gamma j$ 
are free (resp.\,closed free) cofibrations so that 
$\Gamma j=(1_X\cup\Gamma j)\circ (j\cup\Gamma 1_A)$ and $Sj$ are free (resp.\,closed free) cofibrations. 

Suppose that $j$ is pointed. 
Consider the commutative diagram
$$
\xymatrix{
A\cup_{1_A}\Gamma A=\Gamma A \ar[d]_-{j\cup\Gamma 1_A} \ar[r]^-{q_4} & A\cup_{1_A} CA=CA \ar[d]^-{j\cup C1_A}\\
X\cup_j\Gamma A \ar[d]_-{1_X\cup\Gamma j} \ar[r]^-{q_5} & X\cup_j CA \ar[d]^-{1_X\cup Cj}\\
X\cup_{1_X}\Gamma X \ar[r]^-{q_6} & X\cup_{1_X} CX=CX
}
$$
where $q_4,q_5,q_6$ are quotients. 
Since the two squares of the diagram are push-outs in $\mathrm{TOP}$, $
Cj=(1_X\cup Cj)\circ(j\cup C1_A)$ is a free (resp.\,closed free) cofibration by Lemma~A.1. 
This completes the proof of Corollary A.2.
\end{proof}

\section{J. Cohen's higher Toda brackets}

First we recall from \cite[Theorem 2]{S2} that the inclusion $j:X\subset  Y$ is a free cofibration 
if and only if there exists a retraction $r:Y\times I\to Y\times\{0\}\cup X\times I$. 
When $j$ is pointed, the pointed map $\delta : Y/X\to \Sigma X$ 
which makes the following diagram to be commutative was called ``canonical'' in \cite{C}.
$$
\xymatrix{
Y \ar[d]_-q \ar[r]^-{i_1^Y} & Y\times I \ar[r]^-r & Y\times\{0\}\cup X\times I \ar[d]^-p\\
Y/X \ar[rr]_-\delta & & \Sigma X
}
$$ 
Here $q$ is the quotient, $p(Y\times \{0\})=*$, and $p(x,t)=x\wedge\overline{t}$. 
An important property of $\delta$ is the following. 

\begin{lemma}%B.1
The canonical map $\delta$ is a connecting map in the cofibre sequence 
$$
\begin{CD}
X@>j>>Y@>q>>Y/X@>\delta>>\Sigma X@>-\Sigma j>>\Sigma Y@>-\Sigma q>>\cdots .
\end{CD}
$$
\end{lemma} 
\begin{proof}
Consider the following diagram, where $\pi$ is the usual homotopy equivalence \cite[Satz 2]{P}. 
\begin{equation}
\begin{CD}
X@>j>>Y@>i_j>>Y\cup_j CX@>i_{i_j}>>(Y\cup_j CX)\cup_{i_j}CY@>i_{i_{i_j}}>>\cdots\\
@. @| @V\simeq V\pi V @V\simeq V q_j'V @.\\
@. Y@>q>>Y/X@>\delta>>\Sigma X@>-\Sigma j>> \cdots
\end{CD}
\end{equation}
We define $u:I\times I\to I$ and $H:(Y\cup_j CX)\times I\to \Sigma X$ by 
$$
u(s,t)=\begin{cases} s+t & s+t\le 1 \\ 1 & s+t\ge 1 \end{cases},\quad 
H(y,t)=p\circ r(y,t),\quad H(x\wedge s,t)=x\wedge\overline{u(s,t)}.
$$
Then $H:q_j'\circ i_{i_j}\simeq \delta\circ \pi$. 
Hence the second square of (B.1) is homotopy commutative. 
Since the first square of (B.1) is commutative, 
this completes the proof. 
\end{proof}

J. Cohen \cite{C} defined an $n$-fold bracket $\langle\vec{\bm f}\,\rangle$ in the category $\mathrm{TOP}^*$, where 
$\vec{\bm f}=(f_n,\dots,f_1)$ and $f_i : X_i\to X_{i+1}$ is a map in $\mathrm{TOP}^*\ (1\le i\le n;\ n\ge 3)$.  
We are going to modify $\langle\vec{\bm f}\,\rangle$ to $\langle\vec{\bm f}\,\rangle^w$ 
(resp.\,$\langle\vec{\bm f}\,\rangle^{clw}$) 
by restricting $\mathrm{TOP}^*$ to its full-subcategory $\mathrm{TOP}^w$ (resp.\,$\mathrm{TOP}^{clw}$). 

Let $\circledast$ denote $*$, $w$ or $clw$. 

By $\vec{\bm f}=(f_n,\dots,f_1)\in\mathrm{TOP}^\circledast$, we mean that 
$f_i:X_i\to X_{i+1}$ is in $\mathrm{TOP}^\circledast$ for every $i$. 
To avoid confusions, we paraphrase Cohen's expression ``$X\in\{f_{n-1},\dots,f_2\}$'' in 
``$X$ is a finitely filtered space of type $(f_{n-1},\dots,f_2)$'' \cite{OO1}. 
Given $(f_{n-1},\dots,f_2)\in\mathrm{TOP}^\circledast$, where $f_i:X_i\to X_{i+1}$, 
a pointed space $X$ is a {\it finitely filtered space of type} $(f_{n-1},\dots,f_2)$ 
{\it in} $\mathrm{TOP}^\circledast$ if the following (1) and (2) are satisfied.
\begin{enumerate}
\item[\rm(1)] The pointed space $X$ has a filtration $F_0X=\{*\}\subset F_1X\subset\dots\subset F_{n-1}X=X$ 
such that the inclusion $F_kX\subset F_{k+1}X$ 
is a free, free or closed free cofibration for every $k$ according as $\circledast$ is $*$, $w$ or $clw$. 
(Hence $X,F_kX\in\mathrm{TOP}^w$.)\\
\item[\rm(2)] There exists $g_k:\Sigma^kX_{n-k}\simeq F_{k+1}X/F_kX$ for $0\le k\le n-2$ such that 
the next diagram is homotopy commutative for $1\le k\le n-2$.  
\begin{equation}
\begin{split}
\xymatrix{
\Sigma\Sigma^{k-1}X_{n+1-k} \ar[d]_-{\Sigma g_{k-1}} & & \Sigma^kX_{n-k} \ar[ll]_-{\Sigma^kf_{n-k}} \ar[d]^-{g_k}\\
\Sigma(F_kX/F_{k-1}X) & \Sigma F_kX \ar[l]^-{\Sigma q} & F_{k+1}X/F_kX \ar[l]^-{\delta}
}
\end{split}
\end{equation}
\end{enumerate}
Under the above situation, we set 
\begin{equation}
\left\{\begin{array}{@{\hspace{0.6mm}}l}
j_X: X_n=\Sigma^0X_n\overset{g_0}{\longrightarrow}F_1X\subset X,\\
\sigma_X:X=F_{n-1}X\overset{q}{\to} F_{n-1}X/F_{n-2}X\overset{g_{n-2}^{-1}}{\longrightarrow} \Sigma^{n-2}X_2.
\end{array}\right.
\end{equation}

We define $\langle\vec{\bm f}\,\rangle^\circledast$ for $\vec{\bm f}\in\mathrm{TOP}^\circledast$ to be 
the set of all $\alpha\in[\Sigma^{n-2}X_1,X_{n+1}]$ such that there is a finitely filtered space $X$ 
of type $(f_{n-1},\dots,f_2)$ in $\mathrm{TOP}^\circledast$ and a couple of maps $g,h$ which make 
the following diagram homotopy commutative and $\alpha$ is the homotopy class of $h\circ g$. 
\begin{equation}
\begin{split}
&\xymatrix{
 & \Sigma^{n-2}X_1 \ar[dl]_-{\Sigma^{n-2}f_1} \ar[d]^-g & \\
\Sigma^{n-2}X_2 & X \ar[l]^-{\sigma_X} \ar[d]_-h & X_n \ar[l]_-{j_X} \ar[dl]^-{f_n}\\
& X_{n+1} &
}
\end{split}
\end{equation}
Note that $\langle\vec{\bm f}\,\rangle^*$ is the bracket $\langle\vec{\bm f}\,\rangle$ defined by Cohen, 
$\langle\vec{\bm f}\,\rangle^w=\langle\vec{\bm f}\,\rangle^*$ if $\vec{\bm f}\in\mathrm{TOP}^w$, and 
$\langle\vec{\bm f}\,\rangle^{clw}\subset \langle\vec{\bm f}\,\rangle^w=\langle\vec{\bm f}\,\rangle^*$ 
if $\vec{\bm f}\in\mathrm{TOP}^{clw}$. 
If $\vec{\bm f}=(f_n,\dots,f_1)$ and $\vec{\bm f'}=(f_n',\dots,f_1')$ are in $\mathrm{TOP}^\circledast$ and 
satisfy $f_i\simeq f_i'$ for all $i$, 
and if $X$ is a finitely filtered space of type $(f_{n-1},\dots,f_2)$ in $\mathrm{TOP}^\circledast$, 
then $X$ is a finitely filtered space of type $(f_{n-1}',\dots,f_2')$ in $\mathrm{TOP}^\circledast$ so that 
$\langle\vec{\bm f}\, \rangle^\circledast=\langle\vec{\bm f'}\, \rangle^\circledast$. 

The following holds obviously from definitions. 

\begin{prop}
Given a map $f_0:X_0\to X_1$ in $\mathrm{TOP}^\circledast$, we have 
$$
\langle f_n,\dots,f_1\rangle^\circledast\circ \Sigma^{n-2}f_0\subset\langle f_n,\dots,f_2,f_1\circ f_0\rangle^\circledast.
$$ 
\end{prop}

\begin{lemma}%B.3
Given $\vec{\bm f}\in\mathrm{TOP}^\circledast$, 
if $X$ is a finitely filtered space of type $(f_{n-1},\dots,f_2)$ in $\mathrm{TOP}^\circledast$, then $\Sigma X$ is 
a finitely filtered space of type 
$(\Sigma f_{n-1},\dots,\Sigma f_2)$ in $\mathrm{TOP}^\circledast$ such that $j_{\Sigma X}=\Sigma j_X:\Sigma X_n\to \Sigma X$ and 
$\sigma_{\Sigma X}\simeq (-1)^n \Sigma\sigma_X: \Sigma X\to \Sigma^{n-1}X_2$. 
\end{lemma}
\begin{proof}
From the definitions, there is a filtration $F_0X=\{*\}\subset F_1X\subset\cdots\subset F_{n-1}X=X$ with 
$F_{k-1}X\subset F_kX$ a free or closed free cofibration for $k\ge 1$ and maps $j_X, \sigma_X$ of (B.3). 
Define $F_k\Sigma X=\Sigma F_kX$ for $0\le k\le n-1$. 
Set $g_k^*=\Sigma g_k\circ(1_{X_{n-k}}\wedge\tau(\s^1,\s^k))$ for $0\le k\le n-2$. 
Then 
$g_k^*$ is a homotopy equivalence and $g_k^*\simeq (-1)^k\Sigma g_k$ under the identification 
(2.1): $\Sigma^k\Sigma X_{n-k}=\Sigma\Sigma^kX_{n-k}$. 
By suspending (B.2), the diagram
$$
\xymatrix{
&\Sigma^k\Sigma X_{n+1-k} \ar[d]_-{1_{X_{n+1-k}}\wedge \tau(\s^1,\s^{k-1}\wedge\s^1)}  & 
\Sigma^k\Sigma X_{n-k} \ar[l]_-{\Sigma^k\Sigma f_{n-k}} \ar[d]_-{1_{X_{n-k}}\wedge\tau(\s^1,\s^k)} \ar@/^10mm/[dd]^-{g_k^*}\\
\Sigma\Sigma^{k-1}\Sigma X_{n+1-k} \ar[d]_-{\Sigma g_{k-1}^*} & 
\Sigma\Sigma^k X_{n+1-k} \ar[l]^-{1_{X_{n+1-k}}\wedge\tau(\s^{k-1},\s^1)\wedge 1_{\s^1}}
\ar[d]^-{\Sigma^2g_{k-1}} & 
\Sigma\Sigma^k X_{n-k} \ar[d]_-{\Sigma g_k} \ar[l]^-{\Sigma\Sigma^k f_{n-k}}\\
\Sigma^2(F_kX/F_{k-1}X) \ar@{=}[r] & \Sigma^2 (F_k X/F_{k-1}X)  & 
\Sigma(F_{k+1}X/F_k X) \ar[l]_-{\Sigma^2q_{k-1}\circ\Sigma\delta_k} 
}
$$
is homotopy commutative for $1\le k\le n-2$, 
where $\delta_k$ is a connecting map of 
the cofibre sequence $F_kX\overset{j_k}{\subset} F_{k+1}X\overset{q_k}{\to} F_{k+1}X/F_kX$ 
and $q_{k-1}:F_kX\to F_kX/F_{k-1}X$ is the quotient map. 
Set 
$$
\delta_k^*=-\Sigma\delta_k:F_{k+1}\Sigma X/F_k\Sigma X=\Sigma(F_{k+1}X/F_kX)\to\Sigma F_k\Sigma X=\Sigma^2 F_kX
$$
which is a connecting map of 
the cofibre sequence $F_k\Sigma X\overset{-\Sigma j_k}{\longrightarrow}F_{k+1}\Sigma X
\overset{-\Sigma q_k}{\longrightarrow} F_{k+1}\Sigma X/F_k\Sigma X$. 
We have
\begin{align*}
&\Sigma^2 q_{k-1}\circ\delta_k^*\circ g_k^*=\Sigma^2q_{k-1}\circ(-\Sigma\delta_k)\circ g_k^*\\
&\simeq(-1_{\Sigma(F_k\Sigma X/F_{k-1}\Sigma X)})\circ\Sigma^2q_{k-1}\circ\Sigma\delta_k\circ g_k^*\\
&\simeq (-1_{\Sigma(F_k\Sigma X/F_{k-1}\Sigma X)})\circ\Sigma^2g_{k-1}\circ
(1_{X_{n+1-k}}\wedge\tau(\s^1,\s^{k-1}\wedge\s^1))\circ \Sigma^k\Sigma f_{n-k}\\
&=(-1_{\Sigma(F_k\Sigma X/F_{k-1}\Sigma X)})\circ\Sigma g_{k-1}^*\circ(1_{X_{n+1-k}}\wedge\tau(\s^{k-1},\s^1)\wedge 1_{\s^1})\\
&\hspace{2cm} \circ(1_{X_{n+1-k}}\wedge\tau(\s^1,\s^{k-1}\wedge\s^1))\circ\Sigma^k\Sigma f_{n-k}\\
&\simeq \Sigma g_{k-1}^*\circ(-1_{\Sigma\Sigma^{k-1}\Sigma X})\circ (1_{X_{n+1-k}}\wedge\tau(\s^{k-1},\s^1)\wedge 1_{\s^1})\\
&\hspace{2cm} \circ (1_{X_{n+1-k}}\wedge\tau(\s^1,\s^{k-1}\wedge\s^1))\circ\Sigma^k\Sigma f_{n-k}\\
&\simeq \Sigma g_{k-1}^*\circ\Sigma^k\Sigma f_{n-k}.
\end{align*}
Hence $\Sigma X$ is a finitely filtered space of type $(\Sigma f_{n-1},\dots,\Sigma f_2)$ such that 
$j_{\Sigma X}=\Sigma j_X$ and 
$\sigma_{\Sigma X}=g_{n-2}^{*-1}\circ \Sigma q_{n-2}\simeq (-1)^{n-2}\Sigma g_{n-2}^{-1}\circ \Sigma q_{n-2}= (-1)^n \Sigma\sigma_X$. 
\end{proof}

\begin{thm}%B.4
$\Sigma \langle\vec{\bm f}\,\rangle^\circledast\subset(-1)^n\langle \Sigma\vec{\bm f}\,\rangle^\circledast$. 
\end{thm}
\begin{proof} 
Take $\alpha\in\langle\vec{\bm f}\,\rangle^\circledast$. 
Then there is a finitely filtered space $X$ of type $(f_{n-1},\dots,f_2)$ in $\mathrm{TOP}^\circledast$ 
and the homotopy commutative diagram (B.4) with $\alpha=h\circ g$. 
Suspending (B.4), we have the next homotopy commutative diagram, 
where $\tau=1_{X_2}\wedge\tau(\s^{n-2},\s^1)$.
$$
\xymatrix{
&& \Sigma^{n-2}\Sigma X_1 \ar[ddll]_-{\Sigma^{n-2}\Sigma f_1} \ar[d]^-{1_{X_1}\wedge\tau(\s^1,\s^{n-2})} & \\
& & \Sigma\Sigma^{n-2}X_1 \ar[dl]_-{\Sigma\Sigma^{n-2}f_1} \ar[d]^-{\Sigma g} & \\
\Sigma^{n-2}\Sigma X_2 &\Sigma\Sigma^{n-2}X_2 \ar[l]^-{\tau} & \Sigma X \ar[l]^-{\Sigma\sigma_X} \ar[d]_-{\Sigma h}
 & \Sigma X_n \ar[l]_-{\Sigma j_X} \ar[dl]^-{\Sigma f_n}\\
&& \Sigma X_{n+1} &
}
$$
Since $\Sigma X$ is a finitely filtered space of type $(\Sigma f_{n-1},\dots,\Sigma f_2)$ in $\mathrm{TOP}^\circledast$, 
we have $\tau\circ \Sigma\sigma_X\simeq \sigma_{\Sigma X}$ and $\Sigma j_X=j_{\Sigma X}$ by Lemma B.3. 
It follows from the homotopy commutativity of the above diagram that 
\begin{gather*}
\sigma_{\Sigma X}\circ \Sigma g\circ(1_{X_1}\wedge\tau(\s^1,\s^{n-2}))\simeq \Sigma^{n-2}\Sigma f_1,\\
(-1)^n \Sigma\alpha=\Sigma h\circ \Sigma g\circ (1_{X_1}\wedge\tau(\s^1,\s^{n-2}))\in\langle \Sigma\vec{\bm f}\,\rangle^\circledast,
\end{gather*}
and so $\Sigma\alpha\in (-1)^n\langle \Sigma\vec{\bm f}\,\rangle^\circledast$. 
This completes the proof. 
\end{proof}

The following lemma was used in the proof of Lemma 6.4.2. 
It can be proved easily, so we omit details. 

\begin{lemma}
Let $j:A\subset X$ be a pointed inclusion map which is a free cofibration, and $f:X\to Y$ a pointed map. 
Then there are natural homeomorphisms 
$$(Y\cup_f CX)/(Y\cup_{f\circ j}CA)\approx \Sigma X/\Sigma A\approx \Sigma(X/A).$$ 
\end{lemma}

\section*{Acknowledgement}
We express warm thanks to M. Yasui for guiding us to articles \cite{L,W2}, 
and to the referee for useful suggestions. 

\end{appendix}

\end{document}